\documentclass[cmp,draft,numbook,envcountsect,envcountsame,envcountreset]{svjour}

\usepackage{graphics}

\usepackage{amsmath}
\usepackage{amsfonts}
\usepackage{amssymb}

\usepackage[all]{xy}
\usepackage{enumerate}
\usepackage{color}
\usepackage{verbatim}

\newcommand\bigcheck[1]{#1 \raise1ex\hbox{$\hspace{-1ex}{}^\vee$}}
\newcommand\sucheck[1]{#1 \raise0.5ex\hbox{$\hspace{-1ex}{}^\vee$}}

\newcommand{\mc}[1]{{\mathcal #1}}
\newcommand{\mf}[1]{{\mathfrak #1}}
\newcommand{\mb}[1]{{\mathbb #1}}

\newcommand\tint{{\textstyle\int}}

\renewcommand{\tilde}{\widetilde}

\newcommand{\End}{\mathop{\rm End }}
\newcommand{\Hom}{\mathop{\rm Hom }}
\newcommand{\ad}{\mathop{\rm ad }}
\newcommand{\sign}{\mathop{\rm sign }}

\newcommand{\Tor}{\mathop{\rm Tor }}

\newcommand{\Lie}{\mathop{\rm Lie }}

\renewcommand{\ker}{\mathop{\rm Ker }}
\newcommand{\im}{\mathop{\rm Im }}

\newcommand{\Vect}{\mathop{\rm Vect }}

\definecolor{light}{gray}{.9}

\begin{document}


\title{{Calculus structure on the Lie conformal algebra complex and the variational complex}}
\titlerunning{Calculus structure}

\author{Alberto De Sole\inst{1} \and Pedram Hekmati\inst{2} \and Victor G. Kac\inst{3}}
\institute{
Dipartimento di Matematica, Universit\`a di Roma ``La Sapienza'',
00185 Roma, Italy \\
\email{desole@mat.uniroma1.it}
\and
School of Mathematical Sciences,
University of Adelaide, Adelaide, SA 5005, Australia \\
\email{pedram.hekmati@adelaide.edu.au}
\and
Department of Mathematics, MIT,
77 Massachusetts Avenue, Cambridge, MA 02139, USA \\
\email{kac@math.mit.edu}
}

\authorrunning{A. De Sole, P. Hekmati, V. Kac}


\maketitle


\begin{abstract}
\noindent 
We construct a calculus structure on the Lie conformal algebra cochain complex.
By restricting to degree one chains, we recover the structure of a $\mf{g}$-complex introduced 
in \cite{DSK}. A special case of this construction is the variational calculus, for which we 
provide explicit formulas.
\end{abstract}


\section{Introduction}
\label{sec:intro}

A Lie conformal algebra over a field $\mb{F}$, is an $\mb{F}[\partial]$-module $R$ endowed with 
a bilinear map $[\cdot\,_\lambda\,\cdot]$ with values in $R[\lambda]$,
called the \emph{$\lambda$-bracket}, satisfying certain sesquilinearity, skewcommutativity 
and Jacobi identity. 
In practice, $\lambda$-brackets arise as generating functions for the singular part 
of the operator product expansion in conformal field theory \cite{K}. 
More recently, their domain of applicability has been further extended to encode 
local Poisson brackets in the theory of integrable evolution equations \cite{BDSK}.

Lie conformal algebras resemble Lie algebras in many ways and in particular their cohomology theory
with coefficients in an $R$-module $M$ was developed in \cite{BKV}, \cite{BDAK}.

In \cite{DSK}, it was further shown that when the $R$-module $M$ is endowed 
with a commutative associative product, on which $\partial$ and $R$ act as derivations, 
the Lie conformal algebra cochain complex $(C^\bullet(R,M),d)$ carries a structure of a $\mf{g}$-complex,
where $\mf g$ is the Lie algebra of Lie conformal algebra 1-chains. 
Namely, for each $X\in\mf g$ there exists a contraction operator $\iota_X$ 
and a Lie derivative $L_X$ on $C^\bullet(R,M)$ satisfying the usual rules of Cartan calculus. 

Moreover, it was shown in \cite{DSK} that in the special case of the Lie conformal algebra 
$R=\bigoplus_{i\in I}\mb{F}[\partial]u_i$ with zero $\lambda$-bracket,
acting on an algebra of differential functions $\mc V$ by 
\begin{equation}\label{0329:eq1}
u_{i\lambda}f = \sum_{k\in \mb{Z}_+} \lambda^k \frac{\partial f}{\partial u_i^{(k)}}
\,\,,\,\,\,\,
f\in \mathcal V \,,
\end{equation}
the cochain complex $(C^\bullet(R,\mc V),d)$
is identified with the variational complex, introduced in \cite{GD},
the Lie algebra of 1-chains for the $R$-module $\mc V$
is identified with the Lie algebra of evolutionary vector fields,
and the Cartan calculus turns into the variational calculus.

Our aim in this paper is to extend the structure of a $\mf{g}$-complex on $C^\bullet(R,M)$
to the much richer structure of a calculus structure. 
The notion of a calculus structure originated in Hochschild cohomology theory \cite{DTT}
(in fact, the definition in \cite{DTT} differs from ours by some signs).
It is defined as a representation $(\iota_\cdot,L_\cdot)$ of a Gerstenhaber 
(=odd Poisson) algebra $\mc G$ on a complex $(\Omega,d)$,
such that the usual Cartan's formula holds
\begin{equation}\label{compatibility3}
L_X = [\iota_X,d]\,.
\end{equation}
Here $\iota_\cdot$ (respectively $L_\cdot$) is a representation of $\mc G$
(resp. of $\mc G$ with reversed parity) viewed as an associative (resp. Lie) superalgebra.

The motivating example of a calculus structure comes from differential geometry. 
Namely, let $\mc M$ be a smooth manifold. The space of polyvector fields $\Omega_\bullet(\mathcal M)$, 
is a Gerstenhaber algebra, with the associative product given by the exterior multiplication $\wedge$,
and the bracket given by the Schouten bracket.
Then the representation of $\Omega_\bullet(\mc M)$ on the space $\Omega^\bullet(\mathcal M)$ 
of differential forms, together with the de Rham differential $d$,
is given by the contraction operator
\begin{equation}
(\iota_X\omega)(Y) = (-1)^{\frac{p(X)(p(X)-1)}{2}}\omega(X\wedge Y)
\,\,,\,\,\,\,
X,Y\in \Omega_\bullet(\mathcal M),\,\omega\in \Omega^\bullet(\mc M)
\end{equation}
and the Lie derivative $L_X$
is given by Cartan's formula \eqref{compatibility3}.

In Section \ref{sec:2}, apart from the basic definitions, we introduce the notion
of a rigged representation of a Lie algebroid $(\mf g,A)$,
which allows one to extend a structure of a $(\mf g,A)$-complex
to a calculus structure (Theorem \ref{0217:th} and \ref{0228:th}).

In Section \ref{sec:3}, for any Lie algebra $\mf g$ and a $\mf g$-module $A$,
where $A$ is a commutative associative algebra on which $\mf g$ acts by derivations,
we construct a calculus structure $(\Delta_\bullet(\mf g,A),\Delta^\bullet(\mf g,A))$,
where $\Delta_\bullet(\mf g,A)$ is the space of Lie algebra chains
endowed with a structure of a Gerstenhaber algebra,
and $(\Delta^\bullet(\mf g,A),d)$ is the complex of Lie algebra cochain
(Theorem \ref{0302:th}).
Keeping in mind the annihilation Lie algebra of a Lie conformal algebra,
we construct a ``topological'' calculus structure in the case when $\mf g$
is a linearly compact Lie algebra.

In Section \ref{sec:4} we introduce a Gerstenhaber algebra structure on the space
of Lie conformal algebra chains $C_\bullet(R,M)$
for an arbitrary module $M$ with a commutative associative algebra structure
over a Lie conformal algebra $R$, acting on $M$ by derivations.
This extends the Lie algebra structure on the space of 1-chains with reversed parity,
$\mf g=\Pi C_1(R,M)$, defined in \cite[Theorem 4.8]{DSK}.
This allows us to extend the $\mf g$-structure on the complex of Lie conformal algebra cochains
$C^\bullet(R,M)$ constructed in \cite{DSK} to a calculus structure (Theorem \ref{0319:th}).
Furthermore, we construct a morphism from the topological calculus structure for 
the (linearly compact) annihilation Lie algebra $\Lie_-R$ of a finite Lie conformal algebra $R$
to the calculus structure $(C_\bullet(R,M),C^\bullet(R,M))$
which induces an isomorphism of the reduced by $\partial$ former calculus structure
to the torsionless part of the latter calculus structure (Theorem \ref{0322:th}),
extending that in \cite{DSK} for $\mf g$-structures.
This is used in Section \ref{sec:5} to identify the variational complex $\Omega^\bullet(\mc V)$
over an algebra of differential functions $\mc V$ on $\ell$ differential variables,
with the complex $C^\bullet(R,\mc V)$,
where $R$ is the free $\mb F[\partial]$-module of rank $\ell$ with zero $\lambda$-bracket,
acting on $\mc V$ via \eqref{0329:eq1},
and to extend the identification of $\mf g$-structures obtained in \cite{DSK},
to an explicit construction of the variational calculus structure
$(\Omega_\bullet(\mc V),\Omega^\bullet(\mc V))$,
where $\Omega_\bullet(\mc V)$ is the Gerstenhaber algebra of all evolutionary polyvector fields over $\mc V$.

Throughout the paper
all vector spaces are considered over a field $\mb F$ of characteristics zero.
Unless otherwise specified, direct sums and tensor products are considered over $\mb F$.


\section{Calculus structure on a complex}
\label{sec:2}

In this section we introduce the basic definitions of a Gerstenhaber algebra and of a calculus structure,
and prove some simple related results that will be used throughout the paper. 

\subsection{Rigged representations of Lie superalgebras}

Recall that a vector superspace is a $\mb Z/2\mb Z$-graded vector space $V=V_{\bar 0}\oplus V_{\bar 1}$.
If $a\in V_\alpha$, where $\alpha\in\mb Z/2\mb Z=\{\bar0,\bar1\}$, one says that $a$ has parity $p(a)=\alpha$.
One denotes by $\Pi V$ the superspace obtained from $V$ by reversing the parity, namely $\Pi V=V$ as a vector space,
with parity $\bar p(a)=p(a)+\bar1$.
An endomorphism of $V$ is called even (resp. odd) if it preserves (resp. reverses) the parity.
The superspace $\End(V)$ of all endomorphisms of $V$ is endowed with a Lie superalgebra structure by the formula:
$[A,B]=A\circ B-(-1)^{p(A)p(B)}B\circ A$.

\begin{definition}\label{def:rigged}
A representation of a Lie superalgebra $\mf g$ on a vector superspace $V$, $X\mapsto L_X\in\End(V)$, 
is called \emph{rigged} if it is endowed with an even linear map
$\iota_{\cdot}:\,\Pi\mf g\to\End(V)$ (i.e. a parity reversing map $\mf g\to\End(V)$), denoted $X\mapsto\iota_X$, such that:
\begin{enumerate}[(i)]
\item $[\iota_X,\iota_Y]=0$ for all $X,Y\in\mf g$,
\item $[L_X,\iota_Y]=\iota_{{[X,Y]}}$ for all $X,Y\in\mf g$.
\end{enumerate}
\end{definition}

Throughout the paper we will denote the parity of the Lie superalgebra $\mf g$ by $\bar p$. Hence, for the linear map 
$\iota_\cdot:\,\Pi\mf g\to\End(V)$, we have $p(\iota_X)=\bar p(X)+\bar1$.

Recall that a \emph{complex} $(\Omega,d)$ is a vector superspace $\Omega$, endowed with an odd endomorphism $d\in\End(\Omega)$ 
such that $d^2=0$.
A representation of a Lie superalgebra $\mf g$ on a complex $(\Omega,d)$ is a representation of $\mf g$ on the superspace $\Omega$,
denoted $X\mapsto L_X\in\End(\Omega)$, such that $[L_X,d]=0$.

Recall also (see e.g. \cite{DSK}) that a $\mf g$-\emph{complex} is a pair $(\mf g,\Omega)$, where $\mf g$ is a Lie superalgebra,
$(\Omega,d)$ is a complex, endowed with a linear map $\iota_\cdot:\,\Pi\mf g\to\End(\Omega)$,
satisfying the following conditions:
\begin{enumerate}[(i)]
\item[$(i)$] $[\iota_X,\iota_Y]=0$ for all $X,Y\in\mf g$,
\item[$(ii)$] $[[\iota_X,d],\iota_Y]=\iota_{[X,Y]}$ for all $X\in\mf g$.
\end{enumerate}
This is also called a $\mf g$-\emph{structure} on the complex $(\Omega,d)$.
\begin{lemma}\label{0220:lem1}
Any $\mf g$-complex $(\mf g,\Omega)$ gives rise to a rigged representation of the Lie superalgebra $\mf g$
on the complex $(\Omega,d)$, obtained by defining the map $L_\cdot:\,\mf g\to\End(\Omega)$ by \emph{Cartan's formula}:
\begin{equation} \label{comp3} 
L_X 
= [\iota_X,d]\,.
\end{equation}
\end{lemma}
\begin{proof}
Indeed, conditions $(i)$ and $(ii)$ in Definition \ref{def:rigged} coincide, via Cartan's formula, with conditions $(i)$ and $(ii)$ above.
Moreover, since $d^2=\frac12[d,d]=0$, we immediately get that $[L_X,d]=0$ for every $X\in\mf g$.
Finally we have, by Cartan's formula \eqref{comp3} and condition $(ii)$,
$$
[L_X,L_Y]
= [L_X,[\iota_Y,d]]
= [[L_X,\iota_Y],d] \pm [\iota_Y,[L_X,d]] 
= [\iota_{[X,Y]},d]
= L_{[X,Y]}\,.
$$
\flushright\qed
\end{proof}

\subsection{Rigged representations of Lie algebroids}

Recall that a \emph{Lie superalgebroid} is a pair $(\mf g,A)$,
where $\mf g$ is a Lie superalgebra,
$A$ is a commutative associative algebra,
such that $\mf g$ is a left $A$-module and $A$ is a left $\mf g$-module,
satisfying the following compatibility conditions ($X,Y\in \mf g,\,f,g\in A$):
\begin{enumerate}[(i)]
\item[$(i)$]  $(fX)(g)=f(X(g))$,
\item[$(ii)$]  $X(fg)=X(f)g+fX(g)$,
\item[$(iii)$]  $[X,fY]=X(f)Y+f[X,Y]$.
\end{enumerate}
\begin{remark}
Since we assume $A$ to be purely even, the odd part of $\mf g$ necessarily acts trivially on $A$.
One can consider also $A$ to be a commutative associative superalgebra,
but then the signs in the formulas become more complicated.
\end{remark}

\begin{example}\label{0220:ex1}
If $A$ is a commutative associative algebra and $\mf g$ is a subalgebra of the Lie algebra of derivations of $A$
such that $A\mf g\subset\mf g$, then, obviously, $(\mf g,A)$ is a Lie algebroid.
\end{example}

\begin{example}\label{0219:ex1}
If $\mf g$ is a Lie superalgebra with parity $\bar p$, acting by derivations 
on a commutative associative superalgebra $A$,
then $(A\otimes\mf g,A)$ is a Lie algebroid with Lie bracket
$$
[f\otimes X,g\otimes Y]=fg\otimes[X,Y]+fX(g)\otimes Y-(-1)^{\bar p(X)\bar p(Y)}gY(f)\otimes X\,.
$$
\end{example}

\begin{example}\label{0219:ex2}
Given a Lie superalgebroid $(\mf g,A)$, we can construct two Lie superalgebras:
$\mf g\ltimes A$ and $\mf g\ltimes\Pi A$, with Lie bracket
which extends that on $\mf g$ by letting, for $f,g\in A$ and $X\in\mf g$,
$[f,g]=0$, $[X,f]=X(f)$ and $[f,X]$ given by skewcommutativity.
Both these Lie superalgebras give rise to Lie superalgebroids in the obvious way.
\end{example}

\begin{definition}\label{0220:def2}
\begin{enumerate}[(a)]
\item 
A \emph{representation} of a Lie superalgebroid $(\mf g,A)$ on a vector superspace $V$
is a left $A$-module structure on $V$, denoted $f\mapsto\iota_f$, 
together with a left $\mf g$-module structure on $V$, denoted $X\mapsto L_X$,
such that, for $X\in\mf g,\,f\in A$, we have $[L_X,\iota_f]=\iota_{X(f)}$.
\item
Given $\epsilon\in\mb F$, 
an $\epsilon$-\emph{rigged representation} of a Lie superalgebroid $(\mf g,A)$ 
(with parity of $\mf g$ denoted by $\bar p$) on a vector superspace $V$ is 
a left $A$-module structure on $V$, $\iota_\cdot:\,A\to\End(V)$,
together with a rigged representation of the Lie superalgebra $\mf g\ltimes\Pi A$,
defined by the linear maps
$\iota_\cdot:\,\Pi\mf g\oplus A\to\End(V),\,L_\cdot:\,\mf g\oplus\Pi A\to\End(V)$,
satisfying the following compatibility conditions ($f,g\in A,\,X\in\mf g$):
\begin{enumerate}[(i)]
\item $L_{fg}=L_f\iota_g+\iota_f L_g$,
\item $\iota_{fX}=\iota_f\iota_X$,
\item $L_{fX}=\iota_fL_X-(-1)^{\bar p(X)}L_f\iota_X
-\epsilon\iota_{X(f)}$.
\end{enumerate}
\end{enumerate}
\end{definition}

\begin{remark}
If $(\mf g,A)$ is a Lie superalgebroid, then both $A$ and $\mf g$ are $(\mf g,A)$-modules (but in general they are not rigged).
On the other hand, as we will see in Proposition \ref{schouten}, they extend to a $1$-rigged representation
of the Lie superalgebroid $(\mf g,A)$ on the vector superspace $S_A(\Pi\mf g)$.
However, for the applications to calculus structure, the most important role will be played by the $0$-rigged representations.
Indeed, as we will see in Proposition \ref{0220:prop1} below, any $(\mf g,A)$-complex gives rise to a $0$-rigged representation
of the Lie superalgebroid $(\mf g,A)$.
\end{remark}

\begin{definition}\label{0220:def1}
\begin{enumerate}[(a)]
\item 
A \emph{representation} of a Lie superalgebroid $(\mf g,A)$ \emph{on a complex} $(\Omega,d)$
is a representation of $(\mf g,A)$ on the vector superspace $\Omega$ such that $[L_X,d]=0$ for every $X\in\mf g$.
\item
A $(\mf g,A)$-\emph{complex} $(\Omega,d)$, where $(\mf g,A)$ is a Lie superalgebroid,
is a $\mf g\ltimes\Pi A$-complex (for the Lie superalgebra $\mf g\ltimes\Pi A$ in Example \ref{0219:ex2})
such that the linear map $\iota_\cdot:\,\Pi\mf g\oplus A\to\End(\Omega)$ satisfies the following two additional 
conditions (for $f,g\in A,\,X\in\mf g$):
\begin{enumerate}[(i)]
\item $\iota_{fg}=\iota_f\iota_g$,
\item $\iota_{fX}=\iota_f\iota_X$.
\end{enumerate}
\end{enumerate}
\end{definition}

The following result allows us to extend a $\mf g$-complex to an $(A\otimes\mf g,A)$-complex.
\begin{lemma}\label{0302:lem}
Let $A$ be a commutative associative algebra and let $\mf g$ be a Lie superalgebra,
with parity $\bar p$, acting on $A$ by derivations, 
so that we have the corresponding Lie superalgebroid $(A\otimes\mf g,A)$
from Example \ref{0219:ex1}.
Let $(\Omega,d)$ be a complex 
endowed with a structure of a $\mf g$-complex, $\iota_\cdot:\,\Pi\mf g\to\End(\Omega)$,
and with a structure of a left $A$-module, denoted by $\iota_\cdot:\,A\to\End(\Omega)$.
Define the map $L_\cdot:\,\mf g\oplus\Pi A\to\End(\Omega)$
by Cartan's formula: $L_a=[\iota_a,d],\,a\in\mf g\oplus\Pi A$.
Assume that the following conditions hold:
\begin{enumerate}[(i)]
\item $[\iota_X,\iota_f]=0$ for all $f\in A,\,X\in\mf g$,
\item $[L_f,\iota_g]=0$, for all $f,g\in A$,
\item $[L_X,\iota_f]=\iota_{X(f)}$, for all $f\in A,\,X\in\mf g$.
\end{enumerate}
Then, we have a structure of an $(A\otimes\mf g,A)$-complex on $(\Omega,d)$
by letting $\iota_{f\otimes X}=\iota_f\iota_X$, for $f\in A$ and $X\in\mf g$.
\end{lemma}
\begin{proof}
By definition of a complex over the Lie superalgebroid $(A\otimes\mf g,A)$, 
we need to prove that the following relations hold:
\begin{enumerate}[(1)]
\item $[\iota_a,\iota_b]=0$, for $a,b\in(A\otimes\Pi\mf g)\oplus A$,
\item $[L_a,\iota_b]=\iota_{[a,b]}$, for $a,b\in(A\otimes\mf g)\ltimes\Pi A$,
\item $\iota_{fg}=\iota_f\iota_g$, for $f,g\in A$,
\item $\iota_{f(g\otimes X)}=\iota_f\iota_{g\otimes X}$, for $f,g\in A,\,X\in\Pi\mf g$,
\end{enumerate}
where $L_a$, as before, is defined by Cartan's formula for $a\in (A\otimes\mf g)\oplus\Pi A$.
Relation (1) is immediate by the definition of $\iota_{f\otimes X}$ and assumption $(i)$.
Relation (3) holds by the assumption that $\iota:\,A\to\End(\Omega)$ defines a structure 
of a left $A$-module.
Relation (4) is also immediate.
We are left to prove relation (2).
When $a,b\in\Pi A$, it holds by assumption $(ii)$.
When $a=f\in\Pi A,\,b=g\otimes X\in A\otimes\mf g$, it follows 
by a straightforward computation using the following identity,
$$
[L_f,\iota_X]=-(-1)^{\bar p(X)}\iota_{X(f)}\,,
$$
which can be easily checked.
Finally, when $a=f\otimes X\in A\otimes\mf g$ and $b=g\in\Pi A$ or $b=g\otimes Y\in A\otimes\mf g$, 
relation (2) follows using the identity,
$$
L_{f\otimes X}=\iota_fL_X-(-1)^{\bar p(X)}L_f\iota_X\,,
$$
which is again straightforward to check.
\flushright\qed
\end{proof}

The following result generalizes Lemma \ref{0220:lem1} to the case of Lie superalgebroids.
\begin{proposition}\label{0220:prop1}
Any $(\mf g,A)$-complex $(\Omega,d)$ gives rise to a $0$-rigged representation 
of the Lie superalgebroid $(\mf g,A)$ 
on the vector superspace $\Omega$, 
obtained by defining the map $L_\cdot:\,\mf g\ltimes\Pi A\to\End(\Omega)$ by Cartan's formula: 
$L_a = [\iota_a,d]$ for all $a\in\mf g\ltimes\Pi A$.
\end{proposition}
\begin{proof}
Condition $(i)$ in Definition \ref{0220:def1}(b) guarantees that $\Omega$ is a left $A$-module.
By Lemma \ref{0220:lem1} we know that the map $L_\cdot:\,\mf g\ltimes\Pi A\to\End(V)$, given by Cartan's formula, 
is a Lie superalgebra homomorphism.
Moreover, condition $(ii)$ in Definition \ref{0220:def2}(b) coincides with condition $(ii)$ in Definition \ref{0220:def1}(b).
Hence, to conclude the proof, we are left to check that $\iota_\cdot$ and $L_\cdot$ satisfy the compatibility conditions
$(i)$ and $(iii)$ in Definition \ref{0220:def2}(b). Both of them follow immediately by Cartan's formula.
\flushright\qed
\end{proof}

\begin{example}\label{0220:ex2}
Let $A$ be the algebra of smooth functions on a smooth manifold $\mathcal M$, $\mf g$ be the Lie algebra of smooth vector fields on $\mathcal M$, 
and $\Omega$ be the complex of smooth differential forms on $\mathcal M$ with the de Rham differential $d$.
Then $(\mf g,A)$ is a Lie algebroid. 
Moreover, the map $\iota_\cdot:\,\Pi\mf g\oplus A\to\End(V)$, 
where $\iota_f$ is the multiplication by $f\in A$, 
and $\iota_X$ is the contraction operator by the vector field $X\in\mf g$,
defines a structure of a $(\mf g,A)$-complex on $\Omega$.
Hence, by Proposition \ref{0220:prop1}, we get a $0$-rigged representation of the Lie algebroid $(\mf g,A)$ on the complex $\Omega$,
where $L_f$ is the multiplication by $-df$ in the algebra $\Omega$, for $f\in A$,
and $L_X$ is the Lie derivative by the vector field $X\in\mf g$.
\end{example}

\subsection{Gerstenhaber (= odd Poisson) algebras}

Recall that, given a commutative associative algebra $A$, and an $A$-module structure on a vector superspace $V$,
the symmetric, (respectively exterior) superalgebra $S_A(V)$ (resp. $\bigwedge_A(V)$)
is defined as the quotient of the tensor superalgebra $\mc T_A(V)$ by the relations
$u\otimes_A v-(-1)^{p(u)p(v)}v\otimes_A u$ (resp. $u\otimes_A v+(-1)^{p(u)p(v)}v\otimes_A u$).
Note that $S_A(\Pi V)$ is the same as $\bigwedge_A V$ as an $A$-module (but not as a vector superspace).

\begin{definition}\label{def:gerst}
A \emph{Gerstenhaber algebra} (also known as an \emph{odd Poisson algebra}) 
is a vector superspace $\mc{G}$, with parity $p$, endowed with a product $\wedge:\,\mc{G}\otimes \mc{G}\to \mc{G}$, 
and a bracket $[\cdot,\cdot]:\,\mc{G}\otimes \mc{G}\to \mc{G}$ satisfying the following properties:
\begin{enumerate}[(i)]
\item 
$(\mc{G},\wedge)$ is a commutative associative superalgebra,
\item 
$(\Pi\mc{G},[\cdot,\cdot])$ is a Lie superalgebra,
\item
the following left Leibniz rule holds:
\begin{equation}\label{Leibniz} 
[X, Y \wedge Z] = [X,Y]\wedge Z+(-1)^{(p(X)+\bar1)p(Y)}Y\wedge[X,Z]\,.
\end{equation}
\end{enumerate}
\end{definition}
From the left Leibniz rule \eqref{Leibniz} and skewcommutativity, we get the \emph{right Leibniz rule}:
\begin{equation}\label{r-Leibniz} 
[X\wedge Y,Z] = X\wedge[Y,Z]+(-1)^{p(Y)(p(Z)+\bar1)}[X,Z]\wedge Y\,.
\end{equation}

\begin{proposition}\label{schouten}
Let $(\mf g,A)$ be a Lie superalgebroid.
Then there exists a unique structure of a Gerstenhaber algebra on the superspace
$\mc{G}=S_A(\Pi\mf g)$, with parity denoted by $p$,
where the commutative associative superalgebra product $\wedge$ on $\mc{G}$ is the product 
in the symmetric superalgebra $S_A(\Pi\mf g)$,
and the Lie superalgebra bracket $[\cdot,\cdot]$ on $\Pi\mc{G}$, called the \emph{Schouten bracket},
extends inductively that on the Lie superalgebra $\mf g\ltimes\Pi A$ from Example \ref{0219:ex2}
by the Leibniz rule \eqref{Leibniz}.
\end{proposition}
\begin{proof} 
The symmetric superalgebra $S_A(\Pi\mf g)$ is defined as the quotient 
of the tensor superalgebra $\mc T(\Pi\mf g\oplus A)$ by the two-sided ideal $\mc K$ 
generated by the relations
\begin{equation}\label{0222:eq2}
\begin{array}{rll}
(i) & a\otimes b = (-1)^{p(a)p(b)}b\otimes a \,,\,\, & a,b\in\Pi\mf g\oplus A\,, \\
(ii) & f\otimes X = fX\,, & f\in A,\,X\in\Pi\mf g\,.
\end{array}
\end{equation}
Therefore, in order to prove that the Schouten bracket is well defined, we need to do three things.
First, we check that its inductive definition preserves associativity of the tensor product, 
so that we have a well-defined bracket on the whole tensor algebra,
$[\cdot,\cdot]\,\tilde{}:\,\Pi\mc T(\Pi\mf g\oplus A)\times\Pi\mc T(\Pi\mf g\oplus A)
\to \Pi S_A(\Pi\mf g)$.
Second, we argue that, in order to prove that $\mc K$ is in the kernel of this bracket, it suffices to show
that it preserves relations \eqref{0222:eq2}$(i)$ and $(ii)$.
Finally, we prove that these relations are indeed preserved.

We start by defining a bracket 
$[\cdot,\cdot]\,\tilde{}:\,\Pi\mc T(\Pi\mf g\oplus A)\times\Pi\mc T(\Pi\mf g\oplus A)
\to \Pi S_A(\Pi\mf g)$,
such that its restriction to $\mf g\oplus\Pi A$ coincides with the given 
Lie bracket on $\mf g\ltimes\Pi A$.
We do it, inductively, in three steps.
First we extend it to a bracket 
$[\cdot,\cdot]\,\tilde{}:\,(\mf g\oplus\Pi A)\times\Pi\mc T(\Pi\mf g\oplus A)
\to\Pi\mc T(\Pi\mf g\oplus A)$, by the left Leibniz rule \eqref{Leibniz}
with $\wedge$ replaced by $\otimes$ and $[\cdot,\cdot]$ replaced by $[\cdot,\cdot]\,\tilde{}$.
To prove that this map is well defined we check that the left Leibniz rule preserves
the associativity relation in the tensor algebra.
Indeed, both $[X,Y\otimes(Z\otimes W)]\,\tilde{}$
and $[X,(Y\otimes Z)\otimes W]\,\tilde{}$ are equal to
\begin{equation}\label{0225:eq1}
\begin{array}{c}
[X,Y]\,\tilde{}\otimes Z\otimes W
+(-1)^{(p(X)+1)p(Y)}Y\otimes[X,Z]\,\tilde{}\otimes W \\
+(-1)^{(p(X)+1)(p(Y)+p(Z))}Y\otimes Z\otimes[X,W]\,\tilde{}\,.
\end{array}
\end{equation}
We then further extend it to a bracket
$[\cdot,\cdot]\,\tilde{}:\,\Pi\mc T(\Pi\mf g\oplus A)\times\Pi\mc T(\Pi\mf g\oplus A)
\to\Pi\mc T(\Pi\mf g\oplus A)$, by the right Leibniz rule \eqref{r-Leibniz},
with the same changes in notation.
Again, we prove that this map is well defined by checking that the right Leibniz rule preserves
associativity. Indeed, both $[X\otimes (Y\otimes Z),W)]\,\tilde{}$
and $[(X\otimes Y)\otimes Z,W]\,\tilde{}$ are equal to
\begin{equation}\label{0225:eq2}
\begin{array}{c}
X\otimes Y\otimes [Z,W]\,\tilde{}
+(-1)^{p(Z)(p(W)+1)}X\otimes [Y,W]\,\tilde{}\otimes Z \\
+(-1)^{(p(Y)+p(Z))(p(W)+1)}[X,W]\,\tilde{}\otimes Y\otimes Z\,.
\end{array}
\end{equation}
Finally, we compose the bracket $[\cdot,\cdot]\,\tilde{}$
with the canonical quotient map $\Pi\mc T(\Pi\mf g\oplus A)\to \Pi S_A(\Pi\mf g)$,
and we keep the same notation for the resulting map:
$[\cdot,\cdot]\,\tilde{}:\,\Pi\mc T(\Pi\mf g\oplus A)\times\Pi\mc T(\Pi\mf g\oplus A)
\to \Pi S_A(\Pi\mf g)$.
We claim that this map satisfies both the left and the right Leibniz rules.
The right Leibniz rule holds by construction, while for the left one we have to check
that, when computing $[X\otimes Y,Z\otimes W]$,
we get the same result if we first apply the left Leibniz rule and then the right one,
or vice versa.
As the reader can easily check, the results are not equal in the tensor algebra,
but they become equal after we pass to the symmetric algebra.

Next, it is immediate to check that the bracket $[\cdot,\cdot]\,\tilde{}$ preserves
the relations \eqref{0222:eq2}$(i)$ and $(ii)$,
namely, the differences between the LHS and RHS in both relations
lie in the center of this bracket.
This allows us to conclude, recalling \eqref{0225:eq1} and \eqref{0225:eq2},
that two-sided ideal $\mc K\subset\mc T(\Pi\mf g\oplus A)$ 
generated by the relations \eqref{0222:eq2}$(i) $and $(ii)$
is in the center of the bracket $[\cdot,\cdot]\,\tilde{}$.
Hence it factors through a well defined bracket 
$[\cdot,\cdot]:\,S_A(\Pi\mf g)\times S_A(\Pi\mf g)\to S_A(\Pi\mf g)$
satisfying both the left and the right Leibniz rules \eqref{Leibniz} and \eqref{r-Leibniz}.
Using this it is easy to check, by induction, that the bracket is skewcommutative,
and after that, again by induction, that it satisfies the Jacobi identity.
\flushright\qed
\end{proof}

\begin{remark}
If $\mf g$ is a Lie superalgebra with parity $\bar p$, the corresponding parity $p$ in the
Gerstenhaber algebra $\mc G=S_A(\Pi\mf g)$ is
\begin{equation}\label{0225:eq3}
p(X_1\wedge\cdots\wedge X_m)=\bar p(X_1)+\cdots+\bar p(X_m)+m\,,
\end{equation}
and the parity $\bar p$ of the Lie superalgebra $\Pi\mc G$ is
\begin{equation}\label{0225:eq4}
\bar p(X_1\wedge\cdots\wedge X_m)=\bar p(X_1)+\cdots+\bar p(X_m)+m+1\,.
\end{equation}
One derives from the left and right Leibniz rules \eqref{Leibniz} and \eqref{r-Leibniz}
explicit formulas for the Schouten bracket between two arbitrary elements
of the Gerstenhaber algebra $\mc G$.
For $f\in A$ and $X=X_1\wedge\cdots\wedge X_m\in\mc G$, with $X_i\in\mf g$, we have
\begin{equation}\label{0225:eq5}
\begin{array}{l}
\vphantom{\Bigg(}
\displaystyle{
[f,X_1\wedge\cdots\wedge X_m]=(-1)^{\bar p(X_1)+\cdots+\bar p(X_m)+m}[X_1\wedge\cdots\wedge X_m,f]
}\\
\qquad\qquad\qquad\qquad
\displaystyle{
=\sum_{i=1}^m(-1)^{\bar p(X_1)+\cdots+\bar p(X_{i-1})+i} X_i(f) \, X_1\wedge \stackrel{i}{\check{\cdots}} \wedge X_m\,,
}\\
\end{array}
\end{equation}
while for $X=X_1\wedge\cdots\wedge X_m,\, Y=Y_1\wedge\cdots\wedge Y_n\in\mc G$, 
with $X_i,Y_j\in\mf g$, we have
\begin{equation}\label{0225:eq6}
\begin{array}{l}
\displaystyle{
{\big[X_1\wedge \dots \wedge X_m , Y_1\wedge \dots \wedge Y_n\big]}
}\\
\qquad\displaystyle{
=\sum_{i=1}^{m}\sum_{j=1}^n(-1)^{s_{ij}(X,Y)}
[X_i,Y_j]\wedge X_1\wedge \stackrel{i}{\check{\cdots}} \wedge X_m \wedge Y_1\wedge \stackrel{j}{\check{\cdots}} \wedge Y_n\,,
}
\end{array}
\end{equation}
where
$$
\begin{array}{rcl}
s_{ij}(X,Y)
&=& \big(\bar p(X_i)+1\big)\big(\bar p(X_1)+\cdots+\bar p(X_{i-1})+i+1\big) \\
&& +\big(\bar p(Y_j)+1\big)\big(\bar p(Y_1)+\cdots+\bar p(Y_{j-1})+j+1\big) \\
&& +\bar p(Y_j)\big(\bar p(X_1)+\stackrel{i}{\check{\cdots}}+\bar p(X_m)+m+1\big)\,.
\end{array}
$$
In particular, if $\mf g$ is a Lie algebra, then $s_{ij}(X,Y)=i+j$.
\end{remark}

\subsection{Representations of a Gerstenhaber algebra}

\begin{definition}\label{0220:def}
A \emph{representation} of a Gerstenhaber algebra $\mc{G}$ with parity $p$ on a superspace $V$ is 
a module structure over the commutative associative superalgebra $(\mc{G},\wedge)$,
denoted by $\iota_\cdot:\mc{G} \otimes V\to V,\,X\otimes v\mapsto\iota_X(v)$, and called \emph{contraction},
together with 
a module structure over the Lie superalgebra $(\Pi \mc{G},[\cdot,\cdot])$,
denoted by $L_\cdot:\mc{G} \otimes V\to V,\,X\otimes v\mapsto L_X(v)$, and called \emph{Lie derivative},
such that the left Leibniz rule is preserved:
\begin{equation} \label{comp1} 
[L_X,\iota_Y]
\,\Big( 
= L_X \iota_Y - (-1)^{(p(X)+\bar1)p(Y)} \iota_Y L_X
\Big)
= \iota_{[X,Y]} \,.
\end{equation}
\end{definition}

For example, letting $\iota_X=X\wedge$ and $L_X=\ad X$, we get a representation of a Gerstenhaber algebra $\mc{G}$ on itself,
called its adjoint representation.

\begin{remark}
Note that a representation of a Gerstenhaber algebra $(\mc{G},\wedge,[\cdot,\cdot])$ on $V$
is the same as a rigged representation of the Lie superalgebra $(\Pi\mc{G},[\cdot,\cdot])$
such that the rigging $X\mapsto\iota_X$ is a representation of the associative superalgebra $(\mc{G},\wedge)$.
\end{remark}

\begin{theorem}\label{0217:th}
Let $(\mf g,A)$ be a Lie superalgebroid,
and consider the Gerstenhaber algebra $\mc{G}=S_A(\Pi\mf g)$, with parity $p$.
Then any $\epsilon$-rigged representation of the Lie superalgebroid $(\mf g,A)$ on a vector superspace $V$,
extends uniquely to a representation of the Gerstenhaber algebra $\mc{G}$ on $V$ such that, for every $X,Y\in\mc{G}$,
the following $\epsilon$-\emph{right Leibniz rule} holds:
\begin{equation}\label{0217:eq1}
\begin{array}{rcl}
L_{X\wedge Y} &=& \iota_X L_Y+(-1)^{p(Y)}L_X\iota_Y-\epsilon(-1)^{p(Y)}\iota_{[X,Y]} \\
\Big( &=& \iota_X L_Y+(-1)^{p(X)p(Y)}\iota_YL_X+(1-\epsilon)(-1)^{p(Y)}\iota_{[X,Y]}
\Big)\,.
\end{array}
\end{equation}
\end{theorem}
\begin{proof}
Since the contraction $\iota_\cdot:\,\mc{G}\to\End(V)$ is a representation of the commutative associative superalgebra $(\mc{G},\wedge)$,
and it extends the rigging of the representation of the Lie superalgebra $\mf g\ltimes\Pi A$ on $V$,
it is forced to be given by the following formula:
\begin{equation}\label{0222:eq1}
\iota_{X_1\wedge\cdots\wedge X_m}=\iota_{X_1}\cdots\iota_{X_m}\,,
\end{equation}
for all $X_1,\dots,X_m\in\mf g$. 
It is immediate to check, using the assumptions that $\iota_{fX}=\iota_f\iota_X$ for all $f\in A,X\in\mf g$, 
and $[\iota_a,\iota_b]=0$ for all $a,b\in\Pi\mf g\oplus A$,
that the contraction map is a well-defined representation of the commutative associative superalgebra $(\mc{G},\wedge)$.

By assumption, the Lie derivative $L_\cdot:\,\Pi\mc{G}\to\End(V)$
is defined by extending, inductively, the representation of the Lie superalgebra $\mf g\ltimes\Pi A$
on $V$, using equation \eqref{0217:eq1}.
In order to prove that the map $L_\cdot$ is well defined, we proceed
as in the proof of Proposition \ref{schouten}.
First, we define a map $\tilde L_\cdot$ from the tensor algebra $\mc T(\Pi\mf g\oplus A)$ to $\End(V)$
(which reverses the parity),
extending $L_\cdot:\,\mf g\oplus\Pi A\to\End(V)$, inductively, by saying that $\tilde L_{X\otimes Y}$
is given by the RHS in \eqref{0217:eq1}.
By applying \eqref{0217:eq1} twice,
we get that both $\tilde L_{X\otimes(Y\otimes Z)}$ and $\tilde L_{(X\otimes Y)\otimes Z}$ are equal to
\begin{equation}\label{0223:eq1}
\begin{array}{l}
\iota_X\iota_Y \tilde L_Z+(-1)^{p(Z)}\iota_X \tilde L_Y\iota_Z+(-1)^{p(Y)+p(Z)}\tilde L_X\iota_Y\iota_Z \\
-\epsilon(-1)^{p(Z)}\Big(
\iota_X\iota_{[Y,Z]}+(-1)^{p(Y)}\iota_{[X,Y]}\iota_Z+(-1)^{p(X)p(Y)}\iota_Y\iota_{[X,Z]}
\Big)\,,
\end{array}
\end{equation}
proving that $\tilde L_\cdot$ preserves the associativity relation for the tensor product.
Above we denoted, by an abuse of notation, the lifts of the contraction map and the Schouten bracket
to the tensor algebra $\mc T(\Pi\mf g\oplus A)$ by $\iota_\cdot$ and $[\cdot,\cdot]$ respectively. 
Hence $\tilde L_\cdot$ is a well defined map: $\mc T(\Pi\mf g\oplus A)\to\End(V)$.
Moreover, the fact that $\tilde L_\cdot$ preserves the defining relations \eqref{0222:eq2}$(i)$ and $(ii)$
is encoded in the assumption that $V$ is an $\epsilon$-rigged representation 
of the Lie superalgebroid $(\mf g,A)$.
More precisely, for the relation \eqref{0222:eq2}$(i)$ with $a=f,b=g\in A$,
we have that $\tilde L_{f\otimes g}=\iota_fL_g+L_f\iota_g$, which is the same as $\tilde L_{g\otimes f}$ 
thanks to condition $(i)$ in Definition \ref{0220:def2}(b) and the fact that $L_f$ and $\iota_g$ commute.
For $a=X,b=Y\in\mf g$, we have 
$\tilde L_{X\otimes Y}-(-1)^{p(X)p(Y)}\tilde L_{Y\otimes X}
=[\iota_X,L_Y]+(-1)^{p(Y)}[L_X,\iota_Y]
-\epsilon\big(\iota_{[X,Y]}+(-1)^{1+(p(X)+\bar1)(p(Y)+\bar1)}\iota_{[Y,X]}\big)$,
and this is zero by the definition of $\epsilon$-rigged representation and 
by the skewcommutativity of the Lie bracket on $\mf g$.
Finally, when $a=X\in\mf g,\,b=f\in A$, we have that 
both $\tilde L_{f\otimes X}$ and $\tilde L_{X\otimes f}$ are equal to
$\iota_fL_X+(-1)^{p(X)}L_f\iota_X-\epsilon\iota_{X(f)}$,
thanks to the assumption that $[L_X,\iota_f]=\iota_{X(f)}$.
Moreover, this expression is equal to $L_{fX}$, by condition $(iii)$ in Definition \ref{0220:def2}(b),
thus proving that $\tilde L_\cdot$ preserves the relation \eqref{0222:eq2}$(ii)$.
What we just proved allows us to conclude that two-sided ideal $\mc K\subset\mc T(\Pi\mf g\oplus A)$ 
generated by the relations \eqref{0222:eq2}$(i)$ and $(ii)$
is in the kernel of $\tilde L_\cdot$ (this is not immediate since 
$\tilde L_\cdot$ is not a homomorphism of associative algebras).
Indeed, both the contraction map $\iota_\cdot$ and the Schouten bracket $[\cdot,\cdot]$ are defined
on the symmetric superalgebra $S_A(\Pi\mf g)$, and hence, when lifted to the tensor algebra
$\mc T(\Pi\mf g\oplus A)$, they map the ideal $\mc K$ to zero.
Therefore it is clear, from the expression \eqref{0223:eq1} for $\tilde L_{X\otimes Y\otimes Z}$,
that $\mc K$ is in the kernel of $\tilde L_\cdot$.
Hence $\tilde L_\cdot$ factors through a well defined map $L_\cdot:\, S_A(\Pi\mf g)\to\End(V)$.

To complete the proof we have to check that the pair $(\iota_\cdot,L_\cdot)$ is a Gerstenhaber algebra representation.
By assumption the left Leibniz rule \eqref{comp1} holds for $X,Y\in\Pi\mf g\oplus A$.
Therefore, in order to prove \eqref{comp1} by induction, we note that,
$$
\begin{array}{rcl}
[L_X,\iota_{Y\wedge Z}] 
&=& [L_X,\iota_Y\iota_Z] 
= [L_X,\iota_Y]\iota_Z+(-1)^{(p(X)+\bar1)p(Y)}\iota_Y[L_X,\iota_Z] \\
&=& \iota_{[X,Y]}\iota_Z+(-1)^{(p(X)+\bar1)p(Y)}\iota_Y\iota_{[X,Z]}
= \iota_{[X,Y\wedge Z]}\,,
\end{array}
$$
for all $X,Y,Z\in\mc{G}$ such that $Y,Z$ have degree at least 1, and
$$
\begin{array}{rcl}
{[L_{X\wedge Y},\iota_Z]} 
&=& [\iota_XL_Y+(-1)^{p(Y)}L_X\iota_Y-\epsilon(-1)^{p(Y)}\iota_{[X,Y]},\iota_Z] \\
&=& \iota_X[L_Y,\iota_Z]+(-1)^{p(Y)(p(Z)+\bar1)}[L_X,\iota_Z]\iota_Y \\
&=& \iota_X\iota_{[Y,Z]}+(-1)^{p(Y)(p(Z)+\bar1)}\iota_{[X,Z]}\iota_Y
= \iota_{[X\wedge Y,Z]}\,,
\end{array}
$$
for all $X,Y,Z\in\mc{G}$ such that $X,Y$ have degree at least 1.
In the computations above we used the inductive assumptions, formula \eqref{0217:eq1}, and the commutation 
relation $[\iota_X,\iota_Y]=0$.

Finally, we use the above results to prove, by induction, that $L_\cdot:\,\Pi\mc{G}\to\End(V)$
is a Lie superalgebra homomorphism: $[L_X,L_Y]=L_{[X,Y]}$ for all $X,Y\in\Pi\mc{G}$.
If both $X,Y$ are in $\mf g\oplus\Pi A$, this holds by assumption.
Moreover, by skewcommutativity, it suffices to check the homomorphism condition for $X,\,Y\wedge Z$,
where both $Y$ and $Z$ have degree greater or equal than 1:
\begin{equation}\label{0224:eq1}
[L_X,L_{Y\wedge Z}]=L_{[X,Y\wedge Z]}\,.
\end{equation}
The LHS of \eqref{0224:eq1} is, by inductive assumption,
$$
\begin{array}{l}
[L_X,L_{Y\wedge Z}] 
= [L_X,\iota_YL_Z+(-1)^{p(Z)}L_Y\iota_Z-\epsilon(-1)^{p(Z)}\iota_{[Y,Z]}] \\
\,\,\,\,\,\,\,\,\,
= [L_X,\iota_Y]L_Z +(-1)^{(p(X)+\bar1)p(Y)}\iota_Y[L_X,L_Z]
+(-1)^{p(Z)}[L_X,L_Y]\iota_Z \\
\,\,\,\,\,\,\,\,\,\,\,\,\,\,\,\,\,\,
 +(-1)^{p(Z)+(p(X)+\bar1)(p(Y)+\bar1)}L_Y[L_X,\iota_Z]
-\epsilon(-1)^{p(Z)}[L_X,\iota_{[Y,Z]}] \\
\,\,\,\,\,\,\,\,\,
= \iota_{[X,Y]}L_Z+(-1)^{(p(X)+\bar1)p(Y)}\iota_YL_{[X,Z]}
+(-1)^{p(Z)}L_{[X,Y]}\iota_Z \\
\,\,\,\,\,\,\,\,\,\,\,\,\,\,\,\,\,\,
 +(-1)^{p(Z)+(p(X)+\bar1)(p(Y)+\bar1)}L_Y\iota_{[X,Z]} 
-\epsilon(-1)^{p(Z)}\iota_{[X,[Y,Z]]}\,.
\end{array}
$$
Similarly, the RHS of \eqref{0224:eq1} is
$$
\begin{array}{l}
L_{[X,Y\wedge Z]}
= L_{[X,Y]\wedge Z}+(-1)^{(p(X)+\bar1)p(Y)}L_{Y\wedge[X,Z]} \\
\,\,\,\,\,\,\,\,\,
= \iota_{[X,Y]}L_Z+(-1)^{p(Z)}L_{[X,Y]}\iota_Z -\epsilon(-1)^{p(Z)}\iota_{[[X,Y],Z]} \\
\,\,\,\,\,\,\,\,\,\,\,\,\,\,\,\,\,\,
+(-1)^{(p(X)+\bar1)p(Y)}\Big(
\iota_YL_{[X,Z]}+(-1)^{p(Z)+(p(X)+\bar1)}L_Y\iota_{[X,Z]} \\
\,\,\,\,\,\,\,\,\,\,\,\,\,\,\,\,\,\,
-\epsilon(-1)^{p(X)+p(Z)+1}\iota_{[Y,[X,Z]]}\Big)\,.
\end{array}
$$
Equation \eqref{0224:eq1} now follows by the Jacobi identity for the Schouten bracket.
\flushright\qed
\end{proof}

\begin{remark}\label{0225:rem}
One can show that among all possible expressions for $L_{X\wedge Y}$ of the form
$$
a\iota_XL_Y+b\iota_YL_X+cL_X\iota_Y+dL_Y\iota_X+e\iota_{[X,Y]}\,,
$$
with $a,b,c,d,e\in\mb F$,
only those given by \eqref{0217:eq1} satisfy the left Leibniz rule \eqref{comp1},
and therefore give rise to a representation of the Gerstenhaber algebra $S_A(\Pi\mf g)$.
\end{remark}

\begin{example}
The adjoint representation of the Gerstenhaber algebra $S_A(\Pi\mf g)$ on itself,
satisfies the $\epsilon$-right Leibniz formula with $\epsilon=1$.

In the next subsection we will see how to construct representations of the Gerstenhaber algebra
$S_A(\Pi\mf g)$ satisfying the $\epsilon$-right Leibniz formula with $\epsilon=0$,
starting from a $(\mf g,A)$-complex and using Cartan's formula.
\end{example}

\begin{example}
If a Lie superalgebroid $(\mf g,A)$ is such that the action of $\mf g$ on $A$ is trivial,
then every $\epsilon_0$-rigged representation of $(\mf g,A)$ on a vector superspace $V$,
for some $\epsilon_0$, is automatically $\epsilon$-rigged for all $\epsilon$.
Hence, by Theorem \ref{0217:th}, we automatically get in this case
a family of representations of the Gerstenhaber algebra $\mc G=S_A(\Pi\mf g)$
on $V$, depending on the parameter $\epsilon$, which satisfies 
the $\epsilon$-right Leibniz formula \eqref{0217:eq1}.
In particular, in this case, the adjoint representation of $\mc G=S_A(\Pi\mf g)$ on itself
admits a 1-parameter family of deformations.
\end{example}

\begin{remark}
Using the $\epsilon$-right Leibniz rule \eqref{0217:eq1}
and recalling the relation \eqref{0225:eq3} for the parity in $\mc G=S_A(\Pi\mf g)$, 
one can find an explicit formula
for the Lie derivative $L_X$, for an arbitrary element $X=X_1\wedge\cdots\wedge X_m$ with $X_i\in\mf g$:
\begin{equation}\label{0226:eq1}
\begin{array}{rcl}
L_X &=&
\displaystyle{
\sum_{i=1}^m (-1)^{\bar p(X_{i+1})+\cdots+\bar p(X_m)+m+i}
\iota_{X_1}\cdots\iota_{X_{i-1}}L_{X_i}\iota_{X_{i+1}}\cdots\iota_{X_m} 
}\\
&& \displaystyle{
-\epsilon \sum_{1\leq i<j\leq m}
(-1)^{t_{ij}(X)}
\iota_{X_1}\cdots\iota_{[X_i,X_j]}\stackrel{j}{\check{\cdots}}\iota_{X_m}\,,
}
\end{array}
\end{equation}
where
$$
\begin{array}{rcl}
&& t_{ij}(X) = \bar p(X_{i+1})+\cdots+\bar p(X_m)+m+i \\
&&\,\,\, +\big(\bar p(X_j)+\bar1\big)
\big(\bar p(X_{i+1})+\cdots+\bar p(X_{j-1})+i+j+\bar1\big)\,.
\end{array}
$$
Using the second formula in \eqref{0217:eq1} one can get a different expression for $L_X$,
which in the case $\epsilon=1$ gives
$$
L_X = 
\sum_{i=1}^m (-1)^{\big(\bar p(X_i)+\bar1\big)\big(\bar p(X_{i+1})+\cdots+\bar p(X_m)+m+i\big)}
\iota_{X_1}\stackrel{i}{\check{\cdots}}\iota_{X_m}L_{X_i}\,.
$$
We will use formula \eqref{0226:eq1} in the special case when $\mf g$ is a Lie algebra
and $\epsilon=0$. In this case it reads
\begin{equation}\label{0304:eq1}
L_X =
\sum_{i=1}^m (-1)^{m+i}
\iota_{X_1}\cdots\iota_{X_{i-1}}L_{X_i}\iota_{X_{i+1}}\cdots\iota_{X_m} 
\end{equation}
\end{remark}

\subsection{Calculus structure on a complex}

\begin{definition}
A \emph{calculus structure} of a Gerstenhaber algebra $\mc G$ on the complex $(\Omega,d)$
is a representation of the Gerstenhaber algebra $\mc G$ on the vector superspace $\Omega$,
denoted by $\iota_\cdot:\,\mc G\to\End(\Omega),\,L_\cdot:\,\Pi\mc G\to\End(\Omega)$,
satisfying Cartan's formula \eqref{comp3}.
Often we will denote this calculus structure as the pair $(\mc G,\Omega)$.
\end{definition}

\begin{remark}\label{0228:rem}
Given a Gerstenhaber algebra $\mc G$ and a complex $(\Omega,d)$,
in order to construct a calculus structure of $\mc G$ on $\Omega$, 
it suffices to define a representation $\iota_\cdot$ of the associative superalgebra
$(\mc G,\wedge)$ on the superspace $\Omega$, satisfying
$$
[[\iota_X,d],\iota_Y]=\iota_{[X,Y]}\,\,,\,\,\,\,
\text{ for all } X,Y\in\mc G\,,
$$
and to define the Lie derivative $L_X$ by Cartan's formula.
This follows from Lemma \ref{0220:lem1}.
\end{remark}

Note that a $\mc G$-complex, for the Gerstenhaber algebra $\mc G$,
is automatically a $\Pi\mc G$-complex, for the Lie algebra $\Pi\mc G$
(and hence for any subalgebra $\mf g\subset\Pi\mc G$).
The following result can be viewed as a converse statement:
to any $\mf g$-complex, or more generally to any $(\mf g,A)$-complex, $(\Omega,d)$, 
we associate a calculus structure of the Gerstenhaber algebra $\mc G=S_A(\Pi\mf g)$
on $\Omega$.
%
\begin{theorem}\label{0228:th}
Let $(\mf g,A)$ be a Lie superalgebroid.
Then any $(\mf g,A)$-complex $(\Omega,d)$ extends uniquely
to a calculus structure of the Gerstenhaber algebra $\mc G=S_A(\Pi\mf g)$ on 
the complex $(\Omega,d)$.
Moreover, the contraction $\iota_\cdot$ and the Lie derivative $L_\cdot$ of this calculus structure
satisfy the $0$-right Leibniz rule \eqref{0217:eq1}.
\end{theorem}
\begin{proof}
By Proposition \ref{0220:prop1}, we have a $0$-rigged representation of the Lie superalgebroid
$(\mf g,A)$ on $\Omega$, satisfying Cartan's formula $[\iota_a,d]=L_a$ for all $a\in\mf g\oplus\Pi A$.
By Theorem \ref{0217:th}, this further extends to a representation of the Gerstenhaber algebra
$\mc G=S_A(\Pi\mf g)$ on $\Omega$ satisfying the $0$-right Leibniz rule.
To prove that this representation is indeed a calculus structure,
we only have to check that Cartan's formula \eqref{comp3} holds for every $X\in\mc G$.
We already know that it holds for $X\in\Pi\mf g\oplus A$, and we have, by induction,
$$
\begin{array}{rcl}
[\iota_{X\wedge Y},d]
&=& [\iota_X\iota_Y,d]
=\iota_X[\iota_Y,d]+(-1)^{p(Y)}[\iota_X,d]\iota_Y \\
&=& \iota_XL_Y+(-1)^{p(Y)}L_X\iota_Y=L_{X\wedge Y}\,.
\end{array}
$$
In the last identity we used the $0$-right Leibniz rule \eqref{0217:eq1}.
Uniqueness of the extension is clear,
since $\iota_\cdot$ extends uniquely to a representation of the associative algebra $(\mc G,\wedge)$,
and $L_\cdot$ is given by Cartan's formula.
\flushright\qed
\end{proof}

\begin{example}
Recall from Example \ref{0220:ex2} that the de Rham complex $(\Omega,d)$
carries a structure of a $(\mf g,A)$-complex,
where $A$ is the algebra of smooth functions
and $\mf g$ is the Lie algebra of smooth vector fields on a smooth manifold $M$.
Hence, by Theorem \ref{0228:th}, this extends uniquely to a calculus structure $(\mc G,\Omega)$,
where $\mc G$ is the Gerstenhaber algebra $S_A(\Pi\mf g)$.
The contraction by a polyvector field $X=X_1\wedge\cdots\wedge X_m$ is
$\iota_X=\iota_{X_1}\cdots\iota_{X_m}$, and the Lie derivative by $X$
is given by \eqref{0304:eq1}.
\end{example}

\begin{remark}
Note that both Cartan's formula \eqref{comp3} and the compatibility condition \eqref{comp1} differ
by a sign from those in \cite{DTT}.
The reason for this change is that, as defined in \cite{DTT},
a calculus is not a representation of the Gerstenhaber algebra $\mc G$.
Also, in \cite{DTT} the definition of a calculus includes the $0$-right Leibniz rule 
(rather a different version of it by a sign), which is of course superfluous
since it is equivalent to the following trivial identity:
$$
[\iota_X\iota_Y,d]=\iota_X[\iota_Y,d]+(-1)^{p(Y)}[\iota_X,d]\iota_Y\,.
$$
\end{remark}

\subsection{$\mb Z_+$-graded calculus structures}

Usually, a complex $(\Omega,d)$ is endowed with a $\mb Z_+$-grading
$\Omega=\bigoplus_{n\in\mb Z_+}\Omega^n$, such that $d(\Omega^{n-1})\subset\Omega^{n}$.
A Gerstenhaber algebra $\mc G$ is called $\mb Z_+$-graded,
with grading $\mc G=\bigoplus_{n\in\mb Z_+}\mc G_n$,
if $\mc G_m\wedge\mc G_n\subset\mc G_{m+n}$ and $[\mc G_m,\mc G_n]\subset\mc G_{m+n-1}$.
A $\mb Z_+$-graded calculus structure of a $\mb Z_+$-graded Gerstenhaber algebra $\mc G$
on a $\mb Z_+$-graded complex is, by definition, a calculus structure 
such that $\iota(\mc G_m\times\Omega^n)\subset\Omega^{n-m}$
and $L(\mc G_m\times\Omega^n)\subset\Omega^{n-m+1}$.

\begin{example}\label{0301:ex}
If $(\mf g,A)$ is a Lie superalgebroid, then $S_A(\Pi\mf g)$ is a $\mb Z_+$-graded
Gerstenhaber algebra with the usual $\mb Z_+$-grading of the symmetric algebra.
Moreover, suppose we have a $(\mf g,A)$-complex $(\Omega,d)$, such that
$(\Omega,d)$ is a $\mb Z_+$-graded complex, and $\iota_f(\Omega^n)\subset\Omega^n$ for all $f\in A$
and $\iota_X(\Omega^n)\subset\Omega^{n-1}$ for all $X\in\Pi\mf g$.
Then the corresponding calculus structure $(S_A(\Pi\mf g),\Omega)$ given by Theorem \ref{0228:th}
is $\mb Z_+$-graded.
\end{example}

\begin{remark}

Given a $\mb Z$-graded vector superspace $V=\bigoplus_{n\in\mb Z}V^n$,
with parity $p$ and degree $\deg$,
we let $\Pi V$ be the $\mb Z$-graded vector superspace 
with opposite parity: $\bar p(v)=p(v)+\bar1$,
and with degree shifted by 1: $\overline\deg(v)=\deg(v)+1$.
In other words, $\Pi V=\bigoplus_{n\in\mb Z}(\Pi V)^n$, 
where $(\Pi V)^n=\Pi(V^{n-1})$.
Using this notation, if $(\Omega,d)$ is a $\mb Z_+$-graded complex, 
it means that $d$ is a parity preserving linear map of degree zero
from $\Omega$ to $\Pi\Omega$.
Moreover, if we have a $\mb Z_+$-graded Gerstenhaber algebra $\mc G=\bigoplus_{n\in\mb Z_+}\mc G_n$,
we consider it as a $\mb Z_-$-graded superspace $\mc G=\bigoplus_{n\in\mb Z_-}\mc G^n$,\
by letting $\mc G^n=\mc G_{-n}$.
With this notation, $(\mc G,\wedge)$ is a $\mb Z_-$-graded commutative associative superalgebra,
and $(\Pi\mc G,[\cdot,\cdot])$ is a $\mb Z_-$-graded Lie superalgebra.
Moreover, in a $\mb Z_+$-graded calculus structure $(\mc G,\Omega)$,
both the contraction map $\iota_\cdot:\,\mc G\times\Omega\to\Omega$
and the Lie derivative $L_\cdot:\Pi\mc G\times\Omega\to\Omega$
become parity preserving maps of degree zero.
\end{remark}
\newpage

\subsection{Morphisms of calculus structures}

\begin{definition}\label{0321:def}
A \emph{morphism} $(\mc G,\Omega)\to (\mc G^\prime,\Omega^\prime)$
of a calculus structure $(\mc G,\Omega)$ to a calculus structure $(\mc G^\prime,\Omega^\prime)$
is a Gerstenhaber algebra homomorphism $\Phi:\,\mc G^\prime\to\mc G$,
together with a homomorphism of complexes $\Psi:\,\Omega\to\Omega^\prime$,
such that, for $X^\prime\in\mc G^\prime$ and $\omega\in\Omega$, we have
\begin{equation}\label{0321:eq1}
\iota_{X^\prime}\,\Psi(\omega)=\Psi(\iota_{\Phi(X^\prime)}\omega)\,.
\end{equation}
Note that, by Cartan's formula, equation \eqref{0321:eq1} holds if the contraction $\iota_\cdot$
is replaced by the Lie derivative $L_\cdot$.

A morphism of $\mb Z_+$-graded calculus structures is one that preserves the $\mb Z_+$-gradings.
\end{definition}

\begin{example}\label{0321:ex}
Let $\mc{G}$ be a Gerstenhaber algebra with a calculus structure on the complex $(\Omega,d)$.
Let $\partial$ be an even endomorphism of the superspace $\Omega$, such that $[d,\partial]=0$. 
Let $\mc{G}^\partial\,=\,\big\{X\in\mc{G}\,\big|\,[\iota_X,\partial]=0\big\}\subset\mc{G}$.
Notice that, by Cartan's formula, $[L_X,\partial]=0$ for all $X\in\mc{G}^\partial$.
It follows that $\mc{G}^\partial$ is a subalgebra 
of the Gerstenhaber algebra $(\mc{G},\wedge,[\cdot,\cdot])$,
and that $(\partial\Omega,d)$ is a subcomplex of $(\Omega,d)$,
such that $\partial\Omega$ is a submodule over the Gerstenhaber algebra $\mc{G}^\partial$.
We can thus consider the quotient $\mc{G}^\partial$-module $\Omega/\partial \Omega$.
This defines an induced calculus structure of the Gerstenhaber algebra $\mc{G}^\partial$
on the complex $(\Omega/\partial \Omega,d)$, which is called the \emph{reduced calculus structure}.
We have the obvious morphism of calculus structures
$(\mc G,\Omega)\to (\mc G^\partial,\Omega/\partial\Omega)$
given by the inclusion map of $\mc G^\partial\to\mc G$ 
and the quotient map $\Omega\to\Omega/\partial\Omega$.
\end{example}


\section{Calculus structure on the Lie algebra complex}
\label{sec:3}

\subsection{Discrete case}

Let $\mf{g}$ be a Lie algebra and $A$ a $\mf{g}$-module, 
endowed with the structure of a unital commutative associative algebra,
on which $\mf{g}$ acts by derivations.
By Example \ref{0219:ex1}, we have a Lie algebroid $(A\otimes\mf g,A)$.
Hence, by Proposition \ref{schouten} and Example \ref{0301:ex}, 
we have a $\mb Z_+$-graded Gerstenhaber algebra $S_A(\Pi(A\otimes\mf g))$,
which we denote by 
$\Delta_\bullet=\Delta_\bullet(\mf g,A)=\bigoplus_{h\in\mb Z_+}\Delta_h(\mf g,A)$.
Note that we have the canonical isomorphism 
$\Delta_\bullet(\mf g,A)
=A\otimes S(\Pi\mf g)=A\otimes\bigwedge\mf g$
(the latter identity is only of vector spaces, not superspaces,
and $A\otimes\bigwedge\mf g$ is considered as a vector superspace with induced
$\mb Z/2\mb Z$-grading).
We shall call this the \emph{superspace of chains}.

Dualizing, the \emph{superspace of cochains} is the $\mb Z_+$-graded vector superspace
$\Delta^\bullet(\mf g,A)=\Hom_A(\Delta_\bullet(\mf g,A),A)=\bigoplus_{k\in\mb Z_+}\Delta^k(\mf g,A)$,
where the $k$-th component is $\Delta^k(\mf g,A)=\Hom_A(\Delta_k(\mf g,A),A)$.
Again, this is the same as the traditional definition due to the canonical
isomorphism of vector spaces
$\Delta^k(\mf g,A)=\Hom_A(A\otimes\bigwedge^k\mf g,A)=\Hom_{\mb F}(\bigwedge^k\mf g,A)$.
Again, we consider the latter as a vector superspace with induced
$\mb Z/2\mb Z$-grading.
The superspace $\Delta^\bullet(\mf g,A)$ is a $\mb Z_+$-graded complex
with differential $d:\,\Delta^\bullet(\mf g,A)\to\Delta^\bullet(\mf g,A)$,
defined by the usual formula, see e.g. \cite{F} ($X_i\in\mf g$):
\begin{equation}\label{0301:eq1}
\begin{array}{c}
\displaystyle{
(d\omega)(X_1\wedge\dots\wedge X_{k+1})
= \sum_{i=1}^{k+1}(-1)^{i+1} X_i\big(\omega(X_1\wedge\stackrel{i}{\check{\cdots}}\wedge X_{k+1})\big)
}\\
\displaystyle{
+ \sum_{\substack{i,j=1\\i<j}}^{k+1}(-1)^{i+j} 
\omega([X_i,X_j]\wedge X_1\wedge\stackrel{i}{\check{\cdots}}\stackrel{j}{\check{\cdots}}\wedge X_{k+1})
\,.
}
\end{array}
\end{equation}

We next define 
a structure of a $(A\otimes\mf g,A)$-complex on $(\Delta^\bullet(\mf g,A),d)$,
see Definition \ref{0220:def1}(b).
For $f\in A$, we let $\iota_f\in\,\End(\Delta^\bullet(\mf g,A))$ be the 
multiplication by $f$, given by the obvious
left $A$-module structure on $\Delta^\bullet(\mf g,A)=\Hom_{\mb F}(\bigwedge\mf g,A)$.
For $X\in\mf g$, we let the contraction operator 
$\iota_X:\,\Delta^k(\mf g,A)\to\Delta^{k-1}(\mf g,A)=\Hom_{\mb F}(\bigwedge^{k-1}\mf g,A)$,
be zero for $k=0$, and, for $k\geq1$, be given by
\begin{equation}\label{0302:eq1}
\begin{array}{c}
(\iota_X\omega)(Y)=\omega(X\wedge Y)\,\,,\qquad Y
\in \bigwedge^{k-1}\mf g\,.
\end{array}
\end{equation}
We then use Cartan's formula \eqref{comp3} to define, for $f\in A$,
$L_f:\,\Delta^k(\mf g,A)\to\Delta^{k+1}(\mf g,A)$,
and, for $X\in\mf g$, $L_X:\,\Delta^k(\mf g,A)\to\Delta^k(\mf g,A),\,k\in\mb Z_+$.
We have ($X_i\in\mf g$):
\begin{equation}\label{0302:eq2}
\begin{array}{rcl}
(L_f\omega)(X_1\wedge\cdots\wedge X_{k+1})
&=& 
\displaystyle{
\sum_{i=1}^{k+1}(-1)^i X_i(f) \omega(X_1\wedge\stackrel{i}{\check{\cdots}}\wedge X_{k+1})\,,
}\\
(L_X\omega)(X_1\wedge\cdots\wedge X_k)
&=& \displaystyle{
X\big(\omega(X_1\wedge\cdots\wedge X_k)\big) 
}\\
&&\displaystyle{
- \sum_{i=1}^{k} \omega(X_1\wedge\cdots[X,X_i]\cdots\wedge X_{k+1})\,.
}
\end{array}
\end{equation}

\begin{theorem}\label{0302:th}
\begin{enumerate}[(a)]
\item
The operators $\iota_f,\,f\in A$, and $\iota_X,\,X\in\mf g$, define a structure 
of a $(A\otimes\mf g,A)$-complex on $(\Delta^\bullet(\mf g,A),d)$.
\item
This extends uniquely to a $\mb Z_+$-graded 
calculus structure of the Gerstenhaber algebra $\Delta_\bullet(\mf g,A)$
on the complex $\Delta^\bullet(\mf g,A)$.
\end{enumerate}
\end{theorem}
\begin{proof}
First, we check that the contraction maps $\iota_X,\,X\in\mf g$,
define a structure of a $\mf g$-complex on $\Delta^\bullet(\mf g,A)$.
This is straightforward, using the second formula in \eqref{0302:eq2}.
Then, in order to prove that we have a structure 
of a $(A\otimes\mf g,A)$-complex on $(\Delta^\bullet(\mf g,A),d)$,
it suffices to check conditions $(i)$, $(ii)$ and $(iii)$ in Lemma \ref{0302:lem}.
Condition $(i)$ is clear, condition $(ii)$ follows from the first formula in \eqref{0302:eq2},
and condition $(iii)$ follows from the second formula in \eqref{0302:eq2}.
This proves part (a).
For part (b), by Theorem \ref{0228:th} the structure of a $(A\otimes\mf g,A)$-complex 
on $(\Delta^\bullet(\mf g,A),d)$ extends 
to a calculus structure $(\Delta_\bullet(\mf g,A),\Delta^\bullet(\mf g,A))$,
and, by Example \ref{0301:ex}, this calculus structure is $\mb Z_+$-graded.
\flushright\qed
\end{proof}

It is not hard to find a general formula for the contraction operator and the Lie derivative
for an arbitrary element 
$f\otimes X=f\otimes X_1\wedge\cdots\wedge X_h\in A\otimes\bigwedge^h\mf g=\Delta_h(\mf g,A)$.
Let $\omega\in\Delta^k(\mf g,A)=\Hom_{\mb F}(\bigwedge^k\mf g,A)$.
If $h\leq k$ and $Y\in\bigwedge^{k-h}\mf g$,
we have, from \eqref{0222:eq1},
\begin{equation}\label{0304:eq2}
(\iota_{f\otimes X}\omega)(Y) = (-1)^{\frac{h(h-1)}{2}}f \omega(X\wedge Y) \, .
\end{equation}
If $h\leq k+1$ and $Y=X_{h+1}\wedge\cdots\wedge X_{k+1}\in\bigwedge^{k+1-h}\mf g$, we have, from \eqref{0304:eq1},
\begin{equation}\label{0304:eq3}
\begin{array}{c}
\displaystyle{
(L_{f\otimes X}\omega)(Y)
=
(-1)^{\frac{h(h-1)}2}\bigg(
\sum_{i=h+1}^{k+1} (-1)^{i} X_i(f) 
\omega(X_1\wedge\stackrel{i}{\check{\cdots}}\wedge X_{k+1}) 
}\\
\displaystyle{
- \sum_{i=1}^h (-1)^{i}
f X_i\big(\omega(X_1\wedge\stackrel{i}{\check{\cdots}}\wedge X_{k+1})\big) 
}\\
\displaystyle{
+ \sum_{i=1}^h\sum_{j=i+1}^{k+1} (-1)^{i} 
f \omega(X_1\wedge\stackrel{i}{\check{\cdots}}\wedge[X_i,X_j]\wedge X_{k+1})
\bigg)\,.
}
\end{array}
\end{equation}

\subsection{Linearly compact case}\label{sec:3.2}

In this subsection we shall assume that $\mf g$ is a Lie algebra with a linearly compact topology,
acting continuously by derivations on a unital commutative associative algebra $A$,
with discrete topology.
For a definition of linearly compact space and relevant properties
which we shall use, see e.g. \cite{G,CK}.

Recall that a {\em linear topology} on a vector space $V$ over $\mb F$
is a topology for which there exists a fundamental system of neighborhoods
of zero, consisting of vector subspaces of $V$. A vector space
$V$ with linear topology is called {\em linearly compact} if there
exists a fundamental system of neighborhoods of zero, consisting of subspaces
of finite codimension in $V$ and, in addition, $V$ is complete in this
topology (equivalently, if $V$ is a topological direct product of some number
of copies of $\mb F$ with discrete topology).
If $\varphi: V \rightarrow U$ is a continuous map
of vector spaces with linear topology and $V$ is linearly
compact, then $\varphi(V)$ is linearly compact.
The basic examples of linearly compact spaces are finite-dimensional
vector spaces with the discrete topology, and the space of formal
power series $V[[x_1, \dots, x_m]]$ over a finite-dimensional vector 
space $V$, with
the topology defined by taking as a fundamental system of
neighborhoods of $0$ the subspaces $\{x_1^{j_1}\dots x_m^{j_m}
V[[x_1, \dots, x_m]]\}_{(j_1, \dots, j_m)\in\mb Z_+^m}$. 

Let $U$ and $V$ be two vector spaces with linear topology, and let
$\Hom^c_{\mb F}(U,V)$ denote the space of all continuous linear maps from $U$ to $V$.
We endow the vector space $\Hom^c_{\mb F}(U,V)$ with a ``compact-open'' linear 
topology, the fundamental system of neighborhoods of zero
being $\{\Omega_{K,W}\}$, where $K$ runs over all linearly compact
subspaces of $U$, $W$ runs over all open subspaces of $V$,
and 
$\Omega_{K,W}=\{\varphi\in \Hom^c_{\mb F}(U,V)~|~\varphi(K)\subset W\}$.
In particular, we define the
dual space of $V$ as $V^*=\Hom^c_{\mb F}(V,\mb F)$, where $\mb F$ is endowed with the 
discrete topology. Then $V$ is linearly compact
if and only if $V^*$ is discrete. Note also that a discrete space $V$
is linearly compact if and only if $\dim V<\infty$.
The tensor product of two vector spaces $U$ and $V$ with linear topology
is defined as
$$
U\hat{\otimes} V=\Hom{}^c_{\mb F}(U^*, V)\,.
$$
Thus, $U\hat{\otimes} V=U\otimes V$ if both $U$ and $V$ are discrete
and $U\hat{\otimes} V=(U^*\otimes V^*)^*$ if both $U$ and $V$ are
linearly compact. Hence the tensor product of linearly compact spaces
is linearly compact.

We can construct a ``topological'' calculus structure 
$(\hat\Delta_\bullet(\mf g,A),\check\Delta^\bullet(\mf g,A))$,
associated to the Lie algebra $\mf g$ and its representation on $A$.

We let $\hat\Delta_\bullet(\mf g,A)=\bigoplus_{h\in\mb Z_+}\hat\Delta_h(\mf g,A)$,
where $\hat\Delta_h(\mf g,A)=A\hat\otimes\hat\bigwedge^h\mf g$,
where $\hat\bigwedge^h\mf g$ is the quotient of $\mf g^{\hat\otimes h}$
by the usual skewsymmetry relations.
Clearly $\Delta_\bullet(\mf g,A)$ is a subspace of $\hat\Delta_\bullet(\mf g,A)$,
and its Gerstenhaber algebra structure,
given by the wedge product and formulas \eqref{0225:eq5} and \eqref{0225:eq6} for the Lie bracket,
extends by continuity to the whole $\hat\Delta_\bullet(\mf g,A)$.

Next, we let $\check\Delta^\bullet(\mf g,A)=\bigoplus_{k\in\mb Z_+}\check\Delta^k(\mf g,A)$,
where $\check\Delta^k(\mf g,A)=\Hom^c_{\mb F}(\hat\bigwedge^k\mf g,A)$.
Since $\bigwedge^k\mf g$ is dense in $\hat\bigwedge^k\mf g$, 
a continuous linear map from $\hat\bigwedge^k\mf g$ to $A$ is determined by its
restriction to $\bigwedge^k\mf g$.
Hence, we can view
$\check\Delta^\bullet(\mf g,A)$ as a subspace of $\Delta^\bullet(\mf g,A)$.
One easily checks that the differential $d$ on $\Delta^\bullet(\mf g,A)$
leaves this subspace invariant, giving it a structure of a $\mb Z_+$-graded complex.

Finally, it is easy to check that
the calculus structure $(\Delta_\bullet(\mf g,A),\Delta^\bullet(\mf g,A))$
extends, by continuity, to a well defined calculus structure
$(\hat\Delta_\bullet(\mf g,A),\check\Delta^\bullet(\mf g,A))$,
and the inclusion maps 
$\Delta_\bullet(\mf g,A)\to\hat\Delta_\bullet(\mf g,A),\,
\check\Delta^\bullet(\mf g,A)\to\Delta^\bullet(\mf g,A)$,
define a morphism of calculus structures.


\section{Calculus structure on the Lie conformal algebra complex}
\label{sec:4}

\subsection{Preliminaries on Lie conformal algebras and their modules}
\label{sec:4.1}

In this section we review basic properties of Lie conformal algebras and their modules, 
following \cite{K}.

\begin{definition} 
A \emph{Lie conformal algebra} $R$ is an $\mb F[\partial]$-module equipped 
with a $\lambda$-\emph{bracket}, that is 
an $\mb F$-linear map $R\otimes R \to \mb F[\lambda]\otimes R$, 
denoted $a\otimes b\mapsto[a_\lambda b]$,
satisfying the following relations ($a,b,c\in R$):
\begin{description}
\item[\emph{(sesquilinearity)}]
\ \ \ \ \ \  $[\partial a_\lambda b] = -\lambda [a_\lambda b]\,,\,\,
[a_\lambda \partial b] = (\partial + \lambda) [a_\lambda b]$,
\item[\emph{(skewcommutativity)}] 
 $[a_\lambda b] = -[b_{-\lambda-\partial}a]$, where $\partial$ is moved to the left,
\item[\emph{(Jacobi identity)}] 
\ \ \ \ \ $[a_\lambda [b_\mu c]]-[b_\mu[a_\lambda c]] = [[a_\lambda b]_{\lambda + \mu}c]$.
\end{description}
\end{definition}
One writes 
$[a_\lambda  b] = \sum_{j\in \mb Z_+} \frac{\lambda^j}{j!} \left(a_{(j)}b\right)$,
where the sum is finite; the bilinear products $a_{(j)}b$ are called $j^{\texttt{th}}$-products.

\begin{definition} \label{lca-mod}
A \emph{module} $M$ over a Lie conformal algebra $R$ is an $\mb F[\partial]$-module,
with action of $\partial$ denoted by $\partial^M$, endowed with a 
$\lambda$-action $R\otimes M \to\mb F[\lambda]\otimes M$,
denoted $a\otimes m\mapsto a_\lambda m$, such that 
\begin{enumerate}[(i)]
\item $(\partial a)_\lambda m = -\lambda a_\lambda m\,,\,\, 
a_\lambda(\partial^M m) = (\partial^M + \lambda) (a_\lambda m)$.
\item $a_\lambda(b_\mu m) - b_\mu(a_\lambda m) = [a_\lambda b]_{\lambda + \mu}m$.
\end{enumerate}
\end{definition}
\begin{remark}\label{100310:rem}
If $R$ is a Lie conformal algebra, then the torsion $\Tor R$ 
of the $\mb F[\partial]$-module $R$ is in the center of the Lie conformal algebra $R$,
and moreover its $\lambda$-action on any $R$-module $M$ is trivial.
Indeed, for $a\in R$ and $P(\partial)\in\mb F[\partial]$, we have
$(P(\partial)a)_\lambda=P(-\lambda)a_\lambda$.
Hence, if $a\in\Tor R$, i.e.\ $P(\partial)a=0$ with $P\neq0$, we get that $a_\lambda=0$
on any $R$-module.
\end{remark}

Recall that the \emph{annihilation Lie algebra} associated to the Lie conformal algebra $R$ is
$$
\Lie{}_- R
=
R[[t]]/(\partial + \partial_t)R[[t]]\,,
$$ 
with the well defined Lie bracket 
\begin{equation} \label{commutator}
[a_m, b_n]
= \sum_{j\in\mb Z_+}\binom{m}{j}(a_{(j)}b)_{m+n-j}\, ,\,\, a,b\in R, \ m,n\in \mb Z_+ \,, 
\end{equation}
where $a_n,\,n\in\mb Z_+$, denotes the image of $at^n$ in $\Lie_-\!\! R$. 
Letting
$a_\lambda=\sum_{n\in\mb Z_+} \frac{\lambda^n}{n!} a_n$,
formula \eqref{commutator} is equivalent to $[a_\lambda,b_\mu]=[a_\lambda b]_{\lambda+\mu}$,
which is equivalent to the Jacobi identity.
Moreover, the identity $(\partial a)_n =\,-na_{n-1}$, which holds on $\Lie_- R$,
is equivalent to $(\partial a)_\lambda=-\lambda a_\lambda$, which is 
the first sesquilinearity relation.
Note also that, if $a\in\Tor R$, then $a_\lambda=0$.
This follows by the same argument as in Remark \ref{100310:rem}.

The action of $\partial$ on $R$ induces a derivation of the 
Lie algebra $\Lie_- R$, by
$\partial(a_n)\,=\,(\partial a)_n$,
which we still denote by $\partial$,
and we may consider the semidirect product $\Lie_- R\rtimes\mb F\partial$.
A module $M$ over the annihilation Lie algebra $\Lie_- R$,
or over the Lie algebra $\Lie_- R\rtimes\mb F\partial$,
is called \emph{conformal} if, 
for any $a\in R$ and $v\in M$, we have $a_n(v)=0$ for $n$ sufficiently large.

\begin{proposition}\label{100310:prop} 
A module over a Lie conformal algebra $R$ is the same as a conformal module over 
the Lie algebra $\Lie_- R\rtimes\mb F\partial$.
\end{proposition}
\begin{proof}
To give a vector space $M$ a structure of an $R$-module means to provide 
an operator $\partial^M$ on $M$ and, for each $a\in R$,
a sequence of operators $a_n, n\in\mb Z_+,$ on $M$, such that
the map $R\otimes M\to\mb F[\lambda]\otimes M$ given by
\begin{equation}\label{0606:eq1}
a_\lambda v = \sum_{n\in\mb Z_+} \frac{\lambda^n}{n!}\,a_n v, \ v\in M \,,
\end{equation}
satisfies relations $(i)$ and $(ii)$ in Definition \ref{lca-mod}.
But this is exactly the same as giving $M$ a structure of a conformal module 
over $\Lie_- R\rtimes\mb F\partial$, with $\partial$ acting as $\partial^M$.
\flushright\qed
\end{proof}

If the Lie conformal algebra $R$ is of finite rank as an $\mb F[\partial]$-module,
then the Lie algebra $\Lie_- R$ can be endowed with a linearly compact topology.
In fact, if we decompose $R=(\mb F[\partial]\otimes U)\oplus\Tor$,
where $U$ is a finite dimensional vector space,
then $\Lie_- R$ is (non canonically) isomorphic to $U[[t]]$ as a vector space,
and it has the usual formal power series topology,
which makes it a linearly compact Lie algebra.
In this case, to say that a $\Lie_-R$-module $M$ is conformal 
is the same as to say that
it is a continuous module, when we endow it with the discrete topology.

\subsection{The Lie conformal algebra cochain complex $(C^\bullet(R,M),d)$}
\label{sec:4.2}

For $k\in\mb Z_+$, denote by $\mb F_{-}[\lambda_1,\dots,\lambda_k]$
the space of polynomials in the $k$ variables $\lambda_1,\dots,\lambda_k$
with the $\mb F[\partial]$-module structure obtained by  
letting $\partial$ act by multiplication by $-(\lambda_1+\cdots+\lambda_k)$.

Let $R$ be a Lie conformal algebra and let $M$ be an $R$-module.
We define the space of $k$-cochains $C^k(R,M),\,k\in\mb Z_+$, as
the space of $\mb F$-linear maps
$$
\begin{array}{rcl}
c\,&:&\,\, R^{\otimes k}\,\,\,\longrightarrow\,\,\,\mb F_-[\lambda_1,\dots,\lambda_k]
\otimes_{\mb F[\partial]} M
\,,\\
&& a_1\otimes\cdots\otimes a_k\mapsto c_{\lambda_1,\cdots,\lambda_k}(a_1,\cdots,a_k)\,,
\end{array}
$$
satisfying the following conditions:
\begin{description}
\item[(sesquilinearity)] 
$c_{\lambda_1,\cdots,\lambda_k}(a_1,\cdots,\partial a_i,\cdots,a_k)
=-\lambda_ic_{\lambda_1,\cdots,\lambda_k}(a_1,\cdots,a_k)$,
\item[(skewsymmetry)]
$c_{\lambda_{\sigma(1)},\cdots,\lambda_{\sigma(k)}}(a_{\sigma(1)},\cdots,a_{\sigma(k)})
=\sign(\sigma)c_{\lambda_1,\cdots,\lambda_k}(a_1,\cdots,a_k)$,
for all permutations $\sigma\in S_k$.
\end{description}
For example, $C^0(R,M)=M/\partial M$ and $C^1(R,M)=\Hom_{\mb F[\partial]}(R,M)$.
We let $C^\bullet(R,M)$ be the $\mb Z_+$-graded vector superspace $\bigoplus_{k\in\mb Z_+}C^k(R,M)$,
with the parity induced by the $\mb Z_+$-grading.

We make $C^\bullet(R,M)$ into a $\mb Z_+$-graded complex by letting the differential
$d:\,C^k(R,M)\to C^{k+1}(R,M),\,k\in\mb Z_+$, defined by the following formula:
\begin{equation}\label{100311:eq1}
\begin{array}{c}
\displaystyle{
(dc)_{\lambda_1,\cdots,\lambda_{k+1}}(a_1,\cdots,a_{k+1})
= \sum_{i=1}^{k+1}(-1)^{i+1}
{a_i}_{\lambda_i}\big(
c_{\lambda_1,\stackrel{i}{\check{\cdots}},\lambda_{k+1}}(a_1,\stackrel{i}{\check{\cdots}},a_{k+1})
\big) 
}\\
\displaystyle{
+ \sum_{1\leq i<j\leq k+1}(-1)^{j+1}
c_{\lambda_1,\cdots \lambda_i+\lambda_j\stackrel{j}{\check{\cdots}},\lambda_{k+1}}
(a_1,\cdots[{a_i}_{\lambda_i}a_j]\stackrel{j}{\check{\cdots}},a_{k+1})\,.
}
\end{array}
\end{equation}
For example, if $\tint v\in M/\partial M=C^0(R,M)$, we have
$\big(d\tint v\big)_\lambda(a)=a_\lambda v= a_{-\partial^M}v
\in \mb F_-[\lambda]\otimes_{\mb F[\partial]}M$,
where $\partial^M$ is moved to the left.
Here, as usual in the variational calculus, $\tint v$ denotes the coset of $v$ in $M/\partial M$.
\begin{proposition}\label{100311:prop1}
Formula \eqref{100311:eq1} gives a well defined map $d:\,C^k(R,M)\to C^{k+1}(R,M)$,
and $d^2=0$.
Hence, $C^\bullet(R,M)$ is a $\mb Z_+$-graded complex.
\end{proposition}
\begin{proof}
This complex is a special case of a Lie pseudoalgebra complex, when the Hopf algebra is $\mb F[\partial]$.
Hence, this statement follows from the results in \cite[Sec.15.1]{BDAK}.
It also follows from \cite{DSK}, where the complex $C^\bullet(R,M)$ is introduced in terms
of poly-$\lambda$-brackets (see Remark \ref{100311:rem} below).
\flushright\qed
\end{proof}

\begin{remark}\label{100311:rem}
We can identify the space $\mb F_-[\lambda_1,\dots,\lambda_k]\otimes_{\mb F[\partial]}M$
with the space $\mb F[\lambda_1,\dots,\lambda_{k-1}]\otimes M$,
by substituting $\lambda_k$ by $\lambda_k^\dagger:=-\lambda_1-\cdots-\lambda_{k-1}-\partial^M$.
Although this identification is not canonical, it is often convenient in practical use,
via the language of poly-$\lambda$-brackets.
Namely, a $k$-cochain $c\in C^k(R,M)$ is described as a $k$-$\lambda$-\emph{bracket},
i.e. a map $c: R^{\otimes k}\to\mb F[\lambda_1,\dots,\lambda_{k-1}]\otimes M,\,
a_1\otimes\cdots\otimes a_k\mapsto\{{a_1}_{\lambda_1}\cdots{a_{k-1}}_{\lambda_{k-1}}a_k\}_c$,
satisfying
\begin{enumerate}
\item[$(i)$]
$\{{a_1}_{\lambda_1}\dots (\partial a_i)_{\lambda_i}\dots {a_{k-1}}_{\lambda_{k-1}}a_k \}_c=
-\lambda_i \{{a_1}_{\lambda_1}\cdots{a_{k-1}}_{\lambda_{k-1}}a_k\}_c$
for $i\leq k-1$,
and 
$\{{a_1}_{\lambda_1}\dots {a_{k-1}}_{\lambda_{k-1}}\partial a_k \}_c=
-\lambda_k^\dagger \{{a_1}_{\lambda_1}\cdots{a_{k-1}}_{\lambda_{k-1}}a_k\}_c$;
\item[$(ii)$] 
$\{{a_{\sigma(1)}}_{\lambda_{\sigma(1)}}\cdots{a_{\sigma(k-1)}}_{\lambda_{\sigma(k-1)}}a_{\sigma(k)}\}_c
=
\sign(\sigma) \{{a_1}_{\lambda_1}\cdots{a_{k-1}}_{\lambda_{k-1}}a_k\}_c$, $\sigma\in S_k$,
where $\lambda_k$ is replaced by $\lambda_k^\dagger$ whenever it appears.
\end{enumerate}
In the language of poly-$\lambda$-brackets the differential $d:\, C^k(R,M)\to C^{k+1}(R,M)$, 
for $k\geq1$, takes the form
$$
\begin{array}{c}
\displaystyle{
\{{a_1}_{\lambda_1}\dots {a_k}_{\lambda_k}a_{k+1} \}_{dc} 
= \sum_{i=1}^{k+1}(-1)^{i+1} 
{a_i}_{\lambda_i}\big\{{a_1}_{\lambda_1}\stackrel{i}{\check{\cdots}} {a_k}_{\lambda_k}a_{k+1} \big\}_c 
}\\
\displaystyle{
+\sum_{1\leq i<j\leq k+1} (-1)^{j+1}
\big\{{a_1}_{\lambda_1}\cdots[a_{i\lambda_i}a_j]_{\lambda_i+\lambda_j}
\stackrel{j}{\check{\cdots}} {a_k}_{\lambda_{k}}a_{k+1} \big\}_c\, ,
}
\end{array}
$$
where, as before, we replace $\lambda_{k+1}$ by $\lambda_{k+1}^\dagger$ whenever it appears.
For example, if $c\in C^1(R,M)=\Hom_{\mb F[\partial]}(R,M)$,
then $\{a_\lambda b\}_{dc}=a_\lambda c(b)-b_{-\lambda-\partial^M}c(a)-c([a_\lambda b])$.
\end{remark}

\begin{remark}\label{0322:rem1}
Following \cite{DSK}, define the subcomplex $\bar C^\bullet(R,M)\subset C^\bullet(R,M)$,
consisting of the $\mb F$-linear maps 
$c:\,R^{\otimes k}\to\mb F_-[\lambda_1,\cdots,\lambda_k]\otimes_{\mb F[\partial]}M$,
such that $c_{\lambda_1,\cdots,\lambda_k}(a_1,\cdots,a_k)=0$
if one of the entries $a_i$ is a torsion element in the $\mb F[\partial]$-module $R$.
Then $\bar C^k(R,M)=C^k(R,M)$ if $k\neq1$,
while $\bar C^1(R,M)=\big\{c\in C^1(R,M)=\Hom_{\mb F[\partial]}(R,M)\,|\,c(\Tor(R))=0\big\}$.
For example, if $R$ is of finite rank as an $\mb F[\partial]$-module,
and it decomposes as 
\begin{equation}\label{0322:eq3}
R=\Tor(R)\oplus(\mb F[\partial]\otimes U)
\,\,,\,\,\,\,
\Tor(R)=\bigoplus_i\mb F[\partial]/(P_i(\partial))\,,
\end{equation}
then $\bar C^1(R,M)=\Hom_{\mb F}(U,M)$,
while $C^1(R,M)=\big(\bigoplus_i\ker(P_i(\partial):\,M\to M)\big)\oplus\Hom_{\mb F}(U,M)$.
\end{remark}

\subsection{The space of chains $C_\bullet(R,M)$ 
and its Gerstenhaber algebra structure}
\label{sec:4.3}

Let $R$ be a Lie conformal algebra and let $M$ be a module over $R$
endowed with the structure of a commutative associative algebra
on which $\partial^M$ and $a_\lambda,\,a\in R$, act by derivations.

For $k\in\mb Z_+$, consider the algebra of formal power series $M[[x_1,\dots,x_k]]$ 
with coefficients in $M$. It is endowed with an $\mb F[\partial]$-module structure,
with $\partial$ acting as $\partial^M$ on coefficients,
and with a $\lambda$-action of $R$ on $M[[x_1,\dots,x_k]]$,
where $a_\lambda$ acts on the coefficients of the formal power series.
Note that this is not an $R$-module structure of $R$ on $M[[x_1,\dots,x_k]]$,
since $a_\lambda\phi$ can be a formal power series in $\lambda$,
but it satisfies all other axioms of a module over a Lie conformal algebra.

Let $\mc M_k\subset M[[x_1,\dots,x_k]]$ be the subspace
of series $\phi(x_1,\dots,x_k)$ such that
\begin{equation}\label{0318:eq1}
\big(\partial_{x_1}+\cdots+\partial_{x_k}\big)\phi(x_1,\dots,x_k)
=\partial^M \phi(x_1,\dots,x_k)\,.
\end{equation}
For example, $\mc M_0=\{m\in M\,|\,\partial^Mm=0\}=M^\partial$,
and $\mc M_1=\{e^{x\partial^M}m\,|\,m\in M\}$, which is naturally identified with $M$.
Note that the subspace $\mc M_k$ is not an $R$-submodule.
We have a natural action of the group of permutations $S_k$ on $M[[x_1,\dots,x_k]]$,
which leaves $\mc M_k$ invariant, given by ($\sigma\in S_k$):
\begin{equation}\label{0318:eq2}
(\sigma\phi)(x_1,\dots,x_k)=\phi(x_{\sigma(1)},\dots,x_{\sigma(k)})\,.
\end{equation}
Also, recall that we have a natural action of $S_k$ on the space $R^{\otimes k}$
given by
\begin{equation}\label{0318:eq3}
\sigma(a_1\otimes\dots\otimes a_k)=a_{\sigma^{-1}(1)}\otimes\dots a_{\sigma^{-1}(k)}\,.
\end{equation}

The space of $k$-\emph{chains} $C_k(R,M)$
is defined as the quotient of the space $R^{\otimes k}\otimes\mc M_k$
by the following relations:
\begin{description}
\item[(sesquilinearity)]
$a_1\otimes \dots \partial a_i \dots \otimes a_k \otimes \phi 
+ a_1\otimes \dots \otimes a_k \otimes (\partial_{x_i} \phi)=0,\, i=1,\dots,k$,
\item[(skewsymmetry)]
$\sigma(a_1\otimes\cdots\otimes a_k\otimes\phi)=\sign(\sigma)a_1\otimes\cdots\otimes a_k\otimes\phi
,\,\sigma \in S_k$.
\end{description}
For example, $C_0(R,M)=\mc M_0= M^\partial$
and, using the identification $\mc M_1=M$ described above,
we have an identification
$C_1(R,M)=(R\otimes M)/\partial(R\otimes M)$.
We let $ C_\bullet=C_\bullet(A,M) = \oplus_{k\in \mb Z_+} C_k(A,M)$ be the $\mb Z_+$-graded
superspace of all chains, with parity induced by the $\mb Z_+$-grading.

The wedge product on $C_\bullet(R,M)$,
which makes it a commutative associative $\mb Z_+$-graded superalgebra, is given by
\begin{equation}\label{0312:eq1}
\begin{array}{c}
\big(a_1\otimes\cdots\otimes a_h\otimes\phi(x_1,\dots,x_h)\big)
\wedge\big(a_{h+1}\otimes\cdots\otimes a_k\otimes\psi(x_1,\dots,x_{k-h})\big) \\
=a_1\otimes\cdots\otimes a_k\otimes\big(\phi(x_1,\dots,x_h)\psi(x_{h+1},\dots,x_k)\big)\,.
\end{array}
\end{equation}
It is immediate to check  that the wedge product is well defined
and it is associative and commutative (in the super sense).
For commutativity we use the skewsymmetry relation in $C_k(R,M)$ where $\sigma$ 
is the permutation $\tau_{h,k-h}\in S_k$ exchanging the first $k-h$ letters
with the last $h$ letters:
\begin{equation}\label{0317:eq0}
\tau_{h,k-h}=
\left(\begin{array}{cccccc}
1 & \cdots & k-h & k-h+1 & \cdots & k \\
h+1 & \cdots & k & 1 & \cdots & h
\end{array}\right)\,.
\end{equation}

To have a Gerstenhaber algebra structure on $C_\bullet(R,M)$ 
we are left to introduce a Lie superalgebra bracket on $\Pi C_\bullet(R,M)$.
In \cite{DSK}, it was shown that 
$\Pi C_1 = (R\otimes M)/\partial(R\otimes M)$ carries a Lie algebra bracket given by the formula:
\begin{equation}\label{0312:eq7}
[a\otimes m, b\otimes n] = [a_{\partial^M_1}b]_{\to}\otimes mn + b\otimes (a_{\partial^M}n)_\to m 
- a\otimes (b_{\partial^M}m)_\to n \ ,
\end{equation}
where, $\partial^M_1$ in the first summand denotes $\partial^M$ acting 
only on the first factor $m$ 
and the right arrow means that $\partial^M$ should be moved to the right. 
Motivated by formula \eqref{0312:eq7} and by the definition of the Schouten bracket, we introduce
the following bracket on $C_\bullet(R,M)$:
\begin{equation}\label{0312:eq12}
\begin{array}{l}
\displaystyle{
\big[a_1\otimes\cdots\otimes a_h\otimes\phi(x_1,\dots,x_h)
\,,\,a_{h+1}\otimes\cdots\otimes a_k
\otimes\psi(x_1,\dots,x_{k-h})\big] 
}\\
\\
\displaystyle{
=
\sum_{i=1}^h\sum_{j=h+1}^k (-1)^{h+j+1} 
\big(
a_1\otimes\cdots[{a_i}\, _{\partial_{x_i}}a_j]
\stackrel{j}{\check{\cdots}}\otimes a_k
}\\
\,\,\,\,\,\,\,\,\,\,\,\,\,\,\,\,\,\,\,\,\,\,\,\,\,\,\,
\,\,\,\,\,\,\,\,\,\,\,\,\,\,\,\,\,\,\,\,\,\,\,\,\,\,\,
\displaystyle{
\otimes\phi(x_1,\dots,x_h)\psi(x_{h+1},\dots\stackrel{j}{y}\dots,x_{k-1})\,
\big) \Big|_{y=x_i}
}\\
\\
\displaystyle{
+ \sum_{i=1}^h (-1)^{h+i} a_1\otimes\stackrel{i}{\check{\cdots}}\otimes a_k
\otimes \big( {a_i}\,_{\partial_y}\psi(x_h,\dots,x_{k-1}) \big) 
\phi(x_1,\dots\stackrel{i}{y}\dots,x_{h-1}) \,\Big|_{y=0}
} \\
\\
\displaystyle{
+\! \sum_{i=h+1}^k (-1)^{h+i} a_1\otimes\stackrel{i}{\check{\cdots}}\otimes a_k
\otimes \big( {a_i}\,_{\partial_y}\phi(x_1,\dots,x_h) \big) 
\psi(x_{h+1},\dots\stackrel{i}{y}\dots,x_{k-1}) \,\Big|_{y=0} .
}
\end{array}
\end{equation}
Here and further, by writing $\stackrel{j}{y}$, we mean that $y$ is in place of the variable $x_j$,
and the remaining variables $x_\ell,\,\ell\geq j$, are shifted to the right.
For example, in the first term of the RHS above, we have
$\psi(x_{h+1},\dots\stackrel{j}{y}\dots,x_{k-1})
=\psi(x_{h+1},\dots,x_{j-1},y,x_{j},\dots,x_{k-1})$.
It is also clear that the bracket \eqref{0312:eq12} coincides with \eqref{0312:eq7} 
when $h=k-h=1$.
\begin{theorem}\label{0317:th}
Formulas \eqref{0312:eq1} for the wedge product  and \eqref{0312:eq12} for the bracket
define a structure of a $\mb Z_+$-graded Gerstenhaber algebra on the space of chains $C_\bullet(R,M)$.
\end{theorem}
\begin{proof}
First, we need to prove that the bracket in \eqref{0312:eq12} is well defined.
It is straightforward to check that each term in the RHS of \eqref{0312:eq12}
lies in $R^{\otimes {k-1}}\otimes\mc M_{k-1}$.
For the second and third term, one needs to use sesquilinearity of the $\lambda$-action
of $R$ on $M$.
It is not hard to check that the sesquilinearity relations defining $C_k(R,M)$
are preserved by the bracket \eqref{0312:eq12},
using the sesquilinearity of the $\lambda$-bracket on $R$ and of the $\lambda$-action of $R$ on $M$.
For the skewsymmetry relation, we have, for $\sigma\in S_h$,
\begin{equation}\label{0317:eq3}
\begin{array}{l}
\displaystyle{
\big[\sigma\big(a_{1}\otimes\cdots\otimes a_{h}\otimes\phi\big)
\,,\,
a_{h+1}\otimes\cdots\otimes a_k\otimes\psi\big] 
}\\
\displaystyle{
=
\sum_{i=1}^h\sum_{j=h+1}^k (-1)^{h+j+1} 
\Big(
a_{\sigma^{-1}(1)}\otimes\cdots[{a_{\sigma^{-1}(i)}}_{\partial_{x_i}}a_j]
\cdots a_{\sigma^{-1}(h)}\otimes a_{h+1}\stackrel{j}{\check{\cdots}}\otimes a_k
}\\
\,\,\,\,\,\,\,\,\,\,\,\,\,\,\,\,\,\,\,\,\,\,\,\,\,\,\,\,\,\,\,\,\,\,\,\,\,\,\,\,\,\,\,\,\,
\displaystyle{
\otimes
\phi(x_{\sigma(1)},\cdots,x_{\sigma(h)})\psi(x_{h+1},\cdots\stackrel{j}{y}\cdots,x_{k-1})
\Big)\,\Big|_{y=x_i}
}\\
\displaystyle{
+ \sum_{i=1}^h (-1)^{h+i} 
a_{\sigma^{-1}(1)}\otimes\stackrel{i}{\check{\cdots}}a_{\sigma^{-1}(h)}\otimes a_{h+1}
\cdots\otimes a_k
} \\
\,\,\,\,\,\,\,\,\,\,\,\,\,\,\,\,\,\,\,\,\,\,\,\,\,\,\,\,\,\,\,\,\,\,\,\,\,\,\,\,\,\,\,\,\,
\displaystyle{
\otimes
\big({a_{\sigma^{-1}(i)}}_{\partial_y}\psi(x_{h},\cdots,x_{k-1})\big)
(\sigma\phi)(x_1,\cdots\stackrel{i}{y}\cdots,x_{h-1})
\,\Big|_{y=0}
} \\
\displaystyle{
+ \sum_{i=h+1}^k (-1)^{h+i} 
a_{\sigma^{-1}(1)}\otimes\cdots a_{\sigma^{-1}(h)}\otimes a_{h+1}
\stackrel{i}{\check{\cdots}}\otimes a_k
} \\
\,\,\,\,\,\,\,\,\,\,\,\,\,\,\,\,\,\,\,\,\,\,\,\,\,\,\,\,\,\,\,\,\,\,\,\,\,\,\,\,\,\,\,\,\,
\displaystyle{
\otimes
\big({a_i}_{\partial_y}\phi(x_{\sigma(1)},\cdots,x_{\sigma(h)})\big)
\psi(x_{h+1},\cdots\stackrel{i}{y}\cdots,x_{k-1})
\,\Big|_{y=0} \,.
}
\end{array}
\end{equation}
It is not hard to check that the first term in the RHS of \eqref{0317:eq3} is equal to
$$
\begin{array}{l}
\displaystyle{
\sigma\bigg(
\sum_{\ell=1}^h\sum_{j=h+1}^k (-1)^{h+j+1} 
a_{1}\otimes\cdots[{a_{\ell}}_{\partial_{x_{\ell}}}a_j]
\cdots a_{h}\otimes a_{h+1}\stackrel{j}{\check{\cdots}}\otimes a_k
}\\
\,\,\,\,\,\,\,\,\,\,\,\,\,\,\,\,\,\,\,\,\,\,\,\,\,\,\,\,\,\,
\displaystyle{
\otimes
\phi(x_{1},\cdots,x_{h})\psi(x_{h+1},\cdots\stackrel{j}{y}\cdots,x_{k-1})
\,\Big|_{y=x_{\ell}}\bigg)\,,
}
\end{array}
$$
where we made the change of the summation index $\ell=\sigma^{-1}(i)$.
Here we consider $\sigma\in S_h$ as an element of $S_{k-1}$ via the obvious
embedding $S_h\subset S_{k-1}$.
Likewise, the last term in the RHS of \eqref{0317:eq3} is equal to
$$
\begin{array}{l}
\displaystyle{
\sigma\bigg(
\sum_{i=h+1}^k (-1)^{h+i} 
a_{1}\otimes\cdots a_{h}\otimes a_{h+1}
\stackrel{i}{\check{\cdots}}\otimes a_k
} \\
\,\,\,\,\,\,\,\,\,\,\,\,\,\,\,\,\,\,\,\,\,\,\,\,\,\,\,
\displaystyle{
\otimes
\big({a_i}_{\partial_y}\phi(x_1,\cdots,x_h)\big)
\psi(x_{h+1},\cdots\stackrel{i}{y}\cdots,x_{k-1})
\,\Big|_{y=0}\bigg) \,.
}
\end{array}
$$
We are left to consider the second term in the RHS of \eqref{0317:eq3}.
Given a permutation $\sigma\in S_h$ and $i=1,\dots,h$, we define the permutation $\sigma_i\in S_{h-1}$
as follows:
$$
\sigma_i:\,\{1,\dots,h-1\}
\stackrel{\sim}{\longrightarrow}\{1,\stackrel{\sigma^{-1}(i)}{\check{\cdots}},h\}
\stackrel{\sigma}{\longrightarrow}\{1,\stackrel{i}{\check{\cdots}},h\}
\stackrel{\sim}{\longrightarrow}\{1,\dots,h-1\}\,,
$$
where the first map is the shift to the right by 1 of indices greater than or equal to $\sigma^{-1}(i)$,
and the last map is the shift to the left by 1 of indices greater than $i$.
Clearly, the element $a_{\sigma^{-1}(1)}\otimes\stackrel{i}{\check{\cdots}}a_{\sigma^{-1}(h)}
\otimes a_{h+1}\cdots\otimes a_k\in R^{k-1}$
is obtained by applying the permutation $\sigma_i\in S_{h-1}\subset S_{k-1}$ to
the element $a_{1}\otimes\stackrel{\sigma^{-1}(i)}{\check{\cdots}}\otimes a_k$.
Moreover, we have the obvious identity
$$
\sigma_i\phi(x_1,\dots\stackrel{\sigma^{-1}(i)}{y}\dots,x_{h-1})
=(\sigma\phi)(x_1,\dots\stackrel{i}{y}\dots,x_{h-1}))\,,
$$
where in the LHS $\sigma_i$ permutes only the variables $x_1,\dots,x_{h-1}$.
These two facts together allow us to rewrite the second term in the RHS of \eqref{0317:eq3} as
$$
\begin{array}{l}
\displaystyle{
\sum_{i=1}^h (-1)^{h+i} \sigma_i\bigg(
a_{1}\otimes\stackrel{i}{\check{\cdots}}\otimes a_k
} \\
\,\,\,\,\,\,\,\,\,\,\,\,\,\,\,\,\,\,\,\,\,\,\,\,
\displaystyle{
\otimes
\big({a_{\sigma^{-1}(i)}}_{\partial_y}\psi(x_{h},\cdots,x_{k-1})\big)
\phi(x_1,\cdots\stackrel{\sigma^{-1}(i)}{y}\cdots,x_{h-1})
\,\Big|_{y=0}\bigg)\,.
}
\end{array}
$$
Combining the above results, and using the sign identity
\begin{equation}\label{0317:eq6}
\sign(\sigma_i)=(-1)^{i+\sigma(i)}\sign(\sigma)\,,
\end{equation}
we conclude that the RHS of \eqref{0317:eq3} is equal, modulo the skewsymmetry relation in $C_{k-1}(R,M)$,
to
$$
\sign(\sigma)[a_1\otimes\cdots\otimes a_h\otimes\phi,a_{h+1}\otimes\cdots\otimes a_k\otimes\psi] \,,
$$
as required.
The fact that the skewsymmetry relations in the right factor is also preserved
can be proved similarly. 
In fact, this will follow from the skewcommutativity of the bracket.
This concludes the proof that the bracket \eqref{0312:eq12} is well defined.

Next, we prove that the bracket \eqref{0312:eq12} is skewcommutative.
We have, after some change in the summation indices,
\begin{equation}\label{0317:eq8}
\begin{array}{l}
\displaystyle{
[a_{h+1}\otimes\cdots\otimes a_k\otimes\psi,a_1\otimes\cdots\otimes a_h\otimes\phi] 
}\\
\displaystyle{
\,\,\,\,\,\,
=
\sum_{i=1}^h \sum_{j=h+1}^k (-1)^{i+1} 
\Big(
a_{h+1}\otimes\cdots[{a_j}_{\partial_{x_{j-h}}}a_i]
\cdots a_k\otimes a_1
\stackrel{i}{\check{\cdots}}\otimes a_h
}\\
\,\,\,\,\,\,\,\,\,\,\,\,\,\,\,\,\,\,
\displaystyle{
\otimes \psi(x_1,\dots,x_{k-h})\phi(x_{k-h+1},\dots\stackrel{k-h+i}{y}\dots,x_{k-1})
\Big)\,\Big|_{y=x_{j-h}}
}\\
\,\,\,\,\,\,
\displaystyle{
+ \sum_{j=h+1}^k (-1)^{k+j} 
a_{h+1}\otimes\stackrel{j}{\check{\cdots}}a_k\otimes a_1\cdots\otimes a_h
}\\
\,\,\,\,\,\,\,\,\,\,\,\,\,\,\,\,\,\,
\displaystyle{
\otimes \big({a_j}_{\partial_y}\phi(x_{k-h},\dots,x_{k-1})\big)
\psi(x_1,\dots\stackrel{j-h}{y}\dots,x_{k-h})\,\Big|_{y=0}
} \\
\,\,\,\,\,\,
\displaystyle{
+ \sum_{i=1}^h (-1)^{i} 
a_{h+1}\otimes\cdots a_k\otimes a_1\stackrel{i}{\check{\cdots}}\otimes a_h
}\\
\,\,\,\,\,\,\,\,\,\,\,\,\,\,\,\,\,\,
\displaystyle{
\otimes \big({a_j}_{\partial_y}\psi(x_{1},\dots,x_{k-h})\big)
\phi(x_{k-h+1},\dots\stackrel{k-h+i}{y}\dots,x_{k-1})\,\Big|_{y=0}\,.
}
\end{array}
\end{equation}
Let us consider the first term in the RHS of \eqref{0317:eq8}.
Skewcommutativity of the $\lambda$-bracket in $R$ gives
$[{a_j}_{\partial_{x_{j-h}}}a_i]=-[{a_i}_{-\partial_{x_{j-h}}-\partial}{a_j}]$.
Combining this with the sesquilinearity relation in $C_{k-1}(R,M)$,
we conclude that we can replace $[{a_j}_{\partial_{x_{j-h}}}a_i]$ by $-[{a_i}_{\partial_y}{a_j}]$.
We then observe that
$$
a_{h+1}\otimes\cdots\stackrel{j}{[{a_i}_{\partial_y}{a_j}]}\cdots a_k
\otimes a_1\stackrel{i}{\check{\cdots}}\otimes a_h
=\sigma\big(
a_{1}\otimes\cdots\stackrel{i}{[{a_i}_{\partial_y}{a_j}]}\stackrel{j}{\check{\cdots}} \otimes a_k
\big)\,,
$$
where $\sigma\in S_{k-1}$ is the following permutation:
$$
\sigma(\ell)=
\left\{\begin{array}{lcl}
k-h+\ell & \text{ if } & \ell=1,\dots,i-1 \\
j-h & \text{ if } & \ell=i \\
k-1-h+\ell & \text{ if } & \ell=i+1,\dots,h \\
\ell-h & \text{ if } & \ell=h+1,\dots,j-1 \\
\ell-h+1 & \text{ if } & \ell=j,\dots,k-1 
\end{array}\right.\,.
$$
Moreover, we have
$$
\begin{array}{l}
\partial_y^n \psi(x_1,\dots,x_{k-h})\phi(x_{k-h+1},\dots\stackrel{k-h+i}{y}\dots,x_{k-1})\,\Big|_{y=x_{j-h}} \\
=
\sigma\Big(
\partial_{x_i}^n\phi(x_1,\dots,x_h)\psi(x_{h+1},\dots\stackrel{j}{y}\dots,x_{k-1})\,\Big|_{y=x_i}
\Big)\,.
\end{array}
$$
Combining the above results, we conclude that the first term in the RHS of \eqref{0317:eq8}
is equal to
$$
\begin{array}{l}
\displaystyle{
- \sum_{i=1}^h \sum_{j=h+1}^k (-1)^{i+1} \sigma\Big(
a_{1}\otimes\cdots[{a_i}_{\partial_{x_i}}{a_j}]\stackrel{j}{\check{\cdots}} \otimes a_k
}\\
\,\,\,\,\,\,\,\,\,\,\,\,\,\,\,\,\,\,
\displaystyle{
\phi(x_1,\dots,x_h)\psi(x_{h+1},\dots\stackrel{j}{y}\dots,x_{k-1})\,\Big|_{y=x_i}
\Big)
}
\end{array}
$$
Since $\sign(\sigma)=(-1)^{(h+1)(k-h+1)+h+j+i}$,
the above expression is the same,
modulo the skewsymmetry relation in $C_{k-1}(R,M)$,
as the first term in the RHS of \eqref{0312:eq12}
multiplied by $(-1)^{1+(h+1)(k-h+1)}$.
Next, consider the second term in the RHS of \eqref{0317:eq8}.
We have that
$$
a_{h+1}\otimes\stackrel{j}{\check{\cdots}}a_k\otimes a_1\cdots\otimes a_h
=
\tau_{k-h-1,h}\Big(a_1\otimes\stackrel{j}{\check{\cdots}}\otimes a_k\Big)\,,
$$
where the permutation $\tau_{k-h-1,h}\in S_{k-1}$ is defined in \eqref{0317:eq0}.
Moreover,
$$
\begin{array}{l}
\big({a_j}_{\partial_y}\phi(x_{k-h},\dots,x_{k-1})\big)
\psi(x_1,\dots\stackrel{j-h}{y}\dots,x_{k-h})\,\Big|_{y=0} \\
=
\tau_{k-h-1,h}\Big(
\big({a_j}_{\partial_y}\phi(x_1,\dots,x_h)\big)
\psi(x_{h+1},\dots\stackrel{j}{y}\dots,x_{k-1})\,\Big|_{y=0}\Big)\,.
\end{array}
$$
Combining the above results, we conclude that the second term in the RHS of \eqref{0317:eq8}
is equal to
$$
\begin{array}{l}
\displaystyle{
\sum_{j=h+1}^k (-1)^{k+j} \tau_{k-h-1,h}\Big(
a_1\otimes\stackrel{j}{\check{\cdots}}\otimes a_k
}\\
\,\,\,\,\,\,\,\,\,\,\,\,\,\,\,\,\,\,
\displaystyle{
\otimes 
\big({a_j}_{\partial_y}\phi(x_1,\dots,x_h)\big)
\psi(x_{h+1},\dots\stackrel{j}{y}\dots,x_{k-1})\,\Big|_{y=0}
\Big)
\,.
} 
\end{array}
$$
Clearly, $\sign(\tau_{k-h-1,h})=h(k-h+1)$.
Hence, due to the skewsymmetry relation in $C_{k-1}(R,M)$,
the above expression is the same as the third term
in the RHS of \eqref{0312:eq12} multiplied by $(-1)^{1+(h+1)(k-h+1)}$.
The third term in the RHS of \eqref{0317:eq8} is similar.
This proves that the bracket \eqref{0312:eq12} is skewcommutative.

We are left to prove the Jacobi identity and the odd Leibniz rule. These identities can be proven 
by a direct lengthy calculation, but instead we provide a short proof in the case when the Lie conformal algebra $R$ is 
a direct sum of a free $\mb F[\partial]$-submodule and torsion. 

First, we note that there is a natural Gerstenhaber algebra structure on the space of \emph{basic} chains $\tilde C_\bullet(R,M)$ (see Remark \ref{0324:rem1} below), defined by extending the Lie bracket on 1-chains $\Pi \tilde C_1(R,M)$ introduced in \cite{DSK},
$$ [a\otimes \phi(x),b\otimes \psi(x)] = [a_{\partial_{x_1}}b]\otimes \phi(x_1)\psi(x)|_{x_1=x}$$
$$- a\otimes (b_{\partial_y}\phi(x))\psi(y)|_{y=0}+ b\otimes (a_{\partial_y}\psi(x))\phi(y)|_{y=0} \,,$$
to all higher degree chains using the Leibniz rule \eqref{Leibniz} and \eqref{0312:eq1}. The resulting bracket coincides with \eqref{0312:eq12} and satisfies the Jacobi identity and the left Leibniz rule by construction. Moreover, the space $\tilde C_\bullet$ carries a $\mb Z$-graded $\mb F[\partial]$-module structure and the subspace of $\partial$-invariant chains $\tilde C_\bullet^\partial$, called the \emph{reduced} chain space, is a kernel of $\tilde C_\bullet$ and hence a Gerstenhaber subalgebra.

When $R$ is a direct sum of a free $\mb F[\partial]$-submodule and torsion, 
there is a bijection, established in 
Proposition 3.12 of \cite{DSK}, between the reduced chain space 
$\tilde C_\bullet^\partial$ and the quotient space $\bar C_\bullet$ of the 
chain space $C_\bullet$ by the 
subspace $T_1$ of $C_1$ (see Remark \ref{0322:rem2} below). 
This endows $\bar C_\bullet$ with the structure of a Gerstenhaber algebra,
in particular, $\Pi \bar C_\bullet$ is a Lie superalgebra. 
For the subalgebra $\Pi C_1$ of $\Pi C_\bullet$ the Jacobi identity was 
proven in Section 3.8 in \cite{DSK}. Next, recall that the Lie conformal 
algebra $R$ acts trivially on the torsion of the $\mb F[\partial]$-module 
$M$ \cite{K}, in particular on $M^\partial$. It follows that $\Pi C_0$ is 
in the center of the Lie algebra $\Pi C_\bullet$. Since $[C_i,C_j]$ lies 
in $C_{i+j-1}$, it follows that the Jacobi identity holds for $\Pi C_\bullet$.

Likewise, it suffices to check the Leibniz rule \eqref{Leibniz} for 
$X=a_1\otimes \phi(x_1)\in C_1$, $Y=m\in C_0$ and 
$Z=a_2\otimes a_3\otimes \psi(x_1,x_2)\in C_2$. We have
$$[a_1\otimes \phi(x_1),m\wedge a_2\otimes a_3\otimes \psi(x_1,x_2)]=
[{a_1}\,_{\partial_{x_1}}a_2]\otimes a_3\otimes \phi(x_1)m\psi(y,x_2)|_{y=x_1}$$
$$-[{a_1}\,_{\partial_{x_1}}a_3]\otimes a_2\otimes \phi(x_1)m\psi(x_1,y)|_{y=x_1}
+a_2\otimes a_3\otimes\left({a_1}\,_{\partial_y} (m\psi(x_1,x_2))\right)\phi(y)|_{y=0}$$
$$-a_1\otimes a_3\otimes ({a_2}\,_{\partial_y}\phi(x_1))m\psi(y,x_2)|_{y=0}+
a_1\otimes a_2\otimes ({a_3}\,_{\partial_y}\phi(x_1))m\psi(x_2,y)|_{y=0} \ .$$

The Leibniz rule follows immediately by expanding the third term on the RHS, 
using the fact that $a_\lambda$ acts by derivations on the commutative associative product in 
$M$. This completes the proof of the theorem.

\flushright\qed
\end{proof}

\begin{remark}\label{0322:rem2}
Consider the subspace $T_k$ of $C_k(R,M)$ spanned by elements
$a_1\otimes\cdots\otimes a_k\otimes\phi(x_1,\dots,x_k)$
such that one of the entries $a_i$ is a torsion element
of the $\mb F[\partial]$-module $R$.
Clearly, by \eqref{0312:eq1} and \eqref{0312:eq12}, $T_\bullet=\bigoplus_{k\in\mb Z_+} T_k$
is an abelian ideal of the Gerstenhaber algebra $C_\bullet(R,M)$.
Let $\bar C_\bullet(R,M)=\bigoplus_{k\in\mb Z_+}\bar C_k(R,M)$ be the corresponding Gerstenhaber 
factor algebra.
It is easy to see \cite{DSK} that $T_k=0$ if $k\neq1$, hence $\bar C_k(R,M)=C_k(R,M)$ for $k\neq1$,
while $C^1(R,M)=R\otimes M/\partial(R\otimes M)$ 
and $\bar C^1(R,M)=R\otimes M/(\Tor(R)\otimes M+\partial(R\otimes M))$.
If $R$ decomposes as in \eqref{0322:eq3},
then $\bar C_1(R,M)=U\otimes M$,
while $C_1(R,M)=\big(\bigoplus_i M/(P_i(\partial)M)\big)\oplus(U\otimes M)$.
\end{remark}

\begin{remark}\label{0326:rem}
We can define a differential $d$ on $C_\bullet(R,M)$ dual to the one on the cochain complex 
only when the $\lambda$-action of $R$ on $M$ is trivial. It is given by the following formula:
$$
\begin{array}{l}
d\big(a_1\otimes\cdots\otimes a_k\otimes\phi(x_1,\dots,x_k)\big) \\
\displaystyle{
=
\sum_{1\leq i<j\leq k} (-1)^{j+1}
a_1\otimes\cdots[{a_i}_{\partial_{x_i}}{a_j}]\stackrel{j}{\check{\cdots}}\otimes a_k\otimes
\phi(x_1,\dots\stackrel{j}{y}\cdots,x_{k-1})\,\big|_{y=x_i}\,.
}
\end{array}
$$
If we try to dualize the differential on $C^\bullet(R,M)$ even when the action is not trivial,
we would have to add the following term \cite{BKV}:
$$
\sum_{i=1}^k (-1)^{i+1}
a_1\otimes\stackrel{i}{\check{\cdots}}\otimes a_k\otimes
\big({a_i}_{\partial_{y}}\phi(x_1,\dots\stackrel{i}{y}\cdots,x_{k-1})\,\big|_{y=0}\,,
$$
but this is a divergent sum, since the $\lambda$-action of $R$ on $\mc M_k$
is not polynomial.
\end{remark}

\subsection{Calculus structure $(C_\bullet(R,M),C^\bullet(R,M))$}
\label{sec:4.4}

Given an $h$-chain $X=a_1\otimes\cdots\otimes a_h\otimes\phi(x_1,\dots,x_h)\in C_h(R,M)$
and a $k$-cochain $c=c_{\lambda_1,\cdots,\lambda_k}(a_1,\dots,a_k)\in C^k(R,M)$,
we define the \emph{contraction} of $c$ by $X$,
denoted $\iota_X(c)$, as the following element of $C^{k-h}(R,M)$:
\begin{equation}\label{0319:eq1}
\begin{array}{l}
(\iota_Xc)_{\lambda_{h+1},\cdots,\lambda_k}(a_{h+1},\cdots,a_k) \\
=
(-1)^{\frac{h(h-1)}2} \phi(\partial_{\lambda_1},\dots,\partial_{\lambda_h})
c_{\lambda_1,\cdots,\lambda_k}(a_1,\dots,a_k)\,\Big|_{\lambda_1=\cdots=\lambda_h=0}\,.
\end{array}
\end{equation}
Note that, for $h=k$, we need to take the integral of the RHS, since $C^0(R,M)=M/\partial M$.
It is proved in \cite[Lemma 7]{DSK}, in the poly-$\lambda$-bracket notation,
that the contraction $\iota_X(c)$ is well defined
and it lies in $C^{k-h}(R,M)$.
As usual, we define the \emph{Lie derivative} $L_X(c)\in C^{k-h+1}(R,M)$ by Cartan's formula \eqref{comp3}.
Recalling \eqref{100311:eq1}, we have, by a straightforward computation which we omit: 
\begin{equation}\label{0319:eq2}
\begin{array}{l}
(L_Xc)_{\lambda_{h+1},\cdots,\lambda_{k+1}}(a_{h+1},\cdots,a_{k+1}) \\
\displaystyle{
=
\!
(-1)^{\frac{h(h-1)}2}\! \bigg(
\sum_{i=1}^h(-1)^{i+1} \phi(\partial_{\lambda_1},\dots,\partial_{\lambda_h})
{a_i}_{\lambda_i} \big( c_{\lambda_1,\stackrel{i}{\check{\cdots}},\lambda_{k+1}}
(a_1,\stackrel{i}{\check{\cdots}},a_{k+1}) \big)
} \\
\,\,\,\,\,\,\,
\displaystyle{
+\sum_{i=h+1}^{k+1}(-1)^{i} \big( {a_i}_{\lambda_i} \phi(\partial_{\lambda_1},\dots,\partial_{\lambda_h}) \big)
c_{\lambda_1,\stackrel{i}{\check{\cdots}},\lambda_{k+1}}(a_1,\stackrel{i}{\check{\cdots}},a_{k+1}) 
} \\
\,\,\,\,\,\,\,
\displaystyle{
+\sum_{i=1}^h\sum_{j=i+1}^{k+1}(-1)^{j+1} \phi(\partial_{\lambda_1},\dots,\partial_{\lambda_h})
} \\
\,\,\,\,\,\,\,\,\,\,\,\,\,\,\,\,\,\,\,\,\,\,\,\,\,\,\,\,\,\,\,\,\,
\displaystyle{
c_{\lambda_1,\dots\lambda_i+\lambda_j\stackrel{j}{\check{\cdots}},\lambda_{k+1}}
(a_1,\dots[{a_i}_{\lambda_i}{a_j}]\stackrel{j}{\check{\cdots}},a_{k+1}) \big)
\bigg)\,\bigg|_{\lambda_1=\cdots=\lambda_h=0}\,.
}
\end{array}
\end{equation}

Recall that $C_1(R,M)$ is canonically identified with $R\otimes M/\partial(R\otimes M)$.
With this identification, formulas \eqref{0319:eq1} and \eqref{0319:eq2}
become, for $h=1$,
$$
\begin{array}{c}
\displaystyle{
\!\!\!\!\!\!\!
(\iota_{a_1\otimes m}c)_{\lambda_2,\cdots,\lambda_k}(a_2,\cdots,a_k)
=
c_{\partial^M,\lambda_2,\cdots,\lambda_k}(a_1,\cdots,a_k)_\to m \,,
}\\
\displaystyle{
(L_{a\otimes m}c)_{\lambda_1,\cdots,\lambda_k}(a_1,\cdots,a_k)
=
\big(a_{\partial^M}c_{\lambda_1,\cdots,\lambda_k}(a_1,\cdots,a_k)\big)_\to m 
}\\
\displaystyle{
+\sum_{i=1}^k c_{\lambda_1,\cdots\lambda_i+\partial^M\cdots,\lambda_k}(a_1,\cdots \stackrel{j}{a}\cdots,a_k)_\to
({a_i}_{\lambda_i}m)
}\\
\displaystyle{
-\sum_{i=1}^k c_{\lambda_1,\cdots\lambda_i+\partial^M\cdots,\lambda_k}
(a_1,\cdots [a_{\partial^M}a_i]\cdots,a_k)_\to m\,.
}
\end{array}
$$
In particular, for $h=k=1$ we have, recalling the identifications 
$C^0(R,M)=M/\partial^MM$ and $C^1(R,M)=\Hom_{\mb F[\partial]}(R,M)$,
$$
\begin{array}{l}
\displaystyle{
(\iota_{a\otimes m}c)=\tint c(a)m\,,
}\\
\displaystyle{
(L_{a\otimes m}c)(b)
=
\big(a_{\partial^M}c(b)\big)_\to m
+_\leftarrow\big(b_{-\partial^M}m\big)c(a)
-c([a_{\partial^M}]b)_\to m\,.
}
\end{array}
$$
In the second term of the RHS, the left arrow means that $\partial^M$ should be moved to the left
to act on the whole expression.

\begin{theorem}\label{0319:th}
The contraction map $\iota_\cdot:\,C_\bullet(R,M)\to\End(C^\bullet(R,M))$,
and the Lie derivative map $L_\cdot:\,\Pi C_\bullet(R,M)\to\End(C^\bullet(R,M))$,
define a $\mb Z_+$-graded calculus structure $(C_\bullet(R,M),C^\bullet(R,M))$.
\end{theorem}
\begin{proof}
According to Remark \ref{0228:rem}, we only need to prove that
$\iota_\cdot$ gives a representation of the associative superalgebra $C_\bullet(R,M)$
on $C^\bullet(R,M)$, and that equation \eqref{comp1} holds.
Applying formula \eqref{0319:eq1} twice and using the skewsymmetry condition on $C^\ell(R,M)$
for the permutation $\tau_{h,k-h}$ in \eqref{0317:eq0},
we get, for $X\in C_h(R,M),\,Y\in C_{k-h}(R,M),\,c\in C^{\ell}(R,M)$,
that $\iota_X\iota_Y(c)=\iota_{X\wedge Y}(c)$, due to the identity
\begin{equation}\label{0319:eq3}
s_{h,k-h}:={\frac{h(h-1)}2+\frac{(k-h)(k-h-1)}2+h(k-h)}\equiv {\frac{k(k-1)}2}\,\text{ mod } 2\,.
\end{equation}
This explains the choice of the sign factor $(-1)^{\frac{h(h-1)}2}$ 
in the definition \eqref{0319:eq1} of the contraction operators.
Next, we prove equation \eqref{comp1}.
Let $X=a_1\otimes\cdots\otimes a_h\otimes\phi\in C_h(R,M),\,
Y=a_{h+1}\otimes\cdots\otimes a_k\otimes\psi\in C_{k-h}(R,M),\,
c=c_{\lambda_1,\cdots,\lambda_\ell}(b_1,\cdots,b_\ell)\in C^\ell(R,M)$.
We want to prove that
\begin{equation}\label{0319:eq4}
\begin{array}{l}
(L_X\iota_Y c)_{\lambda_{k+1},\cdots,\lambda_{\ell+1}}(a_{k+1},\cdots,a_{\ell+1}) \\
\vphantom{\bigg(}
+(-1)^{1+(h+1)(k-h)}
(\iota_Y L_X c)_{\lambda_{k+1},\cdots,\lambda_{\ell+1}}(a_{k+1},\cdots,a_{\ell+1}) \\
= (\iota_{[X,Y]} c)_{\lambda_{k+1},\cdots,\lambda_{\ell+1}}(a_{k+1},\cdots,a_{\ell+1}) \,.
\end{array}
\end{equation}
By \eqref{0319:eq1} and \eqref{0319:eq2}, the first term in the RHS of \eqref{0319:eq4} is
\begin{equation}\label{0319:eq5}
\begin{array}{l}
\displaystyle{
(-1)^{s_{h,k-h}+(k-h)}
\Bigg(
\sum_{i=1}^h (-1)^{i+1}
\phi(\partial_{\lambda_1},\dots,\partial_{\lambda_h})
}\\
\displaystyle{
\,\,\,\,\,\,\,\,\,\,\,\,\,\,\,\,\,\,\,\,\,\,\,\,\,\,\,\,\,\,\,\,\,\,\,\,
\,\,\,\,\,\,\,\,\,\,\,\,\,\,\,\,\,\,\,\,\,\,\,\,\,\,\,\,
{a_i}_{\lambda_i}\big(
\psi(\partial_{\lambda_{h+1}},\dots,\partial_{\lambda_k})
c_{\lambda_1,\stackrel{i}{\check{\cdots}},\lambda_{\ell+1}}(a_1,\stackrel{i}{\check{\cdots}},a_{\ell+1})
\big)
}\\
\displaystyle{
+ \sum_{i=k+1}^{\ell+1} (-1)^{i}
\big( {a_i}_{\lambda_i}
\phi(\partial_{\lambda_1},\dots,\partial_{\lambda_h})
\big)
\psi(\partial_{\lambda_{h+1}},\dots,\partial_{\lambda_k})
}\\
\displaystyle{
\,\,\,\,\,\,\,\,\,\,\,\,\,\,\,\,\,\,\,\,\,\,\,\,\,\,\,\,\,\,\,\,\,\,\,\,
\,\,\,\,\,\,\,\,\,\,\,\,\,\,\,\,\,\,\,\,\,\,\,\,\,\,\,\,\,\,\,\,\,\,\,\,
\,\,\,\,\,\,\,\,\,\,\,\,\,\,\,\,\,\,\,\,\,\,\,\,\,\,\,\,\,\,\,\,\,\,\,\,
\,\,\,\,\,\,\,\,\,\,\,\,\,\,\,
c_{\lambda_1,\stackrel{i}{\check{\cdots}},\lambda_{\ell+1}}(a_1,\stackrel{i}{\check{\cdots}},a_{\ell+1})
}\\
\displaystyle{
+ \sum_{i=1}^h\sum_{j=i+1}^h (-1)^{j+1}
\phi(\partial_{\lambda_1},\dots,\partial_{\lambda_h})
\psi(\partial_{\lambda_{h+1}},\dots,\partial_{\lambda_k})
}\\
\displaystyle{
\,\,\,\,\,\,\,\,\,\,\,\,\,\,\,\,\,\,\,\,\,\,\,\,\,\,\,\,\,\,\,\,\,\,\,\,
\,\,\,\,\,\,\,\,\,\,\,\,\,\,\,\,\,\,\,\,\,\,\,\,\,\,\,\,\,\,\,\,\,\,\,\,
\,\,\,\,\,\,
c_{\lambda_1,\cdots\lambda_i+\lambda_j\stackrel{j}{\check{\cdots}},\lambda_{\ell+1}}
(a_1,\cdots[{a_i}_{\lambda_i}{a_j}]\stackrel{j}{\check{\cdots}},a_{\ell+1})
}\\
\displaystyle{
+ \sum_{i=1}^h\sum_{j=k+1}^{\ell+1} (-1)^{j+1}
\phi(\partial_{\lambda_1},\dots,\partial_{\lambda_h})
\psi(\partial_{\lambda_{h+1}},\dots,\partial_{\lambda_k})
}\\
\displaystyle{
\,\,\,\,\,\,\,\,\,\,\,\,\,\,\,\,\,\,\,\,\,\,\,\,\,\,\,\,\,\,\,\,\,\,\,\,
\,\,\,\,
c_{\lambda_1,\cdots\lambda_i+\lambda_j\stackrel{j}{\check{\cdots}},\lambda_{\ell+1}}
(a_1,\cdots[{a_i}_{\lambda_i}{a_j}]\stackrel{j}{\check{\cdots}},a_{\ell+1})
\Bigg)\,\Bigg|_{\lambda_1=\cdots=\lambda_k=0}.
}
\end{array}
\end{equation}
Similarly, the second term in the RHS of \eqref{0319:eq4} is
\begin{equation}\label{0319:eq6}
\begin{array}{l}
\displaystyle{
(-1)^{s_{h,k-h}+(k-h)+1}
\Bigg(
\sum_{i=1}^h (-1)^{i+1}
\phi(\partial_{\lambda_1},\dots,\partial_{\lambda_h})
\psi(\partial_{\lambda_{h+1}},\dots,\partial_{\lambda_k})
}\\
\displaystyle{
\,\,\,\,\,\,\,\,\,\,\,\,\,\,\,\,\,\,\,\,\,\,\,\,\,\,\,\,\,\,\,\,\,\,\,\,
\,\,\,\,\,\,\,\,\,\,\,\,\,\,\,\,\,\,\,\,\,\,\,\,\,\,\,\,\,\,\,\,\,\,\,\,
\,\,\,\,\,\,\,\,\,\,\,\,\,\,\,\,\,\,\,\,\,\,\,\,\,\,\,\,\,\,\,\,\,\,\,\,\,
\big({a_i}_{\lambda_i}
c_{\lambda_1,\stackrel{i}{\check{\cdots}},\lambda_{\ell+1}}(a_1,\stackrel{i}{\check{\cdots}},a_{\ell+1})
\big)
}\\
\displaystyle{
+ \sum_{i=h+1}^{\ell+1} (-1)^{i}
\psi(\partial_{\lambda_{h+1}},\dots,\partial_{\lambda_k})
\big( {a_i}_{\lambda_i}
\phi(\partial_{\lambda_1},\dots,\partial_{\lambda_h}) \big)
}\\
\displaystyle{
\,\,\,\,\,\,\,\,\,\,\,\,\,\,\,\,\,\,\,\,\,\,\,\,\,\,\,\,\,\,\,\,\,\,\,\,
\,\,\,\,\,\,\,\,\,\,\,\,\,\,\,\,\,\,\,\,\,\,\,\,\,\,\,\,\,\,\,\,\,\,\,\,
\,\,\,\,\,\,\,\,\,\,\,\,\,\,\,\,\,\,\,\,\,\,\,\,\,\,\,\,\,\,\,\,\,\,\,\,
\,\,\,\,\,\,\,\,\,\,\,\,\,\,\,\,\,
c_{\lambda_1,\stackrel{i}{\check{\cdots}},\lambda_{\ell+1}}(a_1,\stackrel{i}{\check{\cdots}},a_{\ell+1})
}\\
\displaystyle{
+ \sum_{i=1}^h\sum_{j=i+1}^{\ell+1} (-1)^{j+1}
\phi(\partial_{\lambda_1},\dots,\partial_{\lambda_h})
\psi(\partial_{\lambda_{h+1}},\dots,\partial_{\lambda_k})
}\\
\displaystyle{
\,\,\,\,\,\,\,\,\,\,\,\,\,\,\,\,\,\,\,\,\,\,\,\,\,\,\,\,\,\,\,\,\,\,\,\,
\,\,\,\,
c_{\lambda_1,\cdots\lambda_i+\lambda_j\stackrel{j}{\check{\cdots}},\lambda_{\ell+1}}
(a_1,\cdots[{a_i}_{\lambda_i}{a_j}]\stackrel{j}{\check{\cdots}},a_{\ell+1})
\Bigg)\,\Bigg|_{\lambda_1=\cdots=\lambda_k=0}.
}
\end{array}
\end{equation}
Combining \eqref{0319:eq5} and \eqref{0319:eq6}, we get that the LHS of \eqref{0319:eq4} is
\begin{equation}\label{0319:eq7}
\begin{array}{l}
\displaystyle{
(-1)^{s_{h,k-h}+(k-h)}
\Bigg(
\sum_{i=1}^h (-1)^{i+1}
\phi(\partial_{\lambda_1},\dots,\partial_{\lambda_h})
\big({a_i}_{\lambda_i}
\psi(\partial_{\lambda_{h+1}},\dots,\partial_{\lambda_k})\big)
}\\
\displaystyle{
\,\,\,\,\,\,\,\,\,\,\,\,\,\,\,\,\,\,\,\,\,\,\,\,\,\,\,\,\,\,\,\,\,\,\,\,
\,\,\,\,\,\,\,\,\,\,\,\,\,\,\,\,\,\,\,\,\,\,\,\,\,\,\,\,\,\,\,\,\,\,\,\,
\,\,\,\,\,\,\,\,\,\,\,\,\,\,\,\,\,\,\,\,\,\,\,\,\,\,\,\,\,\,\,\,\,\,\,\,
\,\,\,\,\,\,\,\,\,\,\,\,\,\,\,\,\,
c_{\lambda_1,\stackrel{i}{\check{\cdots}},\lambda_{\ell+1}}(a_1,\stackrel{i}{\check{\cdots}},a_{\ell+1})
}\\
\displaystyle{
+ \sum_{i=h+1}^{k} (-1)^{i+1}
\psi(\partial_{\lambda_{h+1}},\dots,\partial_{\lambda_k})
\big( {a_i}_{\lambda_i}
\phi(\partial_{\lambda_1},\dots,\partial_{\lambda_h}) \big)
}\\
\displaystyle{
\,\,\,\,\,\,\,\,\,\,\,\,\,\,\,\,\,\,\,\,\,\,\,\,\,\,\,\,\,\,\,\,\,\,\,\,
\,\,\,\,\,\,\,\,\,\,\,\,\,\,\,\,\,\,\,\,\,\,\,\,\,\,\,\,\,\,\,\,\,\,\,\,
\,\,\,\,\,\,\,\,\,\,\,\,\,\,\,\,\,\,\,\,\,\,\,\,\,\,\,\,\,\,\,\,\,\,\,\,
\,\,\,\,\,\,\,\,\,\,\,\,\,\,\,\,\,
c_{\lambda_1,\stackrel{i}{\check{\cdots}},\lambda_{\ell+1}}(a_1,\stackrel{i}{\check{\cdots}},a_{\ell+1})
}\\
\displaystyle{
+ \sum_{i=1}^h\sum_{j=h+1}^k (-1)^{j}
\phi(\partial_{\lambda_1},\dots,\partial_{\lambda_h})
\psi(\partial_{\lambda_{h+1}},\dots,\partial_{\lambda_k})
}\\
\displaystyle{
\,\,\,\,\,\,\,\,\,\,\,\,\,\,\,\,\,\,\,\,\,\,\,\,\,\,\,\,\,\,\,\,\,\,\,\,
\,\,\,\,
c_{\lambda_1,\cdots\lambda_i+\lambda_j\stackrel{j}{\check{\cdots}},\lambda_{\ell+1}}
(a_1,\cdots[{a_i}_{\lambda_i}{a_j}]\stackrel{j}{\check{\cdots}},a_{\ell+1})
\Bigg)\,\Bigg|_{\lambda_1=\cdots=\lambda_k=0}.
}
\end{array}
\end{equation}
We next observe that, for $i=1,\dots,h$,
$$
\phi(\partial_{\lambda_1},\dots,\partial_{\lambda_h})
\big({a_i}_{\lambda_i}
\psi(\partial_{\lambda_{h+1}},\dots,\partial_{\lambda_k})\big)
c_{\lambda_1,\stackrel{i}{\check{\cdots}},\lambda_{\ell+1}}(a_1,\stackrel{i}{\check{\cdots}},a_{\ell+1})
\Big|_{\lambda_1=\cdots=\lambda_k=0}\,,
$$
is the contraction of $c$ by
$$
a_1\otimes\stackrel{i}{\check{\cdots}}\otimes a_k
\otimes \big( {a_i}\,_{\partial_y}\psi(x_h,\dots,x_{k-1}) \big) 
\phi(x_1,\dots\stackrel{i}{y}\dots,x_{h-1}) \,\Big|_{y=0}\,,
$$
that, for $i=h+1,\dots,k$,
$$
\psi(\partial_{\lambda_{h+1}},\dots,\partial_{\lambda_k})
\big( {a_i}_{\lambda_i}
\phi(\partial_{\lambda_1},\dots,\partial_{\lambda_h}) \big)
c_{\lambda_1,\stackrel{i}{\check{\cdots}},\lambda_{\ell+1}}(a_1,\stackrel{i}{\check{\cdots}},a_{\ell+1})
\Big|_{\lambda_1=\cdots=\lambda_k=0}\,,
$$
is the contraction of $c$ by
$$
a_1\otimes\stackrel{i}{\check{\cdots}}\otimes a_k
\otimes \big( {a_i}\,_{\partial_y}\phi(x_1,\dots,x_h) \big) 
\psi(x_{h+1},\dots\stackrel{i}{y}\dots,x_{k-1}) \,\Big|_{y=0}\,,
$$
and that, for $i=1,\dots,h,\,k=i+1,\dots,k$,
$$\
\begin{array}{l}
\phi(\partial_{\lambda_1},\dots,\partial_{\lambda_h})
\psi(\partial_{\lambda_{h+1}},\dots,\partial_{\lambda_k}) \\
\,\,\,\,\,\,\,\,\,\,\,\,\,\,\,
c_{\lambda_1,\cdots\lambda_i+\lambda_j\stackrel{j}{\check{\cdots}},\lambda_{\ell+1}}
(a_1,\cdots[{a_i}_{\lambda_i}{a_j}]\stackrel{j}{\check{\cdots}},a_{\ell+1})
\Big|_{\lambda_1=\cdots=\lambda_k=0}\,,
\end{array}
$$
is the contraction of $c$ by
$$
\big( a_1\otimes\cdots[{a_i}\, _{\partial_{x_i}}a_j]
\stackrel{j}{\check{\cdots}}\otimes a_k
\otimes\phi(x_1,\dots,x_h)\psi(x_{h+1},\dots\stackrel{j}{y}\dots,x_{k-1})\,
\big) \Big|_{y=x_i}\,.
$$
Finally, we combine the above results and we use equation \eqref{0319:eq3}
to conclude, 
recalling the definition \eqref{0312:eq12} of the Lie bracket on $C_\bullet(R,M)$,
that \eqref{0319:eq7} is equal to the RHS of \eqref{0319:eq4}.
\flushright\qed
\end{proof}

\begin{remark}\label{0322:rem3}
Recall from Remark \ref{0322:rem1} that $\bar C^\bullet(R,M)$ is a subcomplex of
the complex $C^\bullet(R,M)$,
and from Remark \ref{0322:rem2} that $\bar C_\bullet(R,M)=C_\bullet(R,M)/T_\bullet$
is a Gerstenhaber factor algebra of $C_\bullet(R,M)$.
It is immediate to check from \eqref{0319:eq1} that $\iota_X(c)=0$
if $X\in T_\bullet$ and $c\in\bar C^\bullet(R,M)$.
Hence, we have the induced calculus structure $(\bar C_\bullet(R,M),\bar C^\bullet(R,M))$,
with the canonical morphism of calculus structures 
$(\bar C_\bullet(R,M),\bar C^\bullet(R,M))\to(C_\bullet(R,M),C^\bullet(R,M))$.
\end{remark}

\begin{remark}\label{0324:rem1}
We can define another calculus structure 
$(\tilde C_\bullet(R,M),\tilde C^\bullet(R,M))$ associated to the $R$-module $M$,
called the \emph{basic Lie conformal algebra calculus structure}.
The space of basic $k$-cochains $\tilde C^k(R,M)$
is the space of $\mb F$-linear maps: 
$c:\, R^{\otimes k}\to\mb F[\lambda_1,\dots,\lambda_k]\otimes M$,
satisfying the same sesquilinearity and skewsymmetry conditions as for $C^k(R,M)$.
The differential $d$ on $\tilde C^\bullet(R,M)$ is defined by the same formula \eqref{100311:eq1}
as for the complex $C^\bullet(R,M)$.
The space of $k$-chains $\tilde C_k(R,M)$ is defined as the quotient 
of the space $R^{\otimes k}\otimes M[[x_1,\dots,x_k]]$ 
by the same sesquilinearity and skewsymmetry relations as for $C_k(R,M)$.
The structure of Gerstenhaber algebra on $\tilde C_\bullet(R,M)$ is given by the same formulas
\eqref{0312:eq1} and \eqref{0312:eq12} as for $C_\bullet(R,M)$.
We then define the contraction map 
$\iota_\cdot:\,\tilde C_\bullet(R,M)\to\End(\tilde C^\bullet(R,M))$
by formula \eqref{0319:eq1},
and the same arguments (in a simpler form) as in the proof of Theorem \ref{0317:th}
show that $(\tilde C_\bullet(R,M),\tilde C^\bullet(R,M))$
is a calculus structure.
Moreover, we have an obvious morphism of calculus structures
$(\tilde C_\bullet(R,M),\tilde C^\bullet(R,M))\to(C_\bullet(R,M),C^\bullet(R,M))$.
\end{remark}

\subsection{Calculus structure for a Lie conformal algebra complex and a reduction 
of the calculus structure for a Lie algebra complex}
\label{sec:4.5}

As before, let $R$ be a Lie conformal algebra and $M$ be a module over $R$
endowed with the structure of a commutative associative algebra
on which $\partial^M$ and $a_\lambda,\,a\in R$, act by derivations.
We have associated to the pair $(R,M)$ the calculus structure
$(C_\bullet(R,M),C^\bullet(R,M))$.
Recall also from Remark \ref{0322:rem3} that it has the ``torsion free''
subcalculus structure $(\bar C_\bullet,\bar C^\bullet)\to(C_\bullet(R,M),C^\bullet(R,M))$.

Furthermore, we assume that $R$ is of finite rank as an $\mb F[\partial]$-module,
so that the annihilation Lie algebra $\Lie_- R$ is a linearly compact Lie algebra.
By Proposition \ref{100310:prop}, $M$, endowed with the discrete topology, 
is a continuous module over $\Lie_-R$,
and moreover $\Lie_-R$ acts by derivations of the algebra $M$
and its action extends to the semidirect product $(\Lie_-R)\rtimes\mb F \partial$,
with $\partial$ acting as $\partial^M$.

Recall the calculus structure 
$(\hat\Delta_\bullet(\Lie_-R,M),\check\Delta^\bullet(\Lie_-R,M))$
from Section \ref{sec:3.2}.
The action of $\partial$ on $\Lie_-R$ and $M$
induces its natural action on both $\hat\Delta_\bullet(\Lie_-R,M)$
and $\check\Delta^\bullet(\Lie_-R,M))$, preserving the $\mb Z_+$-gradings.
More precisely, the action of $\partial$ on 
$\hat\Delta_k(\Lie_-R,M)=M\hat\otimes\hat\bigwedge^k\Lie_- R$
is given by $\partial(f\otimes X)=(\partial^M f)\otimes X+f\otimes\partial X$,
and its action on 
$\check\Delta^k(\Lie_-R,M))=\Hom^c_{\mb F}(\hat\bigwedge^k\Lie_- R,M)$
is given by $(\partial\omega)(X)=\partial^M(\omega(X))-\omega(\partial X)$.
It is immediate to check that the action of $\partial$ on $\check\Delta^\bullet(\Lie_-R,M)$
commutes with the action of the differential $d$ in \eqref{0301:eq1}.
Hence we can consider the complex 
$(\check\Delta^\bullet(\Lie_-R,M)/\partial\check\Delta^\bullet(\Lie_-R,M),d)$.
Furthermore, for $a\in\hat\Delta_k(\Lie_-R,M)$, we have the identity 
$$
[\partial,\iota_a]=\iota_{\partial a}\,.
$$
Hence, the Gerstenhaber subalgebra $\hat\Delta_\bullet(\Lie_-R,M)^\partial$ from 
Example \ref{0321:ex} is the kernel of the action of $\partial$
on $\hat\Delta_\bullet(\Lie_-R,M)$.
Now we can consider the reduced calculus structure
$(\hat\Delta_\bullet(\Lie_-R,M)^\partial,\check\Delta^\bullet(\Lie_-R,M)/\partial\check\Delta^\bullet(\Lie_-R,M))$.

In this section we will relate all these calculus structures.
We define a $\mb Z_+$-grading preserving linear map 
$\Phi_\bullet:\,C_\bullet(R,M)\to\hat\Delta_\bullet(\Lie_-R,M)$
by the following formula
\begin{equation}\label{0322:eq1}
\begin{array}{l}
\displaystyle{
\Phi_k(a_1\otimes\cdots\otimes a_k\otimes\phi(x_1,\dots,x_k)) 
}\\
\displaystyle{
=\sum_{m_1,\cdots,m_k\in\mb Z_+}
a_{1,m_1}\wedge\cdots\wedge a_{k,m_k}\otimes
\frac{\partial_{x_1}^{m_1}\cdots\partial_{x_k}^{m_k}}{m_1!\cdots m_k!}\phi\,\big|_{x_1=\cdots=x_k=0}\,.
}
\end{array}
\end{equation}
Similarly, we define a $\mb Z_+$-grading preserving linear map 
$\Psi^\bullet:\,\check\Delta^\bullet(\Lie_-R,M)\to C^\bullet(R,M)$
by the following formula
\begin{equation}\label{0322:eq2}
(\Psi^k\omega)_{\lambda_1,\cdots,\lambda_k}(a_1,\cdots,a_k)
=
\!\!\!\sum_{m_1,\cdots,m_k\in\mb Z_+}\!\!\!
\frac{\lambda_1^{m_1}\cdots\lambda_k^{m_k}}{m_1!\cdots m_k!}
\omega(a_{1,m_1}\wedge\cdots\wedge a_{k,m_k})\,.
\end{equation}

\begin{theorem}\label{0322:th}
Let $R$ be a Lie conformal algebra of finite rank as an $\mb F[\partial]$-module,
and let $M$ be a module over $R$
endowed with the structure of a commutative associative algebra
on which $\partial^M$ and $a_\lambda,\,a\in R$, act by derivations.
Then the maps $\Phi_\bullet$ and $\Psi^\bullet$ defined by \eqref{0322:eq1} and \eqref{0322:eq2}
give a morphism of calculus structures
\begin{equation}\label{0324:eq8}
(\hat\Delta_\bullet(\Lie{}_-R,M),\check\Delta^\bullet(\Lie{}_-R,M))\to(C_\bullet(R,M),C^\bullet(R,M))\,,
\end{equation}
which induces a calculus structure isomorphism 
$$
(\hat\Delta_\bullet(\Lie{}_-R,M)^\partial,
\check\Delta^\bullet(\Lie{}_-R,M)/\partial\check\Delta^\bullet(\Lie{}_-R,M))
\simeq
(\bar C_\bullet(R,M),\bar C^\bullet(R,M))\,.
$$
\end{theorem}
\begin{proof}
First, we need to check that the maps $\Phi_\bullet$ and $\Psi^\bullet$ are well-defined. Let 
$X=a_1\otimes\cdots \otimes a_h\otimes\phi\in C_h(R,M)$ and $\omega\in \check\Delta^\bullet(\Lie{}_-R,M)$.
Using the fact that $\partial$ acts by $-\partial_t$ on $\Lie{}_-R$, we have 
$$
\Phi_h(a_1\otimes\cdots\partial a_i\cdots \otimes a_h\otimes\phi) =
$$
$$
=-\sum_{m_1,\cdots,m_h\in\mb Z_+}
a_{1,m_1}\wedge\cdots m_i a_{i,m_i-1} \cdots \wedge a_{h,m_h}\otimes
\frac{\partial_{x_1}^{m_1}\cdots\partial_{x_h}^{m_h}}{m_1!\cdots m_h!}\phi\,\big|_{x_1=\cdots=x_h=0}
$$
$$
=-\sum_{m_1,\cdots,m_h\in\mb Z_+}
a_{1,m_1}\wedge\cdots \wedge a_{h,m_h}\otimes
\frac{\partial_{x_1}^{m_1}\cdots\partial_{x_i}^{m_i+1}\cdots\partial_{x_h}^{m_h}}{m_1!\cdots m_h!}\phi\,\big|_{x_1=\cdots=x_h=0}
$$
$$
=\Phi_h(-a_1\otimes\cdots \otimes a_h\otimes\partial_{x_i}\phi)\,,
$$
thus proving that $\Phi_h(X)$ satisfies the sesquilinearity condition in $C_h(R,M)$. The skewsymmetry property 
follows immediately by the skewsymmetry of the wedge product. Likewise, using $\partial a_m = -m a_{m-1}$ and 
relabelling the indices, it is immediate to check that $\Psi^k(\omega)$ fulfills the sesquilinearity and skewsymmetry
conditions in $C^k(R,M)$. By the same token, one has that the action of $\partial$ on 
$\hat\Delta_\bullet(\Lie{}_-R,M)$ is given by
$$
\partial \Phi_h(a_1\otimes\cdots \otimes a_k\otimes \phi)= \Phi_h\left(a_1\otimes\cdots \otimes a_k
\otimes (\partial^M-\partial_{x_1}-\cdots-\partial_{x_h})\phi\right)\,.
$$
Recalling condition \eqref{0318:eq1} satisfied by $\phi\in \mc M_h$, we conclude that the image of 
$\Phi_\bullet$ is $\ker(\partial)$. Moreover, the kernel of $\Phi_\bullet$ is precisely the torsion part 
in $C_\bullet(R,M)$. Indeed, as pointed out in Remark \ref{0322:rem2}, all torsion is contained in $C_1(R,M)$
and since $R$ is assumed to be of finite rank, it follows that $\ker(\Phi_\bullet)=\big(\bigoplus_i M/(P_i(\partial)M)\big)$.

Next, we prove that $\Phi_\bullet$ is a homomorphism of Gerstenhaber algebras. The identity 
$\Phi_{h+k}(a\wedge b) = \Phi_h(a)\wedge \Phi_k(b)$ follows immediately by \eqref{0312:eq1}
and definition \eqref{0322:eq1}. To see that the Gerstenhaber bracket is preserved by $\Phi_\bullet$,
we first note that the bracket on $\hat\Delta_\bullet(\Lie{}_-R,M)$ is given by \eqref{0225:eq6}, 
extended with two additional terms due to the non-trivial action of $\Lie{}_-R$ on the coefficient module $M$,
\begin{equation}\label{0527:eq1}
\begin{array}{l}
\displaystyle{
\big[a_{1,m_1}\wedge\cdots \wedge a_{h,m_h}\otimes u
\,,\,a_{h+1,m_{h+1}} \wedge\cdots \wedge a_{k,m_k}\otimes v\big] 
}\\
\displaystyle{
=
\sum_{i=1}^h\sum_{j=h+1}^k (-1)^{h+j+1}
a_{1,m_1} \wedge\cdots [{a_{i,m_i}}\,,a_{j,m_j}] 
\stackrel{j}{\check{\cdots}} \wedge a_{k,m_k}\otimes uv
}\\
\displaystyle{
+\! \sum_{i=1}^h (-1)^{h+i} a_{1,m_1} \wedge\stackrel{i}{\check{\cdots}} \wedge a_{k,m_k}
\otimes ( {a_{i,m_i}}.v)u}\\
\displaystyle{
+\! \sum_{i=h+1}^k (-1)^{h+i} a_{1,m_1} \wedge\stackrel{i}{\check{\cdots}} \wedge a_{k,m_k}
\otimes( {a_{i,m_i}}.u)v\,,
}
\end{array}
\end{equation}
where $[{a_{i,m_i}}\,,a_{j,m_j}]$ is defined by \eqref{commutator} and $u,v\in M$. Here we have 
shifted the commutator in the RHS of \eqref{0225:eq6} to position $i$, changing the overall sign by $(-1)^{i-1}$. Expanding the first term in \eqref{0312:eq12} using 
$[{a_i}\, _{\partial_{x_i}}a_j]=\sum_{l\in\mb Z_+}(a_{i(l)}a_j)\frac{\partial_{x_i}^l}{l!}$, a straightforward 
computation yields
\begin{equation*}
\begin{array}{l}
\displaystyle{
\sum_{m_1,\cdots,m_k\in\mb Z_+}\sum_{i=1}^h\sum_{j=h+1}^k (-1)^{h+j+1} a_1\otimes \cdots\sum_{l\in\mb Z_+}\binom{m_i}{l}(a_{i(l)}a_j)
\cdots\otimes a_k\otimes
}\\
\,\,\,\,\,\,\,\,\,\,\,\,\,\,\,\,\,\,\,\,\,\,\,\,\,\,\,
\displaystyle{
\otimes\phi_{(m_1,\dots,m_h)}\psi_{(m_{h+1},\dots,m_k)}\frac{x_1^{m_1}\cdots x_i^{m_i+m_j-l} \cdots
x_{j-1}^{m_{j-1}}x_j^{m_{j+1}}\cdots x_{k-1}^{m_k}}{m_1!\cdots m_k!}\,,
}
\end{array}
\end{equation*}
which corresponds to the first term in \eqref{0527:eq1} under the map $\Phi_{k-1}$. Likewise, expanding
the second term in \eqref{0312:eq12} using \eqref{0606:eq1}, we get
\begin{equation*}
\begin{array}{l}
\displaystyle{
\sum_{m_1,\cdots,m_k\in\mb Z_+} \sum_{i=1}^h (-1)^{h+i} a_1\otimes\stackrel{i}{\check{\cdots}}\otimes a_k
\otimes \big(a_{i(m_i)}\psi_{(m_{h+1},\dots,m_k)}\big)\phi_{(m_1,\dots,m_h)}
}\\
\,\,\,\,\,\,\,\,\,\,\,\,\,\,\,\,\,\,\,\,\,\,\,\,\,\,\,
\,\,\,\,\,\,\,\,\,\,\,\,\,\,\,\,\,\,\,\,\,\,\,\,\,\,\,
\,\,\,\,\,\,\,\,\,\,\,\,\,\,\,\,\,\,\,\,\,\,\,\,\,\,\,
\,\,\,\,\,\,\,\,\,\,\,\,\,\,\,\,\,\,\,\,\,\,\,\,\,\,\,
\displaystyle{
\frac{x_1^{m_1}\cdots
x_{i-1}^{m_{i-1}}x_i^{m_{i+1}}\cdots x_{k-1}^{m_k}}{m_1!\cdots m_k!}\,,
}
\end{array}
\end{equation*}
which is sent by $\Phi_{k-1}$ to the second term in \eqref{0527:eq1}, and similarly for the last term.

We show next that $\Psi^\bullet$ commutes with the action of the differentials $d$ and hence defines
a morphism of complexes. By \eqref{0322:eq2}, we have
$$
\Psi^{k+1}(d\omega)_{\lambda_1,\cdots,\lambda_{k+1}}(a_1,\dots,a_{k+1}) =
$$
$$=
\!\!\!\sum_{m_1,\cdots,m_{k+1}\in\mb Z_+}\!\!\!
\frac{\lambda_1^{m_1}\cdots\lambda_k^{m_{k+1}}}{m_1!\cdots m_{k+1}!}
(d\omega)(a_{1,m_1}\wedge\cdots\wedge a_{k+1,m_{k+1}})\,,
$$
where $d\omega$ is given by the formula \eqref{0301:eq1}. Combining \eqref{100311:eq1} and \eqref{0606:eq1},
it is not hard to check that the first terms in $d\Psi^k(\omega)$ and $\Psi^{k+1}(d\omega)$ are equal. Expanding the last 
term in $d\Psi^k(\omega)$, we obtain
\begin{equation*}
\begin{array}{l}
\displaystyle{
\sum_{m_1,\cdots,m_{k+1}\in\mb Z_+}\sum_{i<j} (-1)^{j+1} \Psi^k(\omega)_{(m_1,\dots,m_{k+1})}
\Big(a_1,\cdots,\sum_{l\in\mb Z_+}(a_{i(l)}a_j)\frac{\lambda_i^l}{l!},\stackrel{j}{\check{\cdots}} ,a_{k+1}\Big)
}\\
\,\,\,\,\,\,\,\,\,\,\,\,\,\,\,\,\,\,\,\,\,\,\,\,\,\,\,
\,\,\,\,\,\,\,\,\,\,\,\,\,\,\,\,\,\,\,\,\,\,\,\,\,\,\,
\,\,\,\,\,\,\,\,\,\,\,\,\,\,\,\,\,\,\,\,\,\,\,\,\,\,\,
\displaystyle{
\frac{\lambda_1^{m_1}\cdots (\lambda_i+\lambda_j)^{m_i}\stackrel{j}{\check{\cdots}}\lambda_{k+1}^{m_{k+1}}}
{m_1!\cdots m_{k+1}!}\,.
}
\end{array}
\end{equation*}
Using the following identity,
$$
\frac{(\lambda_i+\lambda_j)^{m_i}\lambda_i^l}{m_i!l!}=\frac{1}{m_i!l!}\sum_{m_j\in\mb Z_+}
\binom{m_i}{m_j}\lambda_i^{m_i-m_j+l}\lambda_j^{m_j}=\sum_{m_j\in\mb Z_+}
\binom{\tilde m_i}{l}\frac{\lambda_i^{\tilde m_i}\lambda_j^{m_j}}{\tilde m_i!m_j!}\,,
$$
where $m_i=\tilde m_i+m_j-l$, we conclude that also the last terms in $d\Psi^k(\omega)$ and $\Psi^{k+1}(d\omega)$
coincide, thus proving the claim.

The image of $\Psi^\bullet$ corresponds to the free part in $C^\bullet(R,M)$. Indeed, by Remark \ref{0322:rem1}, 
only $C^1(R,M)$ contains torsion and assuming that $R$ decomposes as in \eqref{0322:eq3}, it is clear by \eqref{0322:eq2} that $\im(\Psi^\bullet)=\bar C^\bullet(R,M)$. Moreover, it is not hard to check that the action  of $\partial$ on $\check\Delta^\bullet(\Lie{}_-R,M)$, 
given by $(\partial\omega)(X)=\partial^M(\omega(X))-\omega(\partial X)$, leads to
$$
\big(\Psi^k(\partial\omega)\big)_{\lambda_1,\cdots,\lambda_k}(a_1,\cdots,a_k) 
= (\partial+\lambda_1+\cdots+\lambda_k)
\big(\Psi^k(\omega)\big)_{\lambda_1,\cdots,\lambda_k}(a_1,\cdots,a_k) \,.
$$
Since $\mb F_-[\lambda_1,\dots,\lambda_k]\otimes_{\mb F[\partial]}\mc V$
is the quotient of $\mb F[\lambda_1,\dots,\lambda_k]\otimes_{\mb F}\mc V$
by the image of $(\partial+\lambda_1+\cdots+\lambda_k)$,
it follows that $\ker(\Psi^k)=\im(\partial)$.

Finally, we check that the contraction operators are compatible with the homomorphisms, namely that 
$\Psi^\bullet(\iota_{\Phi_\bullet(X)}(\omega))=\iota_X(\Psi^\bullet(\omega))$. 
By \eqref{0304:eq2}, it follows that the coefficient of $\lambda_{h+1}^{m_{h+1}}\cdots\lambda_k^{m_k}$
in $\Psi^{k-h}(\iota_{\Phi_\bullet(X)}(\omega))_{\lambda_{h+1},\cdots,\lambda_k}(a_{h+1},\cdots,
a_k)$ is
$$
\omega(a_{1,m_1}\wedge\cdots\wedge a_{k,m_k})
\frac{\partial_{x_1}^{m_1}\cdots\partial_{x_k}^{m_k}}{m_1!\cdots m_k!}\phi(x_1,\cdots,x_h)\,\big|_{x_1=\cdots=x_k=0}\,.
$$  
Likewise, recalling \eqref{0319:eq1}, the corresponding coefficient in the polynomial \\
$\iota_X(\Psi^k(\omega))_{\lambda_{h+1},\cdots,\lambda_k}(a_{h+1},\cdots,a_k)$ is
$$
\omega(a_{1,m_1}\wedge\cdots\wedge a_{k,m_k})
\phi(\partial_{\lambda_1},\cdots,\partial_{\lambda_h})\frac{\lambda_1^{m_1}\cdots\lambda_h^{m_h}}{m_1!\cdots m_h!}\,\big|_{\lambda_1=\cdots=\lambda_k=0}\,,
$$  
which, by noting $\partial^n_\lambda \lambda^m|_{\lambda=0}  = \partial^m_x x^n|_{x=0} = m!\delta_{m,n}$, proves the claim.
\flushright\qed
\end{proof}

\begin{remark}\label{0324:rem2}
Recall the calculus structure $(\tilde C_\bullet(R,M),\tilde C^\bullet(R,M))$
introduced in Remark \ref{0324:rem1}.
Formulas \eqref{0322:eq1} and \eqref{0322:eq2}
define an isomorphism of calculus structures
$(\hat\Delta_\bullet(\Lie{}_-R,M),\check\Delta^\bullet(\Lie{}_-R,M))
\simeq(\tilde C_\bullet(R,M),\tilde C^\bullet(R,M))$.
This isomorphism induces the morphism \eqref{0324:eq8}.
In other words, we have the following commutative diagram of calculus structures:
$$
\UseTips
\xymatrix{
(\hat\Delta_\bullet,\check\Delta^\bullet) \ar[d] \ar[r]^{\sim}
& (\tilde C_\bullet,\tilde C^\bullet) \ar[d] & \\
(\hat\Delta_\bullet^\partial,
\check\Delta^\bullet/\partial\check\Delta^\bullet)  \ar[r]^{\sim}
& (\bar C_\bullet,\bar C^\bullet) \ar[r]
& (C_\bullet,C^\bullet)\,.
}
$$
\end{remark}


\section{The complex of variational calculus}
\label{sec:5}

\subsection{de Rham complex over an algebra of differential functions}
\label{sec:5.1}

\begin{definition}\label{0323:def}\emph{\cite{DSK}}
Let $I = \{1,\dots,l\}$ be a finite index set. 
An \emph{algebra of differential functions} $\mathcal V$ in the variables $u_i,\,i\in I$, 
is a differential algebra, i.e. a unital, commutative, associative algebra 
with a derivation $\partial:\mathcal V\to \mathcal V$, together with commuting 
derivations 
$\frac{\partial}{\partial u_i^{(n)}}:\,\mathcal V \to \mathcal V$, 
such that
\begin{equation}\label{derivation}
\left[\frac{\partial}{\partial u_i^{(n)}}, \partial\right]= \frac{\partial}{\partial u_i^{(n-1)}}\,,
\end{equation}
and, for any $f \in \mathcal V$, $\frac{\partial f}{\partial u_i^{(n)}} = 0$
for all but finitely many $i\in I, n\in\mb Z_+$.
\end{definition}

The image of an element $f\in \mathcal V$ under the quotient map $\mathcal V \to \mathcal V/\partial \mathcal V$ 
is denoted, as before, $\int f$.

\begin{example}\label{0323:ex1}
The polynomial algebra
$R_\ell = \mb F[u_i^{(n)}]_{i\in I, n\in\mb Z_+}$
is an algebra of differential functions with $\partial u_i^{(n)}=u_i^{(n+1)}$
and the usual $\partial/\partial u_i^{(n)}$.
\end{example}
\begin{example}\label{0323:ex2}
The polynomial algebra $R_\ell[x]$ is an algebra of differential functions extension of $R_\ell$,
with $\partial x=1$.
\end{example}

A \emph{vector field} is a derivation of $\mathcal V$ of the form
\begin{equation}
X = \sum_{i\in I, n\in \mb Z_+} P_i^n\frac{\partial}{\partial u^{(n)}_i} , \ \ \ P^n_i \in \mathcal V\,.
\end{equation}
Obviously the space of vector fields is closed under the commutator,
and we denote the resulting Lie algebra by $\Vect(\mc V)$.
In fact, the pair $(\mc V,\Vect(\mc V))$ is a Lie algebroid,
with the obvious actions of $\mc V$ on $\Vect(\mc V)$ and of $\Vect(\mc V)$ on $\mc V$.
Hence we can consider the corresponding Gerstenhaber algebra
of \emph{polyvector fields} $\tilde\Omega_\bullet(\mc V)=S_{\mc V}(\Pi\Vect(\mc V))$,
given by Proposition \ref{schouten}.
Its elements have the following form:
\begin{equation}\label{0323:eq2}
X=\sum_{\substack{i_1,\cdots,i_k\in I \\ n_1,\cdots,n_k\in\mb Z_+}} 
P^{n_1\dots n_k}_{i_1\dots i_k} 
\frac{\partial}{\partial u^{(n_1)}_{i_1}}\wedge \cdots 
\wedge \frac{\partial}{\partial u^{(n_k)}_{i_k}}
\,\,,\,\,\,\,
P^{n_1\dots n_k}_{i_1\dots i_k}\in\mc V\,.
\end{equation}

The associative product in $\tilde\Omega_\bullet(\mc V)$ is just the wedge product,
and the bracket \eqref{0225:eq6} becomes in this case:
\begin{equation}\label{0325:eq1}
\begin{array}{c}
\displaystyle{
\big[
P\frac{\partial}{\partial u^{(m_1)}_{i_1}}\wedge \cdots \wedge \frac{\partial}{\partial u^{(m_h)}_{i_h}},
Q\frac{\partial}{\partial u^{(m_{h+1})}_{i_{h+1}}}\wedge \cdots \wedge \frac{\partial}{\partial u^{(m_k)}_{i_k}}
\big]
} \\
\displaystyle{
=
\sum_{\alpha=1}^h(-1)^{h+\alpha}
P\frac{\partial Q}{\partial u_{i_\alpha}^{(m_\alpha)}}
\frac{\partial}{\partial u^{(m_1)}_{i_1}}\wedge \stackrel{\alpha}{\check{\cdots}} 
\wedge \frac{\partial}{\partial u^{(m_k)}_{i_k}}
} \\
\displaystyle{
\,\,\,\,\,\,\,\,\,
+\sum_{\alpha=h+1}^k(-1)^{h+\alpha}
Q\frac{\partial P}{\partial u_{i_\alpha}^{(m_\alpha)}}
\frac{\partial}{\partial u^{(m_1)}_{i_1}}\wedge \stackrel{\alpha}{\check{\cdots}} 
\wedge \frac{\partial}{\partial u^{(m_k)}_{i_k}}\,.
}
\end{array}
\end{equation}

Note that there is a natural action of $\partial$ as a derivation of the Lie algebra $\Vect(\mc V)$,
given by $X\mapsto[\partial,X]$.
This action extends to a derivation of the Gerstenhaber algebra $\tilde\Omega_\bullet(\mc V)$.
Explicitly, if $X$ is as in \eqref{0323:eq2}, we have, using \eqref{derivation},
\begin{equation}\label{0324:eq2}
\partial(X)=\sum_{\substack{i_1,\cdots,i_k\in I \\ n_1,\cdots,n_k\in\mb Z_+}} 
\Big(\partial P^{n_1\dots n_k}_{i_1\dots i_k}
-\sum_{j=1}^k P^{n_1\dots n_{n_j+1}\cdots n_k}_{i_1\dots i_k}\Big)
\frac{\partial}{\partial u^{(n_1)}_{i_1}}\wedge \cdots 
\wedge \frac{\partial}{\partial u^{(n_k)}_{i_k}}\,.
\end{equation}
A polyvector field $X\in\tilde\Omega_\bullet(\mc V)$ is called \emph{evolutionary}
if $\partial(X)=0$.
Hence, $X$ in \eqref{0323:eq2} is an evolutionary polyvector field if and only if
\begin{equation}\label{0324:eq3}
\partial P^{n_1\dots n_k}_{i_1\dots i_k}
=\sum_{j=1}^k P^{n_1\dots n_{n_j+1}\cdots n_k}_{i_1\dots i_k}
\end{equation}

The \emph{de Rham complex} over $\mathcal V$ is the 
free unital commutative associative superalgebra over $\mc V$
with odd generators $du_i^{(n)},\,i\in I,n\in\mb Z_+$.
It consists of elements of the form
\begin{equation}\label{0324:eq4}
\tilde \omega = 
\frac1{k!}\sum_{\substack{i_1,\cdots,i_k\in I \\ m_1,\cdots,m_k\in\mb Z_+}}
f^{m_1\dots m_k}_{i_1 \dots i_k}
du^{(m_1)}_{i_1}\wedge \dots \wedge du^{(m_k)}_{i_k}
\,\,, \,\,\,\,
f^{m_1\dots m_k}_{i_1 \dots i_k} \in \mathcal V \,,
\end{equation}
where all but finitely many coefficients $f^{m_1\dots m_k}_{i_1 \dots i_k}$ are zero.
It is a $\mb Z_+$-graded complex, with the differential $d$ given by the usual formula:
\begin{equation}\label{0324:eq1}
d\tilde \omega = 
\frac1{k!} \sum_{j\in I, n\in \mb Z_+}
\sum_{\substack{i_1,\cdots,i_k\in I \\ m_1,\cdots,m_k\in\mb Z_+}} 
\frac{\partial f^{m_1\dots m_k}_{i_1 \dots i_k}}{\partial u^{(n)}_{i_j}}
d u^{(n)}_{j}
\wedge d u^{(m_1)}_{i_1}\wedge \dots \wedge d u^{(m_k)}_{i_k} \,.
\end{equation}
Clearly, $d$ is an odd derivation of degree $1$ and one checks easily that $d^2=0$.

Given a vector field $X\in\Vect(\mc V)$, we define the \emph{contraction operator}
$\iota_X$ as the odd derivation of the superalgebra 
$\tilde\Omega^\bullet(\mc V)$ acting trivially on $\mc V$
and such that $\iota_X(du_i^{(n)})=X(u_i^{(n)})$.
Furthermore, for $f\in\mc V$, we let $\iota_f$ be the operator 
of left multiplication by $f$ on $\tilde\Omega^\bullet(\mc V)$.
Recalling Definition \ref{0220:def1}, one easily checks that the resulting map 
$\iota_\cdot:\,\Pi\Vect(\mc V)\oplus\mc V\to\End(\tilde\Omega^\bullet(\mc V))$
defines a structure of a $(\Vect(\mc V),\mc V)$-complex on $(\tilde\Omega^\bullet(\mc V),d)$.
Hence, by Theorem \ref{0228:th}, this extends to a calculus structure
$(\tilde\Omega_\bullet(\mc V),\tilde\Omega^\bullet(\mc V))$.

The action of $\partial$ on $\mc V$ extends to an action on the de Rham complex
$\tilde\Omega^\bullet(\mc V)$
as an even derivation of the associative product such that $\partial(du_i^{(n)})=d(u_i^{(n+1)})$.
It is immediate to check that $\partial$ commutes with the action of $d$ in \eqref{0324:eq1}.
Hence, we can consider the reduced calculus structure 
$(\tilde\Omega_\bullet(\mc V)^\partial,\tilde\Omega^\bullet(\mc V)/\partial\tilde\Omega^\bullet(\mc V))$
(see Example \ref{0321:ex}),
which we call the \emph{variational calculus structure},
and denote by $(\Omega_\bullet(\mc V),\Omega^\bullet(\mc V))$.
It is easy to check that 
\begin{equation}\label{0325:eq6}
[\partial,\iota_X]=\iota_{\partial(X)}
\,\,,\,\,\,\,
X\in\tilde\Omega_\bullet(\mc V)\,,
\end{equation}
hence the Gerstenhaber algebra $\Omega_\bullet(\mc V)$
is the algebra of evolutionary polyvector fields.
The complex $(\Omega^\bullet(\mc V),d)$ is called the \emph{variational complex} \cite{GD}.

\subsection{Variational complex as a Lie conformal algebra complex}
\label{sec:5.2}

The connection between Lie conformal algebra calculus structure 
and the variational calculus is based on the following observation, \cite{DSK}.
Let $\mathcal V$ be an algebra of differential functions and consider the Lie conformal algebra 
$R = \oplus_{i\in I} \mb F[\partial]u_i$, with the zero $\lambda$-bracket. 
Then $\mc V$ is endowed with a structure of an $R$-module,
with the following $\lambda$-action:
\begin{equation}\label{0325:eq2}
u_{i\lambda}f
=
\sum_{n\in \mb Z_+} \lambda^n \frac{\partial f}{\partial u_i^{(n)}}
\,\,,\,\,\,\,
i\in I\,,
\end{equation}
and $R$ acts by derivations on the associative product in $\mathcal V$.
Recalling the construction in Section \ref{sec:4.4},
we consider the calculus structure $(C_\bullet(R,\mc V),C^\bullet(R,\mc V))$
for the $R$-module $\mc V$.

In this section we will identify it with the variational calculus structure,
and in the next section we will describe it more explicitly.

We define a map $\Phi_\bullet:\,C_\bullet(R,\mc V)\to\tilde\Omega_\bullet(\mc V)$
as follows.
Since $R=\bigoplus_i\mb F[\partial]u_i$ is a free $\mb F[\partial]$-module,
the space $C_k(R,M)$ is spanned by elements of the form
$a=u_{i_1}\otimes\cdots\otimes u_{i_k}\otimes\phi(x_1,\cdots,x_k)$,
with $i_1,\dots,i_k\in I$ and $\phi\in\mc M_k$.
Expanding the formal power series $\phi$ as
$$
\phi(x_1,\cdots,x_k)
=\sum_{n_1,\cdots,n_k\in\mb Z_+}
\frac{P^{n_1\dots n_k}_{i_1 \dots i_k}}{n_1!\cdots n_k!}x_1^{n_1}\cdots x_k^{n_k}
\,\,,\,\,\,\,
P^{n_1\dots n_k}_{i_1 \dots i_k}\in\mc V\,,
$$
we let 
\begin{equation}\label{0324:eq7}
\Phi_k(a)
=
\sum_{\substack{i_1,\cdots,i_k\in I \\ n_1,\cdots,n_k\in\mb Z_+}} 
P^{n_1\dots n_k}_{i_1\dots i_k} 
\frac{\partial}{\partial u^{(n_1)}_{i_1}}\wedge \cdots 
\wedge \frac{\partial}{\partial u^{(n_k)}_{i_k}}\,.
\end{equation}
Clearly, $\Phi_k$ is well defined and injective.
Moreover, the condition that $\phi\in\mc M_k$
exactly corresponds, in terms of the coefficients $P^{n_1\dots n_k}_{i_1 \dots i_k}$,
to identity \eqref{0324:eq3}.
Hence, the image of $\Phi_k$ is the space of evolutionary polyvector fields $\Omega_k(\mc V)$,
and $\Phi_\bullet$ induces an isomorphism of $\mb Z_+$-graded vector spaces
$\Phi_\bullet:\,C_\bullet(R,\mc V)\to\Omega_\bullet(\mc V)$.

Next, we define a map $\Psi^\bullet:\,\tilde\Omega^\bullet(\mc V)\to C^\bullet(R,\mc V)$.
Let $\tilde\omega\in\tilde\Omega^k(\mc V)$ be as in \eqref{0324:eq4}.
Given indices $i_1,\dots,i_k\in I$, we let
\begin{equation}\label{0324:eq5}
\begin{array}{l}
\big(\Psi^k(\tilde\omega)\big)_{\lambda_1,\cdots,\lambda_k}(u_{i_1},\cdots,u_{i_k}) \\
\displaystyle{
=
\sum_{m_1,\cdots,m_k\in\mb Z_+}
\lambda_1^{m_1}\cdots \lambda_k^{m_k} \langle f\rangle^{m_1\dots m_k}_{i_1 \dots i_k}
\in\mb F[\lambda_1,\dots,\lambda_k]\otimes\mc V\,,
}
\end{array}
\end{equation}
where $\langle f\rangle$ is the skewsymmetrization of $f$, i.e.
\begin{equation}\label{0325:eq4}
\langle f\rangle^{m_1\dots m_k}_{i_1 \dots i_k}
=\frac1{k!}
\sum_{\sigma\in S_k} f^{m_{\sigma(1)}\dots m_{\sigma(k)}}_{i_{\sigma(1)}\dots i_{\sigma(k)}}\,.
\end{equation}

Note that the RHS is a polynomial in the variables $\lambda_1,\dots,\lambda_k$
since, by assumption, all but finitely many coefficients $f^{m_1\dots m_k}_{i_1 \dots i_k}$ are zero.
Then $\Psi^k(\tilde\omega)\in C^k(R,\mc V)$ is defined by extending the above formula to
$R^{\otimes k}\to\mb F[\lambda_1,\dots,\lambda_k]\otimes \mc V$
by the sesquilinearity relations,
and composing it with the quotient map 
$\mb F[\lambda_1,\dots,\lambda_k]\otimes\mc V
\to\mb F_-[\lambda_1,\dots,\lambda_k]\otimes_{\mb F[\partial]}\mc V$.
Clearly, $\Psi^k(\tilde\omega)$ satisfies the skewsymmetry conditions in $C^k(R,\mc V)$,
thanks to the assumption that the coefficients $f^{m_1\dots m_k}_{i_1 \dots i_k}$ are skewsymmetric.
Hence $\Psi^k(\tilde\omega)$ lies in $C^k(R,\mc V)$.
Moreover, the map $\Psi^k$ is obviously surjective.
To study the kernel of the map $\Psi^k$, we need the following identity,
which can be easily checked directly:
\begin{equation}\label{0324:eq6}
\begin{array}{l}
\big(\Psi^k(\partial\tilde\omega)\big)_{\lambda_1,\cdots,\lambda_k}(u_{i_1},\cdots,u_{i_k}) \\
= (\partial+\lambda_1+\cdots+\lambda_k)
\big(\Psi^k(\tilde\omega)\big)_{\lambda_1,\cdots,\lambda_k}(u_{i_1},\cdots,u_{i_k})
\in\mb F[\lambda_1,\dots,\lambda_k]\otimes_{\mb F}\mc V\,.
\end{array}
\end{equation}
Recalling that $\mb F_-[\lambda_1,\dots,\lambda_k]\otimes_{\mb F[\partial]}\mc V$
is the quotient of $\mb F[\lambda_1,\dots,\lambda_k]\otimes_{\mb F}\mc V$
by the image of $(\partial+\lambda_1+\cdots+\lambda_k)$,
we deduce that $\ker(\Psi^k)=\im(\partial)$.
Thus $\Psi^\bullet$ factors through a bijective map of $\mb Z_+$-graded vector spaces
$\Psi^\bullet:\,\Omega^\bullet(\mc V)\to C^\bullet(R,\mc V)$.
\begin{theorem}\label{0324:th}
Let $\mc V$ be an algebra of differential functions in the variables $u_i,\,i\in I$.
Then the maps $\Phi_\bullet$ and $\Psi^\bullet$ defined by \eqref{0324:eq7} and \eqref{0324:eq5}
give a morphism of calculus structures
\begin{equation}\label{0324:eq9}
(\tilde\Omega_\bullet(\mc V),\tilde\Omega^\bullet(\mc V))\to(C_\bullet(R,\mc V),C^\bullet(R,\mc V))\,,
\end{equation}
which induces a calculus structure isomorphism 
$$
(\Omega_\bullet(\mc V),\Omega^\bullet(\mc V))
\simeq
(\bar C_\bullet(R,\mc V),\bar C^\bullet(R,\mc V))\,.
$$
\end{theorem}
\begin{proof}

We want to prove that $\Phi_\bullet$ is a homomorphism of Gerstenhaber algebras. Let 
$a=u_{i_1}\otimes\cdots\otimes u_{i_h}\otimes\phi(x_1,\cdots,x_h) \in C_h(R,\mc V)$, 
$b=u_{i_{h+1}}\otimes\cdots\otimes u_{i_k}\otimes\psi(x_1,\cdots,x_{k-h})\in C_{k-h}(R,\mc V)$ 
and let $P^{m_1\dots m_h}_{i_1 \dots i_h}$ and $Q^{n_1\dots n_{k-h}}_{i_1 \dots i_{k-h}}$ be the 
coefficients of $x_1^{m_1}\cdots x_h^{m_h}$ respectively $x_1^{n_1}\cdots x_{k-h}^{n_{k-h}}$ in the formal 
power series expansion of $\phi$ and $\psi$. By \eqref{0312:eq1}, the coefficient of 
$x_1^{m_1}\cdots x_k^{m_k}$ in $a\wedge b$ is
$$ 
\sum_{\sigma\in S_{k}}\frac{\text{sign}(\sigma)}{h!(k-h)!}P^{m_{\sigma(1)}\dots m_{\sigma(h)}}_{i_{\sigma(1)} \dots i_{\sigma(h)}}
Q^{m_{\sigma(h+1)}\dots m_{\sigma(k)}}_{i_{\sigma(h+1)} \dots i_{\sigma(k)}} \,,
$$
which, together with \eqref{0324:eq7}, proves that $\Phi_\bullet$ is an associative superalgebra
homomorphism. 

We are left to prove that the Gerstenhaber bracket is preserved by this map. Since the $\lambda$-bracket on $R$ 
is zero by assumption, the expression \eqref{0312:eq12} for the bracket $[a,b]$ reduces to
\begin{equation}\label{0418:eq1}
\begin{array}{l}
\displaystyle{
\sum_{\alpha=1}^h (-1)^{h+\alpha} u_{i_1}\otimes\stackrel{\alpha}{\check{\cdots}}\otimes u_{i_k}
\otimes \big( {u_{i_\alpha}}\,_{\partial_y}\psi(x_h,\dots,x_{k-1}) \big) 
\phi(x_1,\dots\stackrel{\alpha}{y}\dots,x_{h-1}) +
} \\
\displaystyle{
\sum_{\alpha=h+1}^k (-1)^{h+\alpha} u_{i_1}\otimes\stackrel{\alpha}{\check{\cdots}}\otimes u_{i_k}
\otimes \big( {u_{i_\alpha}}\,_{\partial_y}\phi(x_1,\dots,x_h) \big) 
\psi(x_{h+1},\dots\stackrel{\alpha}{y}\dots,x_{k-1}) \,,
}
\end{array}
\end{equation}
evaluated at $y=0$. Using formula \eqref{0325:eq2} for the $R$-module structure on $\mc V$ and expanding in formal power series, 
the first sum  in \eqref{0418:eq1} at $y=0$ is equal to
$$
\sum_{\alpha=1}^h (-1)^{h+\alpha} u_{i_1}\otimes\stackrel{\alpha}{\check{\cdots}}\otimes u_{i_k}\otimes \sum_{m_r\in \mb Z_+}
P^{m_1\dots m_h}_{i_1 \dots i_h}\frac{\partial Q^{m_{h+1}\dots m_{k}}_{i_{h+1} \dots i_k}}{\partial u^{(m_{\alpha})}_{i_\alpha}}$$
$$
x_1^{m_1}\cdots x_{\alpha-1}^{m_{\alpha-1}} x_\alpha^{m_{\alpha+1}}\cdots x_{k-1}^{m_k}\,,
$$
where we have used $\left(\frac{\partial}{\partial y}\right)^ny^{m_\alpha}|_{y=0}=m_\alpha!\delta_{n,m_\alpha}$. 
Similarily, the second sum in \eqref{0418:eq1} at $y=0$ is equal to
$$
\sum_{\alpha=h+1}^k (-1)^{h+\alpha} u_{i_1}\otimes\stackrel{\alpha}{\check{\cdots}}\otimes u_{i_k}\otimes\sum_{m_r\in \mb Z_+}
Q^{m_{h+1}\dots m_{k}}_{i_{h+1} \dots i_k}\frac{\partial P^{m_1\dots m_h}_{i_1 \dots i_h}}{\partial u^{(m_{\alpha})}_{i_\alpha}}
$$$$x_1^{m_1}\cdots x_{\alpha-1}^{m_{\alpha-1}}x_\alpha^{m_{\alpha+1}}\cdots x_{k-1}^{m_k}\,.
$$
The identity $\Phi_{k-1}([a,b])=[\Phi_h(a),\Phi_{k-h}(b)]$ follows by combining the above results with the formula \eqref{0325:eq1} and the definition 
\eqref{0324:eq7} of $\Phi_\bullet$.

Next, we prove that $\Psi^\bullet$ is a morphism of complexes. Let $\tilde \omega\in \tilde \Omega^k(\mc V)$ 
as in \eqref{0324:eq4}. Again, due to the triviality of the $\lambda$-bracket on $R$, 
the second term in \eqref{100311:eq1} vanishes. Recalling the $\lambda$-action \eqref{0325:eq2} of $R$ on $\mc V$ 
and \eqref{0324:eq5}, the coefficient of $\lambda_1^{m_1}\cdots\lambda_{k+1}^{m_{k+1}}$ in the polynomial 
$(d\Psi^k(\tilde\omega))_{\lambda_1,\cdots,\lambda_{k+1}}(u_{i_1},\cdots,u_{i_{k+1}})$ is
$$
\sum_{\alpha=1}^{k+1}(-1)^{\alpha+1}\frac{\partial \langle f\rangle^{m_1\stackrel{\alpha}{\check{\cdots}}m_{k+1}}
_{i_1\stackrel{\alpha}{\check{\cdots}}i_{k+1}}}{\partial u_{i_\alpha}^{(m_\alpha)}}\,.
$$
By \eqref{0324:eq1}, it follows that
$$
\big(\Psi^{k+1}(d\tilde\omega)\big)_{\lambda_1,\cdots,\lambda_{k+1}}(u_{i_1},\cdots,u_{i_{k+1}}) = 
\sum_{m_r\in\mb Z_+}
\lambda_1^{m_1}\cdots \lambda_{k+1}^{m_{k+1}} \left\langle \frac{\partial f^{m_2\dots m_{k+1}}_
{i_2 \dots i_{k+1}}}{\partial u^{(m_1)}_{i_1}}\right\rangle
$$
$$
= \sum_{m_r\in\mb Z_+}\lambda_1^{m_1}\cdots \lambda_{k+1}^{m_{k+1}}
\sum_{\alpha=1}^{k+1}(-1)^{\alpha+1}\frac{\partial \langle f\rangle^{m_1\stackrel{\alpha}{\check{\cdots}}m_{k+1}}
_{i_1\stackrel{\alpha}{\check{\cdots}}i_{k+1}}}{\partial u_{i_\alpha}^{(m_\alpha)}} \,,
$$
thus proving that $\Psi^{k+1}(d\tilde\omega)=d\Psi^k(\tilde\omega)$.

Finally, we show that $\Phi_\bullet$ and $\Psi^\bullet$ are compatible with the contraction operators. 
Let $a=u_{i_1}\otimes \cdots\otimes u_{i_h}\otimes \phi \in C_h(R,\mc V)$ and let $\tilde \omega \in \tilde \Omega^k(\mc V)$ as in above.
We want to prove that $\Psi^{k-h}(\iota_{\Phi_h(a)}(\tilde\omega))=\iota_a(\Psi^k(\tilde\omega))$. By \eqref{0304:eq2}, it follows
that the coefficient of $\lambda_{h+1}^{m_{h+1}}\cdots\lambda_k^{m_k}$ in $\Psi^{k-h}(\iota_{\Phi_h(a)}(\tilde\omega))
_{\lambda_{h+1},\cdots,\lambda_k}(u_{i_{h+1}},\dots,u_{i_k})$ is
$$
(-1)^{\frac{h(h-1)}{2}}\sum_{m_1,\dots, m_h\in\mb Z_+}P^{m_1\cdots m_h}_{i_1\cdots i_h}
\langle f\rangle^{m_1\cdots m_k}_{i_1\cdots i_k}\,.
$$
On the other hand, recalling \eqref{0319:eq1}, we have that
$$
\iota_a(\Psi^k(\tilde\omega))_{\lambda_{h+1},\cdots,\lambda_k}(u_{i_{h+1}},\dots,u_{i_k})=
$$
$$
=(-1)^{\frac{h(h-1)}{2}}\sum_{m_r\in\mb Z_+}\sum_{\substack{i_1,\cdots,i_h\in I \\ n_1,\cdots,n_h\in\mb Z_+}}P^{n_1\cdots n_h}_{i_1\cdots i_h}\frac{\partial_{\lambda_1}^{n_1}}{n_1!}\cdots\frac{\partial_{\lambda_h}^{n_h}}{n_h!}\langle f\rangle^{m_1\cdots m_k}_{i_1\cdots i_k} \lambda_1^{m_1}\cdots \lambda_k^{m_k}
$$
$$
= (-1)^{\frac{h(h-1)}{2}}\sum_{\substack{i_1,\cdots,i_h\in I \\ m_1,\cdots,m_k\in\mb Z_+}}P^{m_1\cdots m_h}_{i_1\cdots i_h}
\langle f\rangle^{m_1\cdots m_k}_{i_1\cdots i_k}\lambda_{h+1}^{m_{h+1}}\cdots\lambda_k^{m_k}\,,
$$
which completes the proof of the theorem. 
\flushright\qed
\end{proof}

\begin{remark}\label{0324:rem3}
Formulas \eqref{0324:eq7} and \eqref{0324:eq5}
define an isomorphism of calculus structures
$(\tilde\Omega_\bullet(\mc V),\tilde\Omega^\bullet(\mc V))
\simeq(\tilde C_\bullet(R,\mc V),\tilde C^\bullet(R,\mc V))$.
This isomorphism induces the morphism \eqref{0324:eq9}.
In other words, we have the following commutative diagram of calculus structures:
$$
\UseTips
\xymatrix{
(\tilde\Omega_\bullet(\mc V),\tilde\Omega^\bullet(\mc V)) \ar[d] \ar[r]^{\sim}
& (\tilde C_\bullet(R,\mc V),\tilde C^\bullet(R,\mc V)) \ar[d] \\
(\Omega_\bullet(\mc V),\Omega^\bullet(\mc V))  \ar[r]^{\sim}
& (C_\bullet(R,\mc V),C^\bullet(R,\mc V))\,.
}
$$
\end{remark}

\subsection{A description of the variational calculus structure}\label{sec:5.3}

Let $\mc V$ be an algebra of differential functions in the variables $u_i,\,i\in I$.
To every $k$-cochain $\tilde{\omega} \in \tilde{\Omega}$ we associate the linear map
$S_{\tilde{\omega}}:\,\tilde{\Omega}_k(\mc V) \to \mc V,\,
X \mapsto S_{\tilde{\omega}}(X)=(-1)^{k(k-1)/2}\iota_X (\tilde{\omega})$.
Explicitly, it is easy to see that for $\tilde\omega$ as in \eqref{0324:eq4}
and $X$ as in \eqref{0323:eq2}, we have
\begin{equation}\label{0325:eq5}
S_{\tilde\omega}(X)
=
\sum_{\substack{i_1,\cdots,i_k\in I \\ m_1,\cdots,m_k\in\mb Z_+}}
\langle f\rangle^{m_1\cdots m_k}_{i_1 \cdots i_k} P^{m_1\cdots m_k}_{i_1 \cdots i_k}\,,
\end{equation}
where $\langle f\rangle$ is the skewsymmetrization defined in \eqref{0325:eq4}.
\begin{lemma}\label{0325:lem}
For $\tilde\omega\in\tilde\Omega^k(\mc V)$ and $X\in\Omega_k(\mc V)\subset\tilde\Omega_k$,
we have $S_{\partial\tilde\omega}(X)\in\partial\mc V$.
\end{lemma}
\begin{proof}
We have, by definition,
$$
S_{\partial\tilde\omega}(X)
=(-1)^{k(k-1)/2}\iota_X(\partial\tilde\omega)
=(-1)^{k(k-1)/2}\big(\partial\iota_X(\tilde\omega)-\iota_{\partial(X)}(\tilde\omega)\big)\,.
$$
In the second identity we used equation \eqref{0325:eq6}.
To conclude we just notice that, by assumption, $\partial(X)=0$.
\flushright\qed
\end{proof}
By Lemma \ref{0325:lem}, for $\omega\in\Omega^k(\mc V)=\tilde\Omega^k(\mc V)/\partial\tilde\Omega^k(\mc V)$
we have the induced map: $\Omega_k(\mc V)\to\mc V/\partial\mc V$.
Recalling the isomorphism 
$\Psi^k:\,\Omega^k(\mc V)\stackrel{\sim}{\to} C^k(R,\mc V)$ 
defined in Theorem \ref{0324:th},
to every $c\in C^k(R,\mc V)$ we associate the induced map
$S_c:\,\Omega_k(\mc V)\to\mc V/\partial\mc V$.
Explicitly, if $c\in C^k(R,\mc V)$ is such that
\begin{equation}\label{0325:eq8}
c_{\lambda_1,\cdots,\lambda_k}(u_{i_1},\cdots,u_{i_k})
=
\sum_{m_1,\cdots,m_k\in\mb Z_+}
\lambda_1^{m_1}\cdots \lambda_k^{m_k} f^{m_1\cdots m_k}_{i_1 \cdots i_k}\,,
\end{equation}
and $X\in\Omega_k$ is as in \eqref{0323:eq2}, we have
\begin{equation}\label{0325:eq7}
S_c(X)
=
\sum_{\substack{i_1,\cdots,i_k\in I \\ m_1,\cdots,m_k\in\mb Z_+}}
\tint f^{m_1\cdots m_k}_{i_1 \cdots i_k} P^{m_1\cdots m_k}_{i_1 \cdots i_k}\,.
\end{equation}

In this section we assume that the algebra of differential functions $\mc V$ is \emph{non-degenerate},
in the sense that the pairing $\mc V\times\mc V\to\mc V/\partial\mc V$,
given by $(f,g)=\tint fg$, is non-degenerate.
By \cite[Lemma 10(c)]{DSK}, any differential algebra extension of the algebra of differential
polynomials in Example \ref{0323:ex1} is non-degenerate.

\begin{proposition}\label{0325:prop}
Suppose that $\mc V$ is a non-degenerate algebra of differential functions. Then
\begin{enumerate}[(i)]
\item if $c\in C^k(R,\mc V)$ is such that $S_c=0$, then $c=0$;
\item if $c\in C^k(R,\mc V)$ is such that 
$S_c(X_1\wedge\cdots\wedge X_k)=0$ for every $X_1,\dots,X_k\in\Omega_1(\mc V)$, then $c=0$.
\end{enumerate}
\end{proposition}
\begin{proof}
Obviously $(i)$ implies $(ii)$. 
Suppose then that $c$ in \eqref{0325:eq8} satisfies the assumption in $(ii)$.
We have, letting 
$X_\alpha=\sum_{i\in I, n\in \mb Z_+} (\partial^n P^\alpha_i) 
\frac{\partial}{\partial u^{(n)}_i}\in\Omega_1,\,\alpha=1,\dots,k$,
$$
\sum_{\substack{i_1,\cdots, i_k \in I \\ m_1,\cdots,m_k \in \mb Z_+}} 
f^{m_1\cdots m_k}_{i_1\cdots i_k}
(\partial^{m_1}P^1_{i_1})\cdots(\partial^{m_{k}}P^{k}_{i_{k}})
\big) = 0\,,
$$
for every of $P^1,\dots,P^k\in\mc V^\ell$.
Integrating by parts, and using the nondegeneracy of the pairing $\mc V\times\mc V\to\mc V/\partial\mc V$,\
we get
$$
\sum_{\substack{i_1,\cdots, i_k \in I \\ m_1,\cdots,m_k \in \mb Z_+}} (-\partial)^{m_k}
\big(
f^{m_1\cdots m_k}_{i_1\cdots i_k}
(\partial^{m_1}P^1_{i_1})\cdots(\partial^{m_{k-1}}P^{k-1}_{i_{k-1}})
\big) = 0
$$
for every $P^1,\cdots,P^{k-1}\in\mc V^\ell$.
Equivalently, we have that
$$
\sum_{\substack{i_1,\cdots, i_k \in I \\ m_1,\cdots,m_k \in \mb Z_+}} 
(-\partial-\lambda_1-\cdots-\lambda_{k-1})^{m_k}
f^{m_1\cdots m_k}_{i_1\cdots i_k}
\lambda_1^{m_1}\cdots\lambda_{k-1}^{m_{k-1}} = 0\,,
$$
as an element of $\mc V[\lambda_1,\dots,\lambda_{k-1}]$.
In other words,
the corresponding $k$-$\lambda$-bracket 
$\{{u_{i_1}}_{\lambda_1}\cdots {u_{i_{k-1}}}_{\lambda_{k-1}}u_{i_k}\}_c$
(see Remark \ref{100311:rem}) is zero.
\flushright\qed
\end{proof}

Thanks to Proposition \ref{0325:prop} 
we can and we will identify the space of $k$-cochains $C^k(R,\mc V)$
with the space of \emph{skewsymmetric local $k$-operators},
namely the maps $S:\,\Omega_k\to\mc V/\partial\mc V$ of the form
\begin{equation}\label{skewdiffop}
S(X) 
= 
\sum_{\substack{i_1,\cdots, i_k \in I \\ m_1,\cdots,m_k \in \mb Z_+}} 
\tint f^{m_1\cdots m_k}_{i_1\cdots i_k} P^{m_1\cdots m_k}_{i_1\cdots i_k}\,,
\end{equation}
where $f^{m_1\cdots m_k}_{i_1\cdots i_k}\in\mc V$ 
are skewsymmetric under simultaneous permutations of upper and lower indices,
and all but finitely many of them are zero,
and $X=\sum  P^{n_1\cdots n_k}_{j_1\cdots j_k}\frac{\partial}{\partial u_{j_1}^{(n_1)}}\wedge 
\cdots \wedge \frac{\partial}{\partial u_{j_k}^{(n_k)}} \in \Omega_k(\mathcal V)$ 
is an evolutionary $k$-vector field.

Next, we will see how the calculus structure on $C^\bullet(R,\mc V)$
translates under this identification.

It is not hard to check, by direct computation, that
the definition \eqref{100311:eq1} of the differential $d:\, C^k(R,\mc V)\to C^{k+1}(R,\mc V)$
gives rise to the following map on skewsymmetric local operators.
Let $S:\,\Omega_k(\mc V)\to\mc V/\partial\mc V$ be as in \eqref{skewdiffop} and 
consider the evolutionary $k+1$-vector field 
$X=\sum  P^{n_1\cdots n_{k+1}}_{j_1\cdots j_{k+1}}\frac{\partial}{\partial u_{j_1}^{(n_1)}}\wedge 
\cdots \wedge \frac{\partial}{\partial u_{j_{k+1}}^{(n_{k+1})}} \in \Omega_{k+1}(\mathcal V)$.
We have
\begin{equation}\label{0325:eq9}
(dS)(X)
=
\sum_{\substack{i_1,\cdots, i_{k+1} \in I \\ m_1,\cdots,m_{k+1} \in \mb Z_+}} 
\sum_{\alpha=1}^{k+1}(-1)^{\alpha+1} \int
\Big(
\frac{\partial}{\partial u_{i_\alpha}^{(m_\alpha)}}
f^{m_1\stackrel{\alpha}{\check{\cdots}} m_{k+1}}_{i_1\stackrel{\alpha}{\check{\cdots}} i_{k+1}}
\Big)
P^{m_1\cdots m_{k+1}}_{i_1\cdots i_{k+1}}\,.
\end{equation}
In particular, if $X=X_1\wedge\cdots\wedge X_{k+1}$, with $X_i\in\Omega_1(\mc V)$,
we recover formula (166) from \cite{DSK}:
$$
(dS)(X)
=
\sum_{\alpha=1}^{k+1}(-1)^{\alpha+1} \int
\big( X_\alpha S \big)(X_1\wedge\stackrel{\alpha}{\check{\cdots}}\wedge X_{k+1})\,,
$$
where $X_\alpha S$ means that $X_\alpha$ acts on the coefficients of $S$.

Next, we see how the contraction operators $\iota_X,\,X\in\Omega_h(\mc V)$, act
on skewsymmetric local $k$-operators.
The action of $\iota_X$ on a skewsymmetric local $k$-operator 
$S:\,\Omega_k(\mc V)\to\mc V/\partial\mc V$
is induced by its action on $\tilde\Omega^k(\mc V)$
via the map $\tilde\omega\mapsto S_{\tilde\omega}$ defined 
at the beginning of the section.
It follows that, for $X\in\Omega_h(\mc V)$ and $Y\in\Omega_{k-h}(\mc V)$,
we must have $(\iota_X S)(Y)=(-1)^{h(h-1)/2}S(X\wedge Y)$.
Here we used equation \eqref{0319:eq3}.
Explicitly,
if $S:\,\Omega_k(\mc V)\to\mc V/\partial\mc V$ is as in \eqref{skewdiffop},
$X=\sum  P^{n_1\cdots n_h}_{j_1\cdots j_h}\frac{\partial}{\partial u_{j_1}^{(n_1)}}\wedge 
\cdots \wedge \frac{\partial}{\partial u_{j_h}^{(n_h)}}$,
and
$Y=\sum  Q^{n_{h+1}\cdots n_k}_{j_{h+1}\cdots j_k}\frac{\partial}{\partial u_{j_{h+1}}^{(n_{h+1})}}\wedge 
\cdots \wedge \frac{\partial}{\partial u_{j_k}^{(n_k)}}$,
we have
\begin{equation}\label{0325:eq10}
(\iota_X S)(Y)
= 
(-1)^{\frac{h(h-1)}2}
\!\!\!\!\!\!\!\!\!
\sum_{\substack{i_1,\cdots, i_k \in I \\ m_1,\cdots,m_k \in \mb Z_+}} 
\tint f^{m_1\cdots m_k}_{i_1\cdots i_k} 
P^{m_1\cdots m_h}_{i_1\cdots i_h} Q^{m_{h+1}\cdots m_k}_{i_{h+1}\cdots i_k}\,.
\end{equation}

Combining formulas \eqref{0325:eq9} and \eqref{0325:eq10}
we get, by Cartan's formula, an explicit expression
for the Lie derivative $L_XS:\,\Omega_{k-h+1}\to\mc V/\partial\mc V$.
For $S$ and $X$ as before, and 
$Y=\sum  Q^{n_{h+1}\cdots n_{k+1}}_{j_{h+1}\cdots j_{k+1}}
\frac{\partial}{\partial u_{j_{h+1}}^{(n_{h+1})}}\wedge 
\cdots \wedge \frac{\partial}{\partial u_{j_{k+1}}^{(n_{k+1})}}\in\Omega_{k-h+1}(\mc V)$,
we have
\begin{equation}\label{0325:eq11}
\begin{array}{c}
(L_XS)(Y)
=
\displaystyle{
(-1)^{\frac{h(h-1)}2}
\!\!\!\!\!\!\!\!\!\!
\sum_{\substack{i_1,\dots, i_{k+1} \in I \\ m_1,\dots,m_{k+1} \in \mb Z_+}} 
\!\!\!\!
\int\Bigg(
\sum_{\alpha=1}^h (-1)^{\alpha+1}
\frac{
\partial f^{m_1\stackrel{\alpha}{\check{\cdots}} m_{k+1}}_{i_1\stackrel{\alpha}{\check{\cdots}} i_{k+1}}
}{\partial u_{i_\alpha}^{(m_\alpha)}}
P^{m_1\dots m_{h}}_{i_1\dots i_{h}} 
} \\
\displaystyle{
\,\,\,\,\,\,\,\,\,\,\,\,\,\,\,\,\,\,\,\,\,\,\,\,\,\,\,\,\,\,\,\,\,\,\,\,
+\sum_{\alpha=h+1}^{k+1} (-1)^{\alpha}
f^{m_1\stackrel{\alpha}{\check{\cdots}} m_{k+1}}_{i_1\stackrel{\alpha}{\check{\cdots}}i_{k+1}}
\frac{\partial P^{m_1\dots m_{h}}_{i_1\dots i_{h}}}{\partial u_{i_\alpha}^{(m_\alpha)}}
\Bigg)Q^{m_{h+1}\dots m_{k+1}}_{i_{h+1}\dots i_{k+1}}\,.
}
\end{array}
\end{equation}
In particular, if $X=X_1\wedge\cdots\wedge X_h,\,Y=X_{h+1}\wedge\cdots\wedge X_{k+1}$, 
with $X_i\in\Omega_1(\mc V)$,
equation \eqref{0325:eq11} becomes
$$
\begin{array}{c}
(L_XS)(Y)
=
\displaystyle{
(-1)^{\frac{h(h-1)}2}
\Big(
\sum_{\alpha=1}^h (-1)^{\alpha+1}
\big(X_\alpha S\big)(X_1\wedge\stackrel{\alpha}{\check{\cdots}}\wedge X_{k+1})
} \\
\displaystyle{
+\sum_{\alpha=1}^h\sum_{\beta=h+1}^{k+1} (-1)^{\beta}
S(X_1\wedge\cdots X_\beta(X_\alpha)\stackrel{\beta}{\check{\cdots}}\wedge X_{k+1})
\Big)\,,
}
\end{array}
$$
where, as before, $X S$ with $X\in\Omega_1(\mc V)$, means that $X$ acts on the coefficients of $S$,
and $Y(X)$, with $X,Y\in\Omega_1(\mc V)$, means that $Y$ acts on the coefficients of $X$.
(In terms of characteristics, $Y(X_P)=X_{Y(P)}$ \cite{DSK}).
In the special case $h=1$ we recover formula (175) in \cite{DSK}
(there is a typo there in the second term of the RHS):
$$
(L_XS)(X_1\wedge\cdots\wedge X_k)
=
\big(X S\big)(X_1\wedge\cdots\wedge X_k)
+\sum_{\beta=1}^{k}
S(X_1\wedge\cdots X_\beta(X)\cdots\wedge X_{k+1})
\,.
$$



\end{document}